\documentclass{article}
\usepackage{fontenc}
\usepackage{amsmath}
\usepackage{amssymb}
\usepackage{amsfonts}
\usepackage{amsthm}
\usepackage{enumerate}

\newtheorem{thm}{Theorem}[section]
\newtheorem{prop}[thm]{Proposition}
\newtheorem{cor}[thm]{Corollary}
\newtheorem{defn}[thm]{Definition}
\newtheorem{lem}[thm]{Lemma}
\newtheorem{conj}[thm]{Conjecture}
\newtheorem*{thmnn}{Theorem}

\begin{document}

\title{A Gross--Kohnen--Zagier Type Theorem for Higher-Codimensional Heegner
Cycles}

\author{Shaul Zemel\thanks{The initial stage of this research has been carried
out as part of my Ph.D. thesis work at the Hebrew University of Jerusalem,
Israel. The final stage of this work was carried out at TU Darmstadt and
supported by the Minerva Fellowship (Max-Planck-Gesellschaft).}}

\maketitle

\section*{Introduction}

A celebrated theorem of Gross, Kohnen, and Zagier in \cite{[GKZ]} states that
the Heegner divisors on modular curves correspond, in parts of the Jacobian
variety of the modular curve, to coefficients of a modular form of weight
$\frac{3}{2}$. This result was proved using height evaluations. Borcherds gave
in \cite{[B2]} another proof of this theorem, using his singular theta lift from \cite{[B1]}. The latter proof generalizes to Shimura curves, and the argument
extends to yield the modularity of Hirzebruch--Zagier divisors from \cite{[HZ]}
on Hilbert modular surfaces, of Humbert surfaces on Siegel threefolds, etc. The
goal of this paper is to establish a modularity result for cycles of higher
codimension inside the universal families arising from the moduli interpretation
of Shimura and modular curves.

Shimura curves parametrize Abelian surfaces with quaternionic multiplication.
This implies (under some technical assumptions) the existence of universal
families of Abelian surfaces over these curves. By taking the $m$th symmetric
power, we obtain a Kuga--Sato type variety, of dimension $2m+1$, in which the
fiber over a point in the Shimura curve is the product of $m$ copies of the
Abelian surface it represents. In correspondence with universal families of
elliptic curves, we denote this variety $W_{2m}$. The CM points on the Shimura
curve correspond to those Abelian surfaces whose endomorphism ring, or
equivalently whose group of cycles, is larger than the generic one. The $m$th
power of such an additional, normalized cycle has dimension $m$, and considering
it as a subvariety of the fiber of $W_{2m}$, it has codimension $m+1$ in the
latter variety. One defines the \emph{Heegner cycles} to be certain combinations
of these cycles in $W_{2m}$. As these cycles are cohomologically trivial, one
would like to investigate their images in the intermediate Jacobian of $W_{2m}$,
or in Hodge components of this Jacobian. There are several results indicating
the modularity of related objects in related groups (see, e.g., \cite{[Zh]} or
\cite{[H]}). This paper is proving another result of this type, namely
\begin{thmnn}
The images of the $(m+1)$-codimensional Heegner cycles in $W_{2m}$ correspond,
in an appropriate quotient group, to the coefficients of a modular form of
weight $\frac{3}{2}+m$.
\end{thmnn}

Before discussing the method of proof, we mention that the CM cycles are defined
in the modular case as actual cycles (normalized graphs of CM isogenies), but in
the general case only cohomology classes are considered (see, e.g., Section 5
of \cite{[Be]}) . A first aim of this paper is to describe the CM cycles also in
the non-split case as actual CM cycles inside the Abelian surface (see
Definition \ref{CMdef}). We then evaluate the fundamental cohomology classes of
these cycles explicitly, as well as simplify the proofs of several assertions
from \cite{[Be]}, using de-Rham cohomology. These CM cycles form, together with the generic cycles in case the quaternion algebra is trivial, all the Abelian subvarieties of Abelian surfaces with QM. However, the main result is proved for the corresponding cohomology classes, or Hodge vectors in variations of Hodge structures.

\medskip

Our method of proof follows \cite{[B2]} closely, but using the more general
theta lift from \cite{[B1]}, allowing the theta kernel of the even lattice $L$
of signature $(2,b_{-})$ to involve certain homogenous polynomials on
$L_{\mathbb{R}}$. The resulting functions are not meromorphic on the complex
manifold $G(L_{\mathbb{R}})$, but are eigenfunctions of eigenvalue $-2mb_{-}$
with respect to the weight $m$ Laplacian on $G(L_{\mathbb{R}})$, with known
singularities. In the case $b_{-}=1$ of Shimura and modular curves, the images
of these theta lifts under the weight raising operators are meromorphic (with
known poles). This establishes a new singular Shimura-type correspondence, as
stated in the following

\begin{thmnn}
Let $m$ be an even positive integer and let $f=\sum_{n}c_{n}q^{n}$ be a weakly
holomorphic modular form of weight $\frac{1}{2}-m$ with respect to $\Gamma_{0}(4)$ which lies in the Kohnen plus-space. Then the function
\[\frac{i^{m}}{2}\delta_{2m}\Phi_{L,m,m,0}(v,F)=\sum_{r=1}^{\infty}\bigg(\sum_{d|r}d^{m+1}c_{d^{2}}\bigg)r^{m}q^{r}\] is a meromorphic modular form of weight $2m+2$ with respect to $SL_{2}(\mathbb{Z})$. Its poles are all of order $m+1$, and they lie at CM points on the upper half-plane. More explicitly, $\frac{i^{m}}{2}\delta_{2m}\Phi_{L,m,m,0}(v,F)$ has a pole at a CM point of negative discriminant $-D$ if and only if the Fourier coefficient $c_{-D}$ of $f$ does not vanish.
\end{thmnn}

The method is developed here in a way aiming for proving similar results also
for higher-dimensional base spaces. Indeed, the subsequent work \cite{[Ze4]}
investigates the results thus obtained for the case $b_{-}=2$. It requires a
non-trivial analysis of the theta lift, which is based on the machinery from
\cite{[B1]}, but has to be examined explicitly since these particular functions
do not appear in that reference. The differential properties are established by
arguments similar to those appearing in \cite{[Bru]}. Additional useful
properties of our theta lifts are studied in \cite{[Ze2]}.

This paper is divided into 4 sections. In Section \ref{CMUF} we discuss
universal families over Shimura curves and the corresponding variations of Hodge
structures, as well as the CM cycles and their properties. In Section
\ref{RelDO} we derive a differential equation which is satisfied by Siegel theta
functions. Section \ref{Lift} presents the singular theta lift which we are
using, as well as its meromorphic image in the case $b_{-}=1$ (the singular
Shimura-type correspondence stated above). Finally, Section \ref{Rels} defines
the relations which we consider between the CM (or Heegner) cycles in question,
proves the main result of this paper, and suggests possible connections between
our results and existing theorems and conjectures.

\medskip

I am deeply indebted to my Ph.D. advisor R. Livn\'{e}, for proposing the
questions discussed here, for many enlightening discussions, and for being a
constant source of support during the work on this project. I would also like to
thank H. Farkas for his help during that time. Many thanks are due to J.
Bruinier for numerous comments and suggestions, which greatly improved this
paper relative to the original version. The help of S. Ehlen in the evaluation
and verification of several explicit examples of the singular Shimura-type lift
is also acknowledged. My special thanks are delivered to the anonymous referee,
whose suggestions greatly improved the readability of this paper.

\section{CM Cycles in Universal Families \label{CMUF}}

In this Section we present universal families of Abelian surfaces over Shimura
(and modular) curves, and the variations of Hodge structures arising from them.
We give the definition of general CM cycles inside such surfaces having also CM,
and show how to evaluate their fundamental cohomology classes.

\subsection{Variations of Hodge Structures over Shimura Curves}

Let $\mathcal{H}$ denote the upper half plane
$\{\tau=x+iy\in\mathbb{C}|\Im\tau>0\}$. The Shimura variation of Hodge structure
over $\mathcal{H}$, denoted $V_{1}$, is defined (see also Section 5 of
\cite{[Be]} and Section 3 of \cite{[Bry]}) as follows. The real local system is
just $\mathcal{H}\times\mathbb{R}^{2}$. Tensoring the fiber over
$\tau\in\mathcal{H}$ with $\mathbb{C}$, we fix $\binom{\tau}{1}$ to have Hodge
weight $(1,0)$ and $\binom{\overline{\tau}}{1}$ to have Hodge weight $(0,1)$
there. The determinant of $2\times2$ matrices defines a map
$\mathbb{R}^{2}\times\mathbb{R}^{2}\to\mathbb{R}$,
$(u,v)\mapsto\det\big((u,v)\big)$, which polarizes $V_{1}$. Let $V_{m}$ be be
the $m$th symmetric power of $V_{1}$, elements of whose fibers we write as
homogenous polynomials of degree $m$ in the fiber of $V_{1}$. Then the vector
$\binom{\tau}{1}^{m_{+}}\binom{\overline{\tau}}{1}^{m_{-}}$, with
$m_{+}+m_{-}=m$, has Hodge type $(m_{+},m_{-})$ in the complexification of the
fiber of $V_{m}$ over $\tau\in\mathcal{H}$.

Given $\tau\in\mathcal{H}$, we define \[M_{\tau}=\binom{\tau\ \ -\tau^{2}}{1 \ \
-\tau\ },\quad J_{\tau}=\frac{1}{y}\binom{x\ \ -|\tau|^{2}}{1\ \ \ -x\ \ },\quad
\mathrm{and}\quad M_{\overline{\tau}}=\binom{\overline{\tau}\ \
-\overline{\tau}^{2}}{1 \ \ -\overline{\tau}\ }.\] Recall that the action of
$A=\binom{a\ \ b}{c\ \ d} \in GL_{2}^{+}(\mathbb{R})$ sends $\tau\in\mathcal{H}$
to $A\tau=\frac{a\tau+b}{c\tau+d}$, with automorphy factor $j(A,\tau)=c\tau+d$.
We have the useful equation
\begin{equation}
A\binom{\tau}{1}=j(A,\tau)\binom{A\tau}{1}, \label{V1mod}
\end{equation}
as well as the relations
\begin{equation}
AM_{\tau}adjA=j(A,\tau)^{2}M_{A\tau}\quad\mathrm{and}\quad
AJ_{\tau}A^{-1}=J_{A\tau}. \label{Jtauconj}
\end{equation}
Here $adjA=\det A \cdot A^{-1}=\binom{\ \ d\ \ -b}{-c\ \ \ \ a}$ for $A$ as
above. More generally, Equation \eqref{V1mod} holds for every $A \in
M_{2}(\mathbb{R})$  and $\tau\in\mathbb{C}$ for which the natural extension of
$j(A,\tau)$ does not  vanish (so that $A\tau$ makes sense). In addition,
Equation \eqref{Jtauconj} extends to the case of $A \in GL_{2}(\mathbb{R})$ and
$\tau\in\mathcal{H}\cup\overline{\mathcal{H}}$ (where
$J_{\overline{\tau}}=-J_{\tau}$). We have
\begin{prop}
The following assertions are true:
\begin{enumerate}[$(i)$]
\item Multiplication of $\binom{\tau}{1}$ by any non-zero complex number can be
achieved by the action of $GL_{2}^{+}(\mathbb{R})$.
\item $V_{1}$, hence also $V_{m}$ for every $m$, is
$GL_{2}^{+}(\mathbb{R})$-equivariant. $SL_{2}(\mathbb{R})$ preserves the
polarization as well. Hence they descent to variations of Hodge structures on
quotients of $\mathcal{H}$ by Fuchsian groups.
\end{enumerate} \label{Vmequiv}
\end{prop}

\begin{proof}
Part $(i)$ follows by Equation \eqref{V1mod} from the fact that $M=dI+cJ_{\tau}$
satisfies $M\tau=\tau$ and $j(M,\tau)=ci+d$. Part $(ii)$ is an immediate
consequence of Equation \eqref{V1mod}. This proves the proposition.
\end{proof}

\medskip

Let $B$ be a \emph{rational quaternion algebra}, i.e., a 4-dimensional central
simple algebra over $\mathbb{Q}$, with reduced trace $Tr:B\to\mathbb{Q}$,
reduced norm $N:B^{\times}\to\mathbb{Q}^{\times}$, and main involution
$x\mapsto\overline{x}$. For background on (rational) quaternion algebras we
refer the reader to \cite{[Vi]}, for example. Assume that $B$ splits over
$\mathbb{R}$, and fix an isomorphism
$i:B_{\mathbb{R}}\stackrel{\sim}{\to}M_{2}(\mathbb{R})$. We identify
$B_{\mathbb{R}}$ with $M_{2}(\mathbb{R})$ via $i$, hence omit $i$ from the
notation. Recall that an \emph{ideal} in a rational quaternion algebra is a
finitely generated subgroup of full rank in $B$, and an \emph{order} is an
ideal which is a subring of $B$. An order is \emph{maximal} if it is not
contained in any larger order in $B$, and is \emph{Eichler} if it is the
intersection of two maximal orders. Given an order $\mathcal{M}$ in $B$, we
consider Abelian surfaces $A$ with quaternionic multiplication (QM) from
$\mathcal{M}$. To be precise, the latter statement means not only that
$\mathcal{M}$ is embedded into $A$, but also that the intersection of $End(A)$
with $B \subseteq End(A)_{\mathbb{Q}}$ is $\mathcal{M}$ (i.e., if
$\widetilde{\mathcal{M}}$ is a larger order in $B$ which is embedded into
$End(A)$ we say that $A$ has QM from $\widetilde{\mathcal{M}}$, and \emph{not}
from $\mathcal{M}$). As an example of such a surface, take
$\tau\in\mathcal{H}\cup\overline{\mathcal{H}}$ and an ideal $I \subseteq B$ such
that the order $L(I)=\{x \in B|xI \subseteq I\}$ is equal to $\mathcal{M}$, and
define $A_{\tau}=\mathbb{C}^{2}/I\binom{\tau}{1}$. $\mathcal{M}$ embeds into
$End(A_{\tau})$ by identifying $\mathbb{C}^{2}$ with
$M_{2}(\mathbb{R})\binom{\tau}{1}$ and letting $\mathcal{M}$ operate on $I$ and
$M_{2}(\mathbb{R})$. The \emph{Shimura curve} associated with $I$ is the
quotient of $\mathcal{H}$ by the group $R(I)^{\times}_{+}$ of (invertible) elements of reduced norm 1 in the order $R(I)=\{x \in B|Ix \subseteq I\}$. We now have
\begin{lem}
Let $\tau$ and $\sigma$ be elements of $\mathcal{H}\cup\overline{\mathcal{H}}$,
and let $I$ and $J$ be ideals in $B$ such that $L(I)=L(J)=\mathcal{M}$. Then the
maps $\mathbb{C}^{2}/J\binom{\sigma}{1}\to\mathbb{C}^{2}/I\binom{\tau}{1}$ which
commute with the action of $\mathcal{M}$ are in one-to-one correspondence with
elements $M \in B$ satisfying $M\tau=\sigma$ and $JM \subseteq I$. Every such
non-zero map is an isogeny, and it is an isomorphism if and only if $JM=I$.
\label{mapsAtau}
\end{lem}

\begin{proof}
Any map between these two Abelian surfaces gives a linear map from
$\mathbb{C}^{2}=M_{2}(\mathbb{R})\binom{\sigma}{1}$ to
$\mathbb{C}^{2}=M_{2}(\mathbb{R})\binom{\tau}{1}$ between the tangent spaces.
Commutation with the (left) action of $\mathcal{M}$ means that this map is
described by right multiplication on $M_{2}(\mathbb{R})$ by some matrix $M$.
Considering the $H_{1}$ groups shows that $M$ must satisfy $JM \subseteq I$. In
particular $M \in B$. On the complex vector spaces, the only maps which
commute with $\mathcal{M}$ are multiplication by scalars, which yields the
equality $M\binom{\tau}{1}=j\binom{\sigma}{1}$. If $j=0$ then $M=0$ and we are
done. Otherwise Equation \eqref{V1mod} yields $\sigma=M\tau$ (and
$j=j(M,\tau)$). This implies $M \in B^{\times}$ (for $M\neq0$), so that a
certain integral multiple of $M^{-1}$ yields a map in the other direction. As
one may use $M^{-1}$ itself for the map
$\mathbb{C}^{2}/I\binom{\tau}{1}\to\mathbb{C}^{2}/J\binom{\sigma}{1}$ if and
only if $JM=I$, this completes the proof of the lemma.
\end{proof}

The geometric interpretation of Shimura curves is given in

\begin{prop}
The moduli space of Abelian surfaces having QM from $\mathcal{M}$ is a disjoint
finite union of Shimura curves. If $\mathcal{M}$ is an Eichler order, this space is a single Shimura curve. \label{modspAVQM}
\end{prop}

\begin{proof}
As a real manifold, an Abelian surface $A$ with QM from $\mathcal{M}$ is
$V/\Lambda$ with $\Lambda=H_{1}(A,\mathbb{Z})$ and
$V=T_{0}(A)=\Lambda_{\mathbb{R}}$. Both $\Lambda$ and $V$ are
$\mathcal{M}$-modules. As $B$ has exactly one module of every
$\mathbb{Q}$-dimension divisible by 4, $\Lambda_{\mathbb{Q}}$ may be identified
with $B$ (hence $V$ with $M_{2}(\mathbb{R})$), and $\Lambda$ with some ideal $I
\subseteq B$ such that $L(I)=\mathcal{M}$. The complex structure commutes with
the action of $\mathcal{M}$, hence is defined through right multiplication by a
matrix in $M_{2}(\mathbb{R})$. This matrix must thus be of the form $J_{\tau}$
for some $\tau\in\mathcal{H}\cup\overline{\mathcal{H}}$. Identifying elements of
the holomorphic tangent space of $A$ with their first column multiplied by $iy$
shows that $A$ is isomorphic to $A_{\tau}$ defined above. Now, Lemma
\ref{mapsAtau} yields an isomorphism
$\mathbb{C}^{2}/IM^{-1}\binom{M\tau}{1}\to\mathbb{C}^{2}/I\binom{\tau}{1}$ for
any $M \in B^{\times}$. Applying this to $M$ with $N(M)<0$ shows that every
$A_{\tau}$ with $\tau\in\overline{\mathcal{H}}$ is isomorphic to some
$A_{\sigma}$ with $\sigma\in\mathcal{H}$ (and perhaps a different ideal). Hence
we may restrict attention to indices from $\mathcal{H}$ alone. The same
argument shows that we may also assume that the ideal $I$ belongs to a set of
representatives for the classes of ideals $I$ with $L(I)=\mathcal{M}$ modulo
right multiplication from $B^{\times}$. There are finitely many such classes, and if $\mathcal{M}$ is an Eichler order then there is only one such class (see
\cite{[Vi]}). In addition, for fixed $I$, Lemma \ref{mapsAtau} shows that if
$\tau$ and $\sigma$ are in $\mathcal{H}$ then $A_{\tau} \cong A_{\sigma}$ if and
only if $\sigma=M\tau$ for $M \in R(I)^{\times}$. Such $M$ must also have
positive reduced norm since $\tau$ and $M\tau$ are both in $\mathcal{H}$. This
proves the proposition.
\end{proof}

\smallskip

Let $\pi:\mathcal{A}\to\mathcal{H}$ be the universal Abelian surface with QM
from the order $\mathcal{M}$ (say with ideal $I$) as in, e.g., Section 4 of
\cite{[Be]}. This means that the fiber $\pi^{-1}(\tau)$ is $A_{\tau}$. Let
$\Gamma \subseteq R(I)^{\times}_{+}$ be of finite index, and let $X(\Gamma)$ be the associated Shimura curve. This is just $\Gamma\backslash\mathcal{H}$ if $B$ is not split. The map in which $M\in\Gamma$ takes $v+I\binom{\tau}{1}\big|_{\tau}$ (with $v\in\mathbb{C}^{2}$) to
$\frac{v}{j(M,\tau)}+I\binom{M\tau}{1}\big|_{M\tau}$ is a well-defined action of
$\Gamma$ on $\mathcal{A}$ over $\mathcal{H}$ (see the proof of Lemma
\ref{mapsAtau}), yielding a universal Abelian surface with QM over $X(\Gamma)$
(at least when $\Gamma$ has no elements with fixed points in $\mathcal{H}$), as
well as variations of Hodge structures on $X(\Gamma)$, which we describe below.
As such a universal family is the quotient of $\mathcal{H}\times\mathbb{C}^{2}$
by the semi-direct product of $\Gamma$ and $I$, this suggests the existence of a
theory of Jacobi-like forms on $\mathcal{H}\times\mathbb{C}^{2}$ in which
$\Gamma$ is the group acting on $\mathcal{H}$.

\smallskip

Let $I$ be an ideal in $B$, with discriminant $disc(I)$, and consider the
corresponding Abelian surface $A_{\tau}$ with $\tau\in\mathcal{H}$. As
$H_{1}(A_{\tau},\mathbb{Z})=I$, the group $H^{2}(A_{\tau},\mathbb{Q})$ is
identified with as the space of alternating rational-valued bilinear maps on
$I_{\mathbb{Q}}=B$. For an element $b$ in the space $B_{0}$ of traceless
elements of $B$, \cite{[Be]} defines two elements of $H^{2}(A,\mathbb{Q})$ by
\[\iota(b):(x,y) \mapsto Tr(bx\overline{y}),\qquad\tilde{\iota}(b):(x,y) \mapsto
Tr(b\overline{x}y)\] (with $x$ and $y$ in $I$, or in $B$). We quote Theorem 3.10
of \cite{[Be]}:
\begin{thm}
The following assertions hold:
\begin{enumerate}[$(i)$]
\item We have
$H^{2}(A_{\tau},\mathbb{Q})=\iota(B_{0})\oplus\tilde{\iota}(B_{0})$, an
orthogonal direct sum with respect to the cup product.
\item Given $b$ and $c$ in $B_{0}$, the pairing of $\iota(b)$ and $\iota(c)$ is
$disc(I)Tr(b\overline{c})$, while that of $\tilde{\iota}(b)$ and
$\tilde{\iota}(c)$ to $-disc(I)Tr(b\overline{c})$.
\item $\iota(B_{0}) \subseteq H^{1,1}(A_{\tau})$, while
$\tilde{\iota}(B_{0})_{\mathbb{C}}$ is isomorphic to the fiber of $V_{2}$ over
$\tau$.
\item For $a\in\mathcal{M} \subseteq End(A_{\tau})$ we have
$a^{*}\iota(b)=\iota(ab\overline{a})$ and
$a^{*}\tilde{\iota}(b)=N(a)\tilde{\iota}(b)$.
\item Let $\tau$, $\sigma$, $I$, $J$, and $M$ as in Lemma \ref{mapsAtau} be
given. Then the equalities $M^{*}\iota(b)=N(M)\iota(b)$ and
$M^{*}\tilde{\iota}(b)=\tilde{\iota}(Mb\overline{M})$ hold.
\end{enumerate} \label{H2Aprop}
\end{thm}
Part $(iii)$ of Theorem \ref{H2Aprop} also includes assertions from Theorem 5.8
and Proposition 5.12 of \cite{[Be]}. Part $(v)$ is a generalization of part (6)
of Theorem 3.10 of \cite{[Be]}, but the (straightforward) proof extends to this
case. We shall later consider these cohomology classes in the de-Rham setting,
which will also simplify the proofs of the first three parts of Theorem
\ref{H2Aprop}. Theorem \ref{H2Aprop} also has the following
\begin{cor}
Let $\tau$, $I$, and $A_{\tau}$ be as above.
\begin{enumerate}[$(i)$]
\item $\iota$ embeds $\Lambda=\big\{b \in
B_{0}\big|\iota(b)(I,I)\subseteq\mathbb{Z}\big\}$ into the N\'{e}ron-Severi
group $NS(A_{\tau})$ of $A_{\tau}$.
\item $\tilde{\Lambda}=\big\{b \in
B_{0}\big|\tilde{\iota}(b)(I,I)\subseteq\mathbb{Z}\big\}$ is an even lattice of
signature $(2,1)$.
\item Any element of the form $\frac{J_{\sigma}}{\sqrt{disc(I)}} \in B$, with
$\sigma\in\mathcal{H}$, yields via $\iota$ a principal polarization of
$A_{\tau}$.
\item All the possible polarizations are related by the action of
$\mathcal{M}$.
\end{enumerate}\label{H2lat}
\end{cor}

\begin{proof}
Part $(i)$ follows from parts $(iii)$ and $(v)$ of Theorem \ref{H2Aprop}
together with the Lefschetz Theorem on $(1,1)$ classes. Part $(ii)$ is a
consequence of part $(ii)$ of Theorem \ref{H2Aprop}. For part $(iii)$ we
apply parts $(ii)$ and $(iii)$ of Theorem \ref{H2Aprop} and the positivity
condition $-Tr(vxJ\overline{x})>0$ for $0 \neq x \in B$. Part $(iv)$ follows
from part $(iv)$ of that theorem. This proves the corollary.
\end{proof}

Part $(iii)$ of Corollary \ref{H2lat} justifies, a fortiori, our reference to
$A_{\tau}$ as an Abelian surface (rather than just a complex torus). On the
other hand, part $(iv)$ shows that the polarization cannot be characterized
further than coming from $\iota(\Lambda)$.

Consider now the universal family $\pi:\mathcal{A}\to\mathcal{H}$, as well as
the ones over $X(\Gamma)$ for $\Gamma$ without fixed points. Theorem
\ref{H2Aprop} then yields
\begin{cor}
The variation of Hodge structure $R^{2}\pi_{*}\mathbb{C}$ is the orthogonal
direct sum of two variations of Hodge structures. One is the constant variation
of Hodge structure $(B_{0})_{\mathbb{C}}[1]$. The other one is isomorphic to
$V_{2}$, but with the polarization multiplied by $disc(I)$. $\iota$ embeds the
lattice $\Lambda$ also into $NS(\mathcal{A})$. \label{R2pi*C}
\end{cor}

In the case where $B=M_{2}(\mathbb{Q})$ and $\mathcal{M}=I=M_{2}(\mathbb{Z})$,
Corollary \ref{R2pi*C} is already established in Section 3 of \cite{[Bry]}.
Indeed, we have $A_{\tau}=E_{\tau} \times E_{\tau}$ with
$E_{\tau}=\mathbb{C}/(\mathbb{Z}\tau\oplus\mathbb{Z})$, and
$\mathcal{A}$ is $\mathcal{E}\times_{\mathcal{H}}\mathcal{E}$ for the universal
elliptic curve $\pi^{\mathcal{E}}:\mathcal{E}\to\mathcal{H}$. The group
$H^{2}(E_{\tau} \times E_{\tau})$ decomposes as \[(H^{2}(E_{\tau}) \otimes
H^{0}(E_{\tau}))\oplus(H^{0}(E_{\tau}) \otimes
H^{2}(E_{\tau}))\oplus\bigwedge^{2}H^{1}(E_{\tau}) \oplus
Sym^{2}H^{1}(E_{\tau})\] (K\"{u}nneth and the action of $S_{2}$ on
$H^{1}(E_{\tau})^{\otimes2}$), where the lattice $\Lambda$ from part $(i)$ of
Corollary \ref{H2lat} is the direct sum of the first three terms. The lattice
$\tilde{\Lambda}$ from part $(ii)$ of that Corollary is the last term. Since
$V_{1}=R^{1}\pi^{\mathcal{E}}_{*}\mathbb{C}$, it follows that
\[R^{2}\pi_{*}\mathbb{C}=(R^{2}\pi^{\mathcal{E}}_{*}\mathbb{C} \otimes
R^{0}\pi^{\mathcal{E}}_{*}\mathbb{C})\oplus(R^{0}\pi^{\mathcal{E}}_{*}\mathbb{C}
\otimes R^{2}\pi^{\mathcal{E}}_{*}\mathbb{C})\oplus\bigwedge^{2}V_{1} \oplus
V_{2},\] with the first three variations of Hodge structures being constant of
type $(1,1)$. The cycles from Corollaries \ref{H2lat} and \ref{R2pi*C} which are
associated with the decompositions are the axes $\{0\} \times E_{\tau}$ and
$E_{\tau}\times\{0\}$ (or their unions $\bigcup_{\tau\in\mathcal{H}}\{0\} \times
E_{\tau}\big|_{\tau}$ and
$\bigcup_{\tau\in\mathcal{H}}E_{\tau}\times\{0\}\big|_{\tau}$ inside the
universal family $\mathcal{E}\times_{\mathcal{H}}\mathcal{E}$) and the
normalized diagonal $\widetilde{\Delta}_{E_{\tau}}=\Delta_{E_{\tau}}-\{0\}
\times E_{\tau}-E_{\tau}\times\{0\}$ (or
$\bigcup_{\tau\in\mathcal{H}}\widetilde{\Delta}_{E_{\tau}}\big|_{\tau}
\subseteq\mathcal{E}\times_{\mathcal{H}}\mathcal{E}$), where $\Delta_{E_{\tau}}$
is the imae of diagonal embedding of $E_{\tau}$ into $A_{\tau}$. More generally,
if $\mathcal{M}$ (and $I$) is the classical Eichler order of level $N$ in
$M_{2}(\mathbb{Q})$, so that $\Gamma=\Gamma_{0}(N)$, then the Abelian surface
$A_{\tau}$ is $E_{\tau} \times E_{N\tau}$. Then the first axis is $\{0\} \times
E_{N\tau}$, and $\widetilde{\Delta}_{E_{\tau}}$ is replaced by the normalized
graph of the classical isogeny $E_{\tau} \to E_{N\tau}$ descending from
multiplication by $N$ on $\mathbb{C}$.

\subsection{De-Rham Cohomology Classes}

The complex multiplication (CM) cycles inside Abelian surfaces with QM are
defined in Section 5 of \cite{[Be]} as certain cohomology classes. We shall
later define them as actual cycles, generalizing the classical graphs of CM
endomorphisms from the split case. Evaluating the cohomology classes of our
cycles is most easily carried out in the de-Rham setting, since it has the
advantage that the calculation is independent of the choice of quaternion
algebra $B$ with which we work. This is so, since $H^{k}_{dR}(A_{\tau}) \cong
H^{k}(A_{\tau},\mathbb{R})$ is the space of algebraic $k$-forms on the space
$T_{0}(A_{\tau})=M_{2}(\mathbb{R})$, which is independent of $B$. The Hodge
decomposition of the complexification of this space is also independent of $B$.
This approach also has the advantage of extending to some base spaces of higher
dimensions (see \cite{[Ze4]} for more details). Now, writing matrices in
$M_{2}(\mathbb{R})$ as $\binom{a\ \ b}{c\ \ d}$ defines $a$, $b$, $c$, and $d$
as algebraic 1-forms on $M_{2}(\mathbb{R})$, and we have
\begin{lem}
Extending scalars to $\mathbb{R}$, the maps $\iota$ and $\tilde{\iota}$ on
$(B_{0})_{\mathbb{R}}=M_{2}(\mathbb{R})_{0}$ may be evaluated by the equalities
\[\iota\binom{0\ \ -1}{0\ \ \ \ \ 0}=d \wedge c,\quad\iota\binom{1\ \ \ \ \
0}{0\ \ -1}=a \wedge d-b \wedge c,\quad\iota\binom{0\ \ \ 0}{1\ \ \ 0}=b \wedge
a,\] \[\tilde{\iota}\binom{0\ \ -1}{0\ \ \ \ \ 0}=c \wedge
a,\quad\tilde{\iota}\binom{1\ \ \ \ \ 0}{0\ \ -1}=d \wedge a+c \wedge
b,\quad\tilde{\iota}\binom{0\ \ \ 0}{1\ \ \ 0}=d \wedge b.\] \label{dReval}
\end{lem}

The proof is straightforward and simple. For the cup product, we prove
\begin{lem}
The form $a \wedge b \wedge c \wedge d=b \wedge a \wedge d \wedge c=d \wedge c
\wedge b \wedge a$ on $A_{\tau}$ is oriented. If $A_{\tau}$ is defined by the
ideal $I$ in $B$ then the integral of this form over $A_{\tau}$ equals
$disc(I)$. \label{volume}
\end{lem}

\begin{proof}
The assertion about orientation is immediate. For the integral, we first observe
that for $I=M_{2}(\mathbb{Z})$ it indeed equals 1. Moreover, $M_{2}(\mathbb{Z})$
is self-dual with respect to the pairing $(x,y) \mapsto Tr(x\overline{y})$ on
$M_{2}(\mathbb{R})$. For the general case, let $T \in M_{4}(\mathbb{R})$ be the
matrix representing a basis of $I$ over $\mathbb{Z}$ by $a$, $b$, $c$, and $d$.
The dual basis for $I^{*}=\big\{x \in
M_{2}(\mathbb{R})\big|(x,I)\subseteq\mathbb{Z}\big\}$ is represented by
$(T^{t})^{-1}$, so that the definition of the discriminant yields
$disc(I)^{2}=\det T/\det(T^{t})^{-1}=(\det T)^{2}$. As the integral in question
equals $|\det T|$, the lemma follows.
\end{proof}

Tensoring with $\mathbb{C}$ we obtain the Hodge types, which are determined by
the following
\begin{lem}
If $dz_{1}$ and $dz_{2}$ are the holomorphic 1-forms on the holomorphic tangent
space $\mathbb{C}^{2}$ of $A_{\tau}$ defined by writing its elements as
$\binom{z_{1}}{z_{2}}$, then the equalities \[d \wedge c=\frac{dz_{2} \wedge
d\overline{z}_{2}}{-2iy},\quad b \wedge a=\frac{dz_{1} \wedge
d\overline{z}_{1}}{-2iy},\quad a \wedge d-b \wedge c=\frac{dz_{2} \wedge
d\overline{z}_{1}+dz_{1} \wedge d\overline{z}_{2}}{2iy},\] \[dz_{2} \wedge
dz_{1}=\tilde{\iota}(M_{\tau}),\quad d\overline{z}_{2} \wedge
d\overline{z}_{1}=\tilde{\iota}(M_{\overline{\tau}}),\quad\mathrm{and}\quad
\frac{dz_{2} \wedge d\overline{z}_{1}-dz_{1} \wedge
d\overline{z}_{2}}{2}=\tilde{\iota}(yJ_{\tau})\] hold. \label{abcdz1z2}
\end{lem}

\begin{proof}
The identification $\mathbb{C}^{2}=M_{2}(\mathbb{R})\binom{\tau}{1}$ gives
\[dz_{1}=a\tau+b,\quad dz_{2}=c\tau+d,\quad
d\overline{z}_{1}=a\overline{\tau}+b,\quad\mathrm{and}\quad
d\overline{z}_{2}=c\overline{\tau}+d.\] The first three equalities follow from a
direct calculation. The other equalities are consequences of Lemma \ref{dReval}.
This proves the lemma.
\end{proof}

As Lemma \ref{volume} shows that the cup product of $\eta$ and $\omega$ in
$H^{2}(A_{\tau},\mathbb{R})$ is $Ndisc(I)$ if $\eta\wedge\omega=Na \wedge b
\wedge c \wedge d$, Lemmas \ref{dReval}, \ref{volume}, and \ref{abcdz1z2} prove
the first three parts of Theorem \ref{H2Aprop} directly. The proof of part
$(iii)$ also makes use of the explicit isomorphism between the representations
$Sym^{2}\mathbb{R}^{2}$ and $M_{2}(\mathbb{R})_{0}$ of $SL_{2}(\mathbb{R})$
(the action on the latter is by conjugation) given by
\[\binom{1}{0}^{2}\longleftrightarrow\binom{0\ \ -1}{0\ \ \ \ \ \
0},\quad\binom{1}{0}\binom{0}{1}\longleftrightarrow\frac{1}{2}\binom{1\ \ \ \ \
\ 0}{0\ \ -1},\quad\mathrm{and}\quad\binom{0}{1}^{2}\longleftrightarrow\binom{0\
\ 0}{1 \ \ 0}\!.\] Similar considerations also prove
\begin{cor}
The fundamental (cohomology) class of a 1-codimensional cycle $C$ in $A_{\tau}$
is represented by \[\frac{1}{disc(I)}\Bigg[\bigg(\int_{C}d \wedge c\bigg)b
\wedge a+\bigg(\int_{C}b \wedge a\bigg)d \wedge c-\bigg(\int_{C}d \wedge
b\bigg)c \wedge a+\] \[-\bigg(\int_{C}c \wedge a\bigg)d \wedge b+\bigg(\int_{C}c
\wedge b\bigg)d \wedge a+\bigg(\int_{C}d \wedge a\bigg)c \wedge b\Bigg].\]
\label{cycform}
\end{cor}

\begin{proof}
the cycle $C$ is represented, in terms of Poincar\'{e} duality, by the 2-form
$\omega$, if and only if the equality
\[\int_{C}\eta=\int_{A_{\tau}}\omega\wedge\eta=\int_{A_{\tau}}\eta\wedge\omega
\quad\mathrm{holds\ for\ every}\quad\eta \in H^{2}(A_{\tau},\mathbb{C}).\] The
corollary now follows from Lemma \ref{volume}.
\end{proof}

We first recover the expressions for the fundamental classes of the graphs of
the generic maps between elliptic curves. Take $I=\mathcal{M}$ to be the
classical Eichler order of level $N$ in $B=M_{2}(\mathbb{Q})$. For any $M \in
N\mathbb{Z}$ we define $\varphi_{M}:E_{\tau} \to E_{N\tau}$ to be the map
descending from multiplication by $M$ on $\mathbb{C}$. Similarly, if
$L\in\mathbb{Z}$ then $\psi_{L}:E_{N\tau} \to E_{\tau}$ descends from
multiplication by $L$ on $\mathbb{C}$. The graphs of these maps are cycles in
$A_{\tau}=E_{\tau} \times E_{N\tau}$, for which we indeed obtain
\begin{prop}
The fundamental class of the graph of $\varphi_{M}$ in $A_{\tau}$ is
\[\frac{M^{2}}{N}b \wedge a+\frac{M}{N}(a \wedge d-b \wedge c)+\frac{d \wedge
c}{N}=\iota\binom{M/N\ \ -1/N}{M^{2}/N\ \ -M/N}.\] For the graph of $\psi_{L}$
it is \[b \wedge a+L(a \wedge d-b \wedge c)+L^{2}d \wedge c=\iota\binom{L\ \
-L^{2}}{1\ \ \ -L}.\] \label{funclgen}
\end{prop}
Note that the first class here lies in $H^{2}(A_{\tau},\mathbb{Z})$ since
$\frac{c}{N} \in H^{1}(A_{\tau},\mathbb{Z})$.

\begin{proof}
The equalities $dz_{2}=Mdz_{1}$ and $d\overline{z}_{2}=Md\overline{z}_{1}$ hold
along the graph of $\varphi_{M}$, which is defined by the relation $z_{2}=Mz_{1}$. This gives $c=Ma$ and $d=Mb$ by the proof of Lemma \ref{abcdz1z2}. The integral of $b \wedge a$ over this cycle coincides with its integral over $E_{\tau}$, which equals 1. Using the relations $c=Ma$ and $d=Mb$ on this cycle we can evaluate the other integrals appearing in Corollary \ref{cycform}. This gives the desired result since $disc(I)=N$. For the graph of $\psi_{L}$ we get the equalities $a=Lc$ and $b=Ld$, while the integral of $d \wedge c$ over $E_{N\tau}$ (and over the cycle) is $N$. A similar argument now completes the proof of the proposition.
\end{proof}
In particular, the axes $E_{\tau}\times\{0\}$ and $\{0\} \times E_{N\tau}$ are
the graphs of $\varphi_{0}$ and $\psi_{0}$ respectively. Proposition
\ref{funclgen} assigns to these the classes $\frac{d \wedge c}{N}=\iota\binom{0\
\ -1/N}{0\ \ \ \ \ 0\ \ }$ and $b \wedge a=\iota\binom{0\ \ 0}{1\ \ 0}$
respectively. The third ``basis element'' $\iota\binom{1\ \ \ \ 0}{0\ \ -1}=a
\wedge d-b \wedge c$ is indeed obtained by subtracting the appropriate multiples
of the axes from the graphs of either $\varphi_{N}$ or $\psi_{1}$ (which are
dual isogenies). Hence for split $B$ the generic subgroup of $NS(A_{\tau})$ is
generated by classes of Abelian subvarieties of $A_{\tau}$. This is never the
case if $B$ is a division algebra---see Theorem \ref{Absub} below.

\medskip

Next, we consider the classical CM isogenies. Let an imaginary quadratic field
$\mathbb{K}=\mathbb{Q}(\sqrt{-D})$ be given. For any
$\tau\in\mathcal{H}\cap\mathbb{K}$, both $E_{\tau}$ and $E_{N\tau}$ have CM from
an order $\mathcal{O}$ in $\mathbb{K}$. Multiplying $D$ by an appropriate
square, we may assume that $\sqrt{-D}\in\mathcal{O}$. Over such an element
$\tau$ of $\mathcal{H}$ one can define the CM isogenies ``multiplication by
$N\sqrt{-D}$" from $E_{\tau}$ to $E_{N\tau}$ and ``multiplication by
$\sqrt{-D}$" from $E_{N\tau}$ to $E_{\tau}$. Let us reproduce the results about
their fundamental classes in $A_{\tau}=E_{\tau} \times E_{N\tau}$:
\begin{prop}
The fundamental class of the graph of the latter isogeny is \[b \wedge a+Dd
\wedge c+\frac{\sqrt{D}}{y}\big[|\tau|^{2}c \wedge a+x(d \wedge a+c \wedge b)+d
\wedge b\big]=\iota\binom{0\ \ -D}{1\ \ \ \ \ 0\
}+\tilde{\iota}(\sqrt{D}J_{\tau}).\] For the former, we get \[\frac{d \wedge
c}{N}+NDb \wedge a-\frac{\sqrt{D}}{y}\big[|\tau|^{2}c \wedge a+x(d \wedge
a+c \wedge b)+d \wedge b\big],\] which equals $\iota\binom{\ \ 0\ \ \ -1/N}{ND\
\ \ \ 0\ \ }-\tilde{\iota}(\sqrt{D}J_{\tau})$.
\label{funclCME}
\end{prop}

\begin{proof}
Along the graph of the CM isogeny $E_{N\tau} \to E_{\tau}$ we have
$dz_{1}=\sqrt{-D}dz_{2}$ and $d\overline{z}_{1}=-\sqrt{-D}d\overline{z}_{2}$,
which translate to $(a\ \ b)=(c\ \ d)\sqrt{D}J_{\tau}$ (as row vectors of length
2). The integral of $d \wedge c$ over the cycle is $N$ (like with $\psi_{L}$
from Proposition \ref{funclgen}). Evaluating the other integrals along the cycle
by relating the 2-forms to $d \wedge c$ and substituting the results as well as
the value $N$ of $disc(I)$ into Corollary \ref{cycform} yields the asserted
value. The equalities satisfied along the graph of the other CM isogeny are
$dz_{2}=N\sqrt{-D}dz_{1}$ and $d\overline{z}_{2}=-N\sqrt{-D}d\overline{z}_{1}$,
whence $(c\ \ d)=(a\ \ b)N\sqrt{D}J_{\tau}$. The integral of $b \wedge a$ along
this graph is 1 (as for $\varphi_{M}$ in Proposition \ref{funclgen}), and again
we obtain the required fundamental class. For the description using $\iota$ and
$\tilde{\iota}$ we apply Lemma \ref{dReval}. This proves the proposition.
\end{proof}

The generic parts in the fundamental classes appearing in Proposition
\ref{funclCME} are multiples of the axes $E_{\tau}\times\{0\}$ and $\{0\} \times
E_{N\tau}$ coming from the intersection numbers. In view of Theorem
\ref{funcltau0} below, we mention that they are equal
$\iota(\sqrt{D}J_{\sqrt{-D}})$ and $\iota\big(\sqrt{D}J_{-1/N\sqrt{-D}}\big)$
respectively. The normalized fundamental classes are the transversal classes
$\pm\tilde{\iota}(\sqrt{D}J_{\tau})$, indeed of Hodge type $(1,1)$ by part
$(iii)$ of Theorem \ref{H2Aprop}. The different signs are related to the fact
that these isogenies are not dual to one another, but to minus one another.

\subsection{CM Points and CM Cycles}

Let $B$ now be arbitrary, with $i$, $\mathcal{M}$, $I$, $\tau$, $A_{\tau}$, and
$\pi:\mathcal{A}\to\mathcal{H}$ as above. For every $\tau\in\mathcal{H}$,
$NS(A_{\tau})$ contains the generic part $\iota(\Lambda)$ (hence is of rank at
least 3), but it cannot be of rank exceeding
$\dim_{\mathbb{C}}H^{1,1}(A_{\tau})=4$. The CM points, which depend on $B$ and
on the isomorphism $i$, are characterized by the following extension of Lemma
7.2 of \cite{[Be]}:

\begin{prop}
For $\tau\in\mathcal{H}$, the following conditions are equivalent:
\begin{enumerate}[$(i)$]
\item There exists $b \in B\setminus\mathbb{Q}$ with $N(b)>0$ such that $b=i(b)
\in GL_{2}^{+}(\mathbb{R})$ fixes $\tau$.
\item Some (real) multiple of $J_{\tau}$ lies in $B=i(B)$.
\item The rank of $NS(A_{\tau})$ is 4 (rather than the generic 3).
\item There is a non-trivial endomorphism of $A_{\tau}$ commuting with the
action of $\mathcal{M}$.
\end{enumerate}\label{CMptchar}
\end{prop}

\begin{proof}
The equivalence of $(i)$ and $(ii)$ follows from the fact that the stabilizer of
$\tau$ in $GL_{2}^{+}(\mathbb{R})$ consists of the matrices of the form
$dI+cJ_{\tau}$ with $c$ and $d$ real. $(ii)$ and $(iii)$ are equivalent since
$H^{1,1}(A_{\tau}) \cap \iota(B_{0})^{\perp}$ consists of multiples of
$\tilde{\iota}(J_{\tau})$, using the Lefschetz Theorem on $(1,1)$ classes. Now,
Lemma \ref{mapsAtau} implies that endomorphisms of $A_{\tau}$ which commute with
the action of $\mathcal{M}$ are given by right multiplication by a matrix from
$R(I)$ whose action preserves $\tau$. As such a non-zero element must have a
positive reduced norm, the equivalence of $(ii)$ and $(iv)$ also follows. This
proves the proposition.
\end{proof}

The proof of Proposition \ref{CMptchar} shows that if $A_{\tau}$ has CM then
$End_{\mathcal{M}}(A_{\tau})$ is generated by some element of the form
$b=cJ_{\tau}$ of $B$, whose square $-c^{2}=-N(b)$ is rational and negative.
Hence $End_{\mathcal{M}}(A_{\tau})$ is an order in an imaginary quadratic field,
which embeds into $B$. This imaginary quadratic field thus splits $B$.

\smallskip

Fixing an imaginary quadratic field $\mathbb{K}$ which splits $B$, we need to
normalize certain embeddings. Fix one embedding of $\mathbb{K}$ into
$\mathbb{C}$ (which we write just as inclusion), so that henceforth
$\sqrt{-D}\in\mathbb{K}$ with positive $D\in\mathbb{Q}$ will mean the element
of $\mathbb{K}\cap\mathcal{H}$ via this inclusion. Any embedding of $\mathbb{K}$
into $B \subseteq M_{2}(\mathbb{R})$ sends $\sqrt{-D}\in\mathbb{K}$ to
$\pm\sqrt{D}J_{\tau}$ for some $\tau\in\mathcal{H}$. We call this embedding
\emph{normalized} if the sign is positive. Considering the action of the
endomorphism $\sqrt{D}J_{\tau}$ of $A_{\tau}$ appearing in part $(iv)$ of
Proposition \ref{CMptchar} on $H^{1,0}(A_{\tau})$, one easily verifies that our
normalization is equivalent to the one given in Section 7 of \cite{[Be]}. All
the embeddings henceforth will be assumed to be normalized.

Recall that in the split case, all the points from $\mathcal{H}\cap\mathbb{K}$
for a fixed imaginary quadratic field $\mathbb{K}$ are related by the action of
$GL_{2}^{+}(\mathbb{Q})$. We now extend this assertion to any quaternion algebra
$B$. Let $\tau_{0}\in\mathcal{H}$ be the fixed point of a normalized embedding
of some imaginary quadratic field $\mathbb{K}$ (of discriminant $-D$) into $B$.
Thus $A_{\tau_{0}}$ has CM from an order in $\mathbb{K}$ and
$\sqrt{D}J_{\tau_{0}} \in B$ by Proposition \ref{CMptchar}. We now prove

\begin{lem}
For a point $\tau\in\mathcal{H}$, the following are equivalent:
\begin{enumerate}[$(i)$]
\item The Abelian surface $A_{\tau}$ (with QM from $\mathcal{M}$) has CM from an
order in $\mathbb{K}$.
\item $\tau_{0}=\gamma\tau$ for some $\gamma \in B$ with positive norm.
\end{enumerate}\label{CMK}
\end{lem}

\begin{proof}
Proposition \ref{CMptchar} shows that condition $(i)$ is equivalent to
$\sqrt{D}J_{\tau}$ being in $B$. Combining Equation \eqref{Jtauconj}, the
Skolem-Noether Theorem, and the fact that $J_{\tau}$ determines $\tau$ with the
fact that $\det\gamma$ must be positive if $\tau_{0}=\gamma\tau$ holds with
$\tau$ and $\tau_{0}$ from $\mathcal{H}$ now yields the desired equivalence.
This proves the lemma.
\end{proof}

Scalar multiplication from $\mathbb{Z}$ can take the element $\gamma$ in
condition $(ii)$ of Lemma \ref{CMK} to any order or ideal in $B$ of our choice.
Moreover, Lemma \ref{CMK} extends to $\tau_{0}\in\overline{\mathcal{H}}$ by
allowing $\gamma$ to have negative reduced norm. This will be useful for
defining the most general CM cycles in $A_{\tau}$ below.

\medskip

We consider the graphs of CM isogenies from the split $B$ case as the images in
$A_{\tau}$ of the lines $\mathbb{C}\binom{1}{N\sqrt{-D}}$ and
$\mathbb{C}\binom{\sqrt{-D}}{1}$ in the universal cover $\mathbb{C}^{2}$
respectively. Multiplying $D$ by an integral square changes the slopes of these
lines, yielding many CM cycles in $A_{\tau}$. The graphs of the generic maps
$\varphi_{M}$ with $M \in N\mathbb{Z}$ or $\psi_{L}$ for $L\in\mathbb{Z}$ are
described similarly, using the lines $\mathbb{C}\binom{1}{M}$ and
$\mathbb{C}\binom{L}{1}$ respectively. With this motivation, let $B$ (and $i$)
be general, let $\tau_{0}\in\mathcal{H}\cup\overline{\mathcal{H}}$ and
$\mathbb{K}$ be as in (the extended version of) Lemma \ref{CMK}, and assume that
the Abelian surface $A_{\tau}$ with QM has also CM from an order in
$\mathbb{K}$. We now prove
\begin{prop}
The image of the line $\mathbb{C}\binom{\tau_{0}}{1}\subseteq\mathbb{C}^{2}$ in
$A_{\tau}$ is an Abelian subvariety of $A_{\tau}$. Furthermore, it is isomorphic
to an elliptic curve with CM from $\mathbb{K}$. \label{CMcyc}
\end{prop}

\begin{proof}
It suffices to show that the intersection
$I\binom{\tau}{1}\cap\mathbb{C}\binom{\tau_{0}}{1}$ is a full lattice in that
line. Lemma \ref{CMK} produces an element $\gamma \in B$ such that
$\gamma\tau=\tau_{0}$, which we assume to lie in $I$ and be primitive there.
Hence $\gamma\binom{\tau}{1} \in I\binom{\tau}{1}$ lies in
$\mathbb{C}\binom{\tau_{0}}{1}$ by Equation \eqref{V1mod}. But $\sqrt{D}J_{\tau}
\in B$ by Proposition \ref{CMptchar}, and by renormalizing $D$ we may assume
that this element lies in $R(I)$. As $\sqrt{D}J_{\tau}$ multiplies
$\binom{\tau}{1}$ by $\sqrt{-D}\in\mathbb{C}\setminus\mathbb{R}$ (Equation
\eqref{V1mod} again), we find that $\gamma\cdot\sqrt{D}J_{\tau}\binom{\tau}{1}$
is another element of $\mathbb{C}\binom{\tau_{0}}{1}$, which is linearly
independent of the previous one over $\mathbb{R}$. Hence the image of
$\mathbb{C}\binom{\tau_{0}}{1}$ in $A_{\tau}$ is an Abelian subvariety of
$A_{\tau}$, which is an elliptic curve. As $\gamma$ is primitive in $I$, we may
find integers $g>0$ and $h$ such that $\gamma\binom{\tau}{1}$ and
$\gamma\frac{\sqrt{D}J_{\tau}-h}{g}\binom{\tau}{1}$ generate the lattice
$I\binom{\tau}{1}\cap\mathbb{C}\binom{\tau_{0}}{1}$. Our elliptic curve is thus
isomorphic to $E_{\tilde{\tau}}$ for
$\tilde{\tau}=\frac{\sqrt{-D}-h}{g}\in\mathcal{H}$, indeed having CM by
$\sqrt{-D}\in\mathbb{K}$. This proves the proposition.
\end{proof}

Proposition \ref{CMcyc} allows us to make the following
\begin{defn}
Given $\tau$ and $\tau_{0}$ as above, we define the \emph{CM cycle corresponding
to $\tau_{0}$} in $A_{\tau}$ to be the cycle described in Proposition
\ref{CMcyc}. \label{CMdef}
\end{defn}
Definition \ref{CMdef} includes the case of graphs of CM isogenies in the split
case from above, as the cases $\tau_{0}=\sqrt{-D}\in\mathcal{H}$ and
$\tau_{0}=\frac{1}{N\sqrt{-D}}\in\overline{\mathcal{H}}$ respectively. The
graphs of the generic maps may be seen as an extension of Definition \ref{CMdef}
to some rational (or infinite) $\tau_{0}$---more on this below.

\smallskip

Let $\tau_{0}=x_{0}+iy_{0}\in\mathcal{H}\cup\overline{\mathcal{H}}$ be as above.
We prove
\begin{thm}
The fundamental class of the CM cycle corresponding to $\tau_{0}$ equals
$sgn(y_{0})\big[\iota(\beta\sqrt{D}J_{x_{0}+i|y_{0}|})+\tilde{\iota}(\beta\sqrt{
D}J_{\tau})\big]$ for some positive $\beta\in\mathbb{Q}$. \label{funcltau0}
\end{thm}

\begin{proof}
The proof of Proposition \ref{CMcyc} shows that as a subset of $A_{\tau}$, the
CM cycle in question is
$\mathbb{C}\binom{\tau_{0}}{1}/j(\gamma,\tau)(\mathbb{Z}\tilde{\tau}
\oplus\mathbb{Z})\binom{\tau_{0}}{1}$ (recall the coefficient $j(\gamma,\tau)$
in Equation \eqref{V1mod}). Along this cycle the relations
$dz_{1}=\tau_{0}dz_{2}$ and
$d\overline{z}_{1}=\overline{\tau_{0}}d\overline{z}_{2}$ hold, and in $A_{\tau}$
these relations are equivalent to $(a\ \ b)=(c\ \ d)(x_{0}I+y_{0}J_{\tau})$. As
in Propositions \ref{funclgen} and \ref{funclCME}, evaluating the integral of $d
\wedge c$ over our cycle suffices for determining the required fundamental
class. Using the first formula in Lemma \ref{abcdz1z2}, we evaluate the integral
of $dz_{2} \wedge d\overline{z}_{2}$ instead. Now, the isomorphism of our cycle
with $E_{\tilde{\tau}}$ is given by the second coordinate divided by
$j(\gamma,\tau)$. Hence we may substitute $dz_{2}=j(\gamma,\tau)dz$ and
$d\overline{z}_{2}=\overline{j(\gamma,\tau)}d\overline{z}$, where $z$ is the
coordinate of $E_{\tilde{\tau}}$. A straightforward evaluation now gives
\begin{equation}
\int_{\left(\mathbb{C}\binom{\tau_{0}}{1} \to A_{\tau}\right)}\!\!dz_{2}
\wedge d\overline{z}_{2}=|j(\gamma,\tau)|^{2}\!\int_{E_{\tilde{\tau}}}\!dz
\wedge d\overline{z}=-2i|j(\gamma,\tau)|^{2}\Im\tilde{\tau}=-2iy\frac{
\det\gamma\cdot\Im\tilde{\tau}}{y_{0}}, \label{intCMcyc}
\end{equation}
where in the last step we have used the equality $y_{0}=\frac{\det\gamma \cdot
y}{|j(\gamma,\tau)|^{2}}$ (recall that $\tau_{0}=\gamma\tau$, and $\det\gamma$
is not necessarily 1). Dividing by $-2iy$ to get the integral of $d \wedge c$
and evaluating the other integrals appearing in Corollary \ref{cycform} as in
the proofs of Propositions \ref{funclgen} and \ref{funclCME} gives the
expression
\[\frac{\det\gamma\cdot\Im\tilde{\tau}}{disc(I)}\bigg[\frac{b \wedge
a}{y_{0}}+\frac{|\tau_{0}|^{2}}{y_{0}}d \wedge c+\frac{|\tau|^{2}}{y}c \wedge
a+\frac{d \wedge b}{y}+\bigg(\frac{x}{y}-\frac{x_{0}}{y_{0}}\bigg)d \wedge
a+\bigg(\frac{x}{y}+\frac{x_{0}}{y_{0}}\bigg)c \wedge b\bigg]\] for the
fundamental class in question. Lemma \ref{dReval} reduces the latter expression
to
$\frac{\det\gamma\cdot\Im\tilde{\tau}}{disc(I)}(\iota(J_{\tau_{0}})+\tilde{\iota
}(J_{\tau}))$. Now, the coefficient
$\frac{\det\gamma\cdot\Im\tilde{\tau}}{disc(I)}$ equals $\tilde{\beta}\sqrt{D}$
for some $\tilde{\beta}\in\mathbb{Q}$ having the same sign as $\det\gamma$ as
well as $y_{0}$. Putting $\beta=|\tilde{\beta}|$ and using the relation
$J_{\tau_{0}}=-J_{\overline{\tau_{0}}}$ if $\tau_{0}\in\overline{\mathcal{H}}$
completes the proof of the theorem.
\end{proof}

The proofs of Proposition \ref{CMcyc} and Theorem \ref{funcltau0} involve the
parameter $\tilde{\tau}$, which depends on the choice of the primitive element
$\gamma \in I$ taking $\tau$ to $\tau_{0}$. However, choosing another such
element $\delta$ and taking the appropriate $g$ and $h$ can be seen to yield
the index $\alpha\tilde{\tau}$ for some $\alpha \in SL_{2}(\mathbb{Z})$.
Moreover, $\gamma^{-1}\delta$ has reduced norm $|j(\alpha,\tilde\tau)|^{2}$ in
this case. As $E_{\alpha\tilde{\tau}} \cong E_{\tilde{\tau}}$ and
$\det\delta\cdot\Im\alpha\tilde{\tau}=\det\gamma\cdot\Im\tilde{\tau}$, both the
description of the CM cycle and the coefficient we wrote as
$\tilde{\beta}\sqrt{D}$ are intrinsic. The actual value of $\beta$ is probably
the minimal positive rational number such that both terms give elements of
$H^{2}(A_{\tau},\mathbb{Z})$, but this point requires further investigation.

We remark that the generic cycles in the split $B$ case can be obtained as the
case $\tau_{0}\in\mathbb{P}^{1}(\mathbb{R})$ of Proposition \ref{CMcyc} and
Theorem \ref{funcltau0}. Indeed, by keeping $|j(\gamma,\tau)|^{2}$ instead of
the (here undefined) expression $\frac{\det\gamma}{y_{0}}$ in Equation
\eqref{intCMcyc}, the same argument works if $\det\gamma=0$ and
$\tau_{0}\in\mathbb{R}$. Similar manipulations extend the assertion to
$\tau_{0}=\infty$ as well. Note that the sign of the normalized cycle from
Theorem \ref{funcltau0} distinguishes between lines having slopes in
$\mathcal{H}$ from those with slopes in $\overline{\mathcal{H}}$ (whence the
signs in Proposition \ref{funclCME}). The vanishing of the $\tilde{\iota}$
part for $\tau_{0}\in\mathbb{P}^{1}(\mathbb{R})$ in Proposition \ref{funclgen}
corresponds to the fact that $\mathbb{P}^{1}(\mathbb{R})$ lies between
$\mathcal{H}$ and $\overline{\mathcal{H}}$.

We remark that this method of calculating fundamental classes of CM cycles is
equivalent to the more direct approach of finding intersection numbers of the
cycles themselves (see, e.g., the proof of Corollary \ref{cycform}). However, it
has the advantage of avoiding the need to analyze the structure of the specific
quaternion algebra $B$. We also note that our set-theoretic presentation depends
on the embedding $i$. However, by the Skolem--Noether Theorem, this makes no
essential difference algebraically.

\smallskip

An Abelian surface $A_{\tau}$ with QM having CM from an order in a field
$\mathbb{K}$ contains infinitely many different CM cycles. For these we prove
\begin{prop}
For $\tau$ as above we have
\begin{enumerate}[$(i)$]
\item The action of $\mathcal{M} \subseteq End(A_{\tau})$ relates all the CM
cycles in $A_{\tau}$ to one another.
\item Given $M \in R(I)^{\times}_{+}$, the isomorphism $A_{M\tau} \to A_{\tau}$ from Lemma \ref{mapsAtau} takes the CM cycle corresponding to $\tau_{0}$ in $A_{M\tau}$ to the one in $A_{\tau}$.
\item If $A_{\tau}$ and its QM and CM endomorphisms are defined over a number
field $\mathbb{F}$ then all the CM cycles are also defined over $\mathbb{F}$.
\item $A_{\tau}$ is isogenous, over $\mathbb{F}$, to the self-product of an
elliptic curve with CM from an order in $\mathbb{K}$.
\end{enumerate}\label{CMrel}
\end{prop}

\begin{proof}
Part $(i)$ follows from Lemma \ref{CMK}. For part $(ii)$, Lemma \ref{mapsAtau}
shows that the isomorphism in question multiplies the holomorphic tangent space
$\mathbb{C}^{2}$ by a scalar. Thus, this map preserves lines in
$\mathbb{C}^{2}$, hence also CM cycles. To prove part $(iii)$ note that Equation
\eqref{V1mod} characterizes $\mathbb{C}\binom{\tau_{0}}{1}$ as those vectors in
$\mathbb{C}^{2}$ on which the QM endomorphism
$\sqrt{D}J_{\tau_{0}}\in\mathcal{M} \subseteq B$ acts like the CM endomorphism
$\sqrt{-D}\in\mathcal{O}\subseteq\mathbb{K}\subseteq\mathbb{C}$ (with positive
imaginary part). By connectedness, the CM cycle corresponding to $\tau_{0}$ is
thus the connected component of the kernel of the endomorphism
$\sqrt{D}J_{\tau_{0}}-\sqrt{-D}\in\mathcal{M}\otimes\mathcal{O}=End(A_{\tau})$.
The kernel itself is clearly defined over $\mathbb{F}$. For the connected
component, observe that any  automorphism of $\mathbb{C}/\mathbb{F}$ gives an
isomorphism of Abelian varieties from $A_{\tau}$ to itself, preserves the kernel
in question, takes connected components to connected components, and maps 0 to
0. Hence any such automorphism preserves the CM cycle in question (this argument
was shown to me by Philipp Habegger), which completes the proof of part $(iii)$.
For part $(iv)$ we recall that embedding two different CM cycles into $A_{\tau}$
defines an isogeny between the product of these CM cycles and $A_{\tau}$, and
that any two elliptic curves with CM from $\mathbb{K}$ are isogenous. The
desired assertion now follows from part $(iii)$ and Proposition \ref{CMcyc}.
This proves the proposition.
\end{proof}

Part $(i)$ of Proposition \ref{CMrel} shows that like with polarizations, there
is no ``canonical'' choice of a CM cycle in $A_{\tau}$. Part $(ii)$ implies that
the CM cycles are well-defined also on fibers over $X(\Gamma)$ for $\Gamma
\subseteq R(I)^{\times}_{+}$. They may therefore be considered as cycles of
codimension 2 in the (algebraic) universal family $\mathcal{A}$ over
$X(\Gamma)$. Parts $(iii)$ and $(iv)$ essentially appear in Section 7 of
\cite{[Be]}, but our proof is simpler and more elementary. The first three
assertions of Proposition \ref{CMrel} extend to the generic cycles in the split
$B$ case, when part $(i)$ is interpreted appropriately.

\medskip

We now have two kinds of (1-dimensional) Abelian subvarieties of an Abelian
surface $A$ with QM: The split $B$ case gives rise to generic cycles (graphs of
generic correspondences), while the CM case (with general $B$) introduces CM
cycles. On the other hand, we can prove
\begin{thm}
For $\tau\in\mathcal{H}$ we have
\begin{enumerate}[$(i)$]
\item If the Abelian surface $A_{\tau}$ with QM has also CM then the lines with
finite non-real slopes in $\mathbb{C}^{2}$ which give Abelian subvarieties of
$A_{\tau}$ are exactly the lines descending to CM cycles as in Proposition
\ref{CMcyc}.
\item If $A_{\tau}$ has no CM then no lines of non-real slopes give Abelian
subvarieties of $A_{\tau}$.
\item If $B=M_{2}(\mathbb{Q})$ then a line of slope from
$\mathbb{P}^{1}(\mathbb{R})$ gives an Abelian subvariety of $A_{\tau}$ if and
only if the slope is in $\mathbb{P}^{1}(\mathbb{Q})$.
\item If $B$ is not split then no lines of real or infinite slope give Abelian
subvarieties of $A_{\tau}$.
\end{enumerate}\label{Absub}
\end{thm}

\begin{proof}
We need to find all slopes $\tau_{0}$ such that $I\binom{\tau}{1}$ intersects
$\mathbb{C}\binom{\tau_{0}}{1}$ in a full lattice in this line. Equivalently, by
Equation \eqref{V1mod} and its extensions we need two linearly independent
elements of $B$ which take $\tau$ to $\tau_{0}$. Now, for part $(i)$ we must
have $\tau_{0}=\gamma\tau$ for $\gamma \in B^{\times}$, and the assertion
follows as in Lemmas \ref{CMptchar} and \ref{CMK}. Now, for two elements
$\gamma$ and $\delta$ sending $\tau$ to $\tau_{0}$, $\gamma^{-1}\delta$
stabilizes $\tau$. Hence part $(ii)$ follows from Proposition \ref{CMptchar} as
well. In addition, the equality $\tau_{0}=\gamma\tau$ with $\tau\in\mathcal{H}$,
$\gamma\neq0$, and $\tau_{0}\in\mathbb{P}^{1}(\mathbb{R})$ implies
$N(\gamma)=0$. Part $(iv)$ is thus also established. For part $(iii)$, recall
that an element $0\neq\gamma \in M_{2}(\mathbb{Q})$ satisfies $\det\gamma=0$ if
and only if its rows are rational multiples of one another. Hence only slopes
from $\mathbb{P}^{1}(\mathbb{Q})$ may be obtained. Finally, for
$\tau_{0}=\frac{L}{M}\in\mathbb{P}^{1}(\mathbb{Q})$ (with
$\infty=\frac{1}{0}$), the elements $\binom{0\ \ L}{0\ \ M}$ and $\binom{L\ \
0}{M\ \ 0}$ of $M_{2}(\mathbb{Q})$ send $\binom{\tau}{1}$ to $\binom{L}{M}$ and
$\tau\binom{L}{M}$, respectively. Thus, lines with such slopes do give Abelian
subvarieties, which proves $(iii)$. This completes the proof of the theorem.
\end{proof}

Theorem \ref{Absub} shows that the only Abelian subvarieties of an Abelian
surface $A$ with QM are the CM cycles (if they exist), plus the generic cycles
if $B$ splits. This result is related to rational points on Brauer--Severi
varieties (the algebra is $End(A)_{\mathbb{Q}}$ considered over its center, the
latter being $\mathbb{K}$ in the case of CM and $\mathbb{Q}$ otherwise). Recall
that any Abelian subvariety $C$ of an Abelian variety $A$ admits a
``complementary'' subvariety $\widetilde{C}$ of $A$ such that $A$ is isogenous
to $C\times\widetilde{C}$ (see, for example, the proof of Proposition 10.1 of
\cite{[M]}). Expressing the latter result in terms of (connected components of)
kernels of endomorphisms of $A$ shows, together with the proof of part $(iii)$
of Proposition \ref{CMrel} here (extended also to the generic cycles), that our
Theorem \ref{Absub} agrees with the expected result. However, it also provides
an explicit, set-theoretic description of the Abelian subvarieties in this case
of Abelian surfaces with QM.

\medskip

Let $\mathcal{A}^{m}$ be the fibered product of $\mathcal{A}$ with itself $m$
times over $\mathcal{H}$, with the projection
$\pi_{m}:\mathcal{A}^{m}\to\mathcal{H}$. Given a CM cycle $C \subseteq
A_{\tau}$, we define its $m$th power ot be \[C^{m}=\big\{(z_{1},\ldots,z_{m})
\in A_{\tau}^{m}\big|z_{j} \in C,1 \leq j \leq m\big\} \subseteq A_{\tau}^{m}.\]
This $m$-cycle in $A_{\tau}^{m}$ is also called a CM cycle. We consider $C^{m}$
as a vertical cycle in $\mathcal{A}^{m}$, which is algebraic in $A_{\tau}^{m}$
as well as in $\mathcal{A}^{m}$ over $X(\Gamma)$. Such a cycle has a fundamental
cohomology class in $H^{2m}(A_{\tau}^{m})$, or more precisely in the
$S_{m}$-invariant part $Sym^{m}H^{2}(A_{\tau})$ of the K\"{u}nneth component
$H^{2}(A_{\tau})^{\otimes m}$ of the latter cohomology group. By letting $\tau$
vary, the normalized fundamental classes of these cycles give elements of the
fiber of the symmetric part $Sym^{m}V_{2}$ of the subvariation of Hodge
structure $V_{2}^{\otimes m}\subseteq(R^{2}\pi_{*}\mathbb{C})^{\otimes m}$ of
$R^{2m}\pi_{m,*}\mathbb{C}$ over $\mathcal{H}$ or over $X(\Gamma)$. Now,
$Sym^{m}V_{2}$ admits a smaller rational subvariation of Hodge structure which
is isomorphic to $V_{2m}$ (it is defined as the kernel of a certain Laplacian
map to $Sym^{m-2}V_{2}$). The objects referred to as CM cycles in Section 5 of
\cite{[Be]} are the images of the normalized fundamental classes of our
$m$-dimensional CM cycles under the projection onto this subvariation of
Hodge structure (denoted $P$ in that reference). Theorem \ref{main} below proves
the modularity of our cycles in the larger subvariation of Hodge structure
$Sym^{m}V_{2}$. We note that \cite{[FM1]} and \cite{[FM2]} have related
modularity results, but with the local system $V_{2m}$ (using other means like
the Shintani and Kudla--Millson lifts). For the split $B$ case, the
$m$-codimensional cycles arising from $\tau_{0}=\sqrt{-D}$ resemble the cycles
defined in page 123 of \cite{[Zh]} (for $m=k-1$). The action of the larger group
$S_{2m}$ (rather than our $S_{m}$) appearing in \cite{[Zh]} is the incarnation
of the projector $P$ in this case.

\section{Differential Operators on Theta Functions \label{RelDO}}

In this Section we prove a relation between the actions of two Laplacians on
Siegel theta functions with polynomials of a specific type. This relation is a
crucial step in showing that the theta lifts considered in the next Section are
eigenfunctions of the Laplacian on the Grassmannian. This property is central in
the following applications.

\subsection{The Basic Summand in a Theta Function}

Let $V$ be a real vector space with a bilinear form $(\cdot,\cdot):V \times
V\to\mathbb{R}$, which is non-degenerate with signature $(b_{+},b_{-})$. We
shorthand $(\lambda,\lambda)$ to $\lambda^{2}$. Hence the associated quadratic
form, denoted $q$ in \cite{[Bru]}, is $\lambda\mapsto\frac{\lambda^{2}}{2}$.
Consider the \emph{Grassmannian} \[G(V)=\{V=v_{+} \oplus
v_{-}|v_{+}\gg0,v_{-}\ll0,v_{+} \perp v_{-}\}\] (i.e., $v_{+}$ is positive
definite and $v_{-}$ is negative definite) of $V$, in elements of which we have
$\dim v_{\pm}=b_{\pm}$. Given $\lambda \in V$, its projection onto $v_{\pm}$ in
a given element of $G(V)$ is denoted $\lambda_{v_{\pm}}$. The Laplacian of
$v_{\pm}$ is
$\Delta_{v_{\pm}}=\pm\sum_{h=1}^{b_{\pm}}\frac{\partial^{2}}{\partial\lambda_{h}
^{2}}$ using an orthonormal basis for $v_{\pm}$. The Laplacian $\Delta_{V}$ of
$V$ is $\Delta_{v_{+}}+\Delta_{v_{-}}$ (this is independent of the choice of the
decomposition), and the \emph{Laplacian corresponding to $v \in G(V)$} is
$\Delta_{v}=\Delta_{v_{+}}-\Delta_{v_{-}}$ (this operator is based on the
majorant corresponding to $v$). For $\tau\in\mathcal{H}$, $v \in G(V)$,
a polynomial $p$ on $V$, and $\lambda \in V$, define the function
\[F(\tau,v,p,\lambda)=e^{-\Delta_{v}/8\pi
y}(p)(\lambda)\mathbf{e}\bigg(\tau\frac{\lambda_{v_{+}}^{2}}{2}+\overline{
\tau}\frac{\lambda_{v_{-}}^{2}}{2}\bigg),\] where $\mathbf{e}(z)=e^{2\pi iz}$
for any complex number $z$. The polynomial $p$ involved with these functions
will usually be \emph{homogenous of degree $(m_{+},m_{-})$ with respect to
$v$}. This means that for $\lambda \in V$, which decomposes as
$\lambda_{v_{+}}+\lambda_{v_{-}}$ with respect to $v$, and two numbers
$\alpha_{\pm}\in\mathbb{R}$, we have
$p(\alpha_{+}\lambda_{v_{+}}+\alpha_{-}\lambda_{v_{-}})=\alpha_{+}^{m_{+}}
\alpha_{-}^{m_{-}}p(\lambda)$.

An \emph{even lattice} in $V$ is a discrete subgroup $L \subseteq V$ of maximal
rank $b_{+}+b_{-}$ such that $\lambda^{2}\in2\mathbb{Z}$ for any $\lambda \in
L$. The dual lattice $L^{*}=Hom(L,\mathbb{Z})$ is a subgroup of $V$ which
contains $L$ with finite index $\Delta_{L}$. The \emph{Siegel theta function of
$L$}, with respect to the element $v \in G(V)$ and a polynomial $p$ which is
homogenous of degree $(m_{+},m_{-})$ with respect to $v$, is the
$\mathbb{C}[L^{*}/L]$-valued
function \[\Theta_{L}(\tau,v,p)=\sum_{\gamma \in
L^{*}/L}\theta_{\gamma}(\tau,v,p)e_{\gamma},
\qquad\theta_{\gamma}(\tau,v,p)=\sum_{\lambda \in
L+\gamma}F(\tau,v,p,\lambda).\] Here $e_{\gamma}$ is the canonical basis
element of $\mathbb{C}[L^{*}/L]$ which corresponds to the element $\gamma \in
L^{*}/L$.

The group $SL_{2}(\mathbb{R})$ admits a double cover, known as the The
\emph{metaplectic group} $Mp_{2}(\mathbb{R})$. Its elements are pairs
$(M,\varphi)$ with $M \in SL_{2}(\mathbb{R})$ and
$\varphi:\mathcal{H}\to\mathbb{C}$ holomorphic such that
$\varphi^{2}(\tau)=j(M,\tau)$ for every $\tau\in\mathcal{H}$. The product rule
in $Mp_{2}(\mathbb{R})$ takes $(M,\varphi)$ and $(N,\psi)$ to $\big(MN,(\varphi
\circ N)\cdot\psi\big)$. Let $Mp_{2}(\mathbb{Z})$ denote the inverse image of
$SL_{2}(\mathbb{Z})$ in $Mp_{2}(\mathbb{R})$. It is generated by the two
elements $T=\big(\binom{1\ \ 1}{\ \ \ \ 1},1\big)$ and $S=\big(\binom{\ \ -1}{1\
\ \ \ },\sqrt{\tau}\in\mathcal{H}\big)$, satisfying the relation
$S^{2}=(ST)^{3}$. This common element $Z=(-I,i)$ has order 4 and generates the
center of $Mp_{2}(\mathbb{Z})$ (as well as that of $Mp_{2}(\mathbb{R})$). The
kernel of the projection to $SL_{2}(\mathbb{R})$ is generated by $Z^{2}=(I,-1)$.
Let $\rho_{V}$ be the Weil representation of $Mp_{2}(\mathbb{R})$ on the space
$\mathcal{S}(V)$ of Schwartz functions on $V$. Explicitly, let $\big(M=\binom{a\
\ b}{c\ \ d},\varphi\big)$ and a Schwartz function $\Phi$ on $V$ be given. If
$c=0$, so that $\varphi(\tau)$ is the constant $\delta\sqrt{|d|}$ for some
$\delta\in\{\pm1,\pm i\}$, then $\rho_{V}(M,\varphi)\Phi$ is the function
\[\lambda\mapsto\delta^{b_{-}-b_{+}}|a|^{(b_{+}+b_{-})/2}\Phi(a\lambda)\mathbf{e
}\bigg(\frac{ab\lambda^{2}}{2}\bigg).\] Otherwise, it gives the function sending
$\lambda \in V$ to
\[\Big(sgn\big(\Re\varphi(\tau)\big)\zeta_{8}^{sgn(c)}\Big)^{b_{-}-b_{+}}|c|^
{-(b_{+}+b_{-})/2}\int_{V}\Phi(\mu)\mathbf{e}\bigg[\frac{d\mu^{2}}{2c}-\frac{
(\lambda,\mu)}{c}+\frac{a\lambda^{2}}{2c}\bigg]d\mu.\] Here
$\zeta_{8}=\mathbf{e}\big(\frac{1}{8}\big)$ is the basic 8th root of unity.
Furthermore, we denote $\rho_{L}$ the \emph{Weil representation of
$Mp_{2}(\mathbb{Z})$} on $\mathbb{C}[L^{*}/L]$, which is defined on the
generators by the well-known formulae
\[\rho_{L}(T)(e_{\gamma})=\mathbf{e}(\gamma^{2}/2)e_{\gamma},\]
\[\rho_{L}(S)(e_{\gamma})=\frac{\zeta_{8}^{b_{-}-b_{+}}}{\sqrt{\Delta_{L}}}\sum_
{\delta \in L^{*}/L}\mathbf{e}(-(\gamma,\delta))e_{\delta}\] (see, for example,
Section 4 of \cite{[B1]} or Section 2 of \cite{[B2]}). The number
$\zeta_{8}^{b_{+}-b_{-}}$ is the \emph{Weil index} of $V$ arising from its
quadratic structure and the character $t\mapsto\mathbf{e}(t)$ of $\mathbb{R}$,
as well as of the discriminant form $L^{*}/L$. This representation factors
through a double cover of $SL_{2}(\mathbb{Z}/N\mathbb{Z})$ for some integer $N$
called the level of $L$, or through $SL_{2}(\mathbb{Z}/N\mathbb{Z})$ itself if
the signature of $L$ is even. The action of a general element of
$Mp_{2}(\mathbb{Z})$ via $\rho_{L}$ is given in in \cite{[Sche]} for even
signature and \cite{[Str]} or \cite{[Ze1]} for the general case.

Let $K$ be the stabilizer of $i\in\mathcal{H}$ in $Mp_{2}(\mathbb{R})$. It is a
maximal compact subgroup of $Mp_{2}(\mathbb{R})$. It consists of the elements
$k_{\theta}=\big(\binom{\cos\theta\ \ -\sin\theta}{\sin\theta\ \ \
\cos\theta},\varphi\big)$, where $\varphi(i)=e^{i\theta/2}$, for
$\theta\in\mathbb{R}/4\pi\mathbb{Z}$. Some properties of the functions $F$ and
$\Theta_{L}$, which will turn out useful in this paper, are summarized in the
following

\begin{thm}
Assume that the polynomial $p$ is homogenous of degree
$(m_{+},m_{-})$ with respect to $v$.
\begin{enumerate}[$(i)$]
\item For any $M \in Mp_{2}(\mathbb{R})$ we have the equality
\[\rho_{V}(M)F(\tau,v,p,\lambda)=j(M,\tau)^{-\frac{b_{+}}{2}-m_{+}}\overline{j(M
,\tau)}^{-\frac{b_{-}}{2}-m_{-}}F(M\tau,v,p,\lambda)\] of functions of $\lambda
\in V$.
\item The action of $\rho_{V}(k_{\theta})$ multiplies $F(i,v,p,\lambda)$ by
$(e^{i\theta})^{\frac{b_{-}}{2}+m_{-}-\frac{b_{+}}{2}-m_{+}}$.
\item If $M \in Mp_{2}(\mathbb{Z})$ then the equality
\[\rho_{L}(M)\Theta_{L}(\tau,v,p)=j(M,\tau)^{-\frac{b_{+}}{2}-m_{+}}\overline{
j(M,\tau)}^{-\frac{b_{-}}{2}-m_{-}}\Theta_{L}(M\tau,v,p)\] holds in
$\mathbb{C}[L^{*}/L]$.
\end{enumerate}\label{modularity}
\end{thm}

\begin{proof}
Part $(i)$ is essentially Lemma 1.2 of \cite{[Sn]} (in fact, only the case
$b_{+}=2$ and certain harmonic polynomials of homogeneity degree $(m_{+},0)$ are
considered there, but using the results of Section 3 of \cite{[B1]} the proof
extends to the general case). Part $(ii)$ is the special case in which we take
$M=k_{\theta} \in K$ and $\tau=i$ in part $(i)$. Part $(iii)$ is the case
$\alpha=\beta=0$ in Theorem 4.1 of \cite{[B1]}. This proves the theorem.
\end{proof}

\medskip

Given real $r$ and $s$, let $\widetilde{\Delta}_{r,s}$ denote the operator
\[y^{2}(\partial_{x}^{2}+\partial_{y}^{2})-iy(r-s)\partial_{x}+y(r+s)\partial_{y
}=4y^{2}\partial_{\tau}\partial_{\overline{\tau}}-2iry\partial_{\overline{\tau}}
+2is\partial_{\tau}\] on functions on $\mathcal{H}$. The operator
$\Delta_{r,s}=\widetilde{\Delta}_{r,s}+(r-1)s$ is the \emph{weight $(r,s)$
Laplacian on $\mathcal{H}$}, and the equality
$\Delta_{r,s}y^{t}=y^{t}\Delta_{r+t,s+t}$ holds for any $t$. The \emph{weight
$r$ Laplacian} is just $\Delta_{r}=\Delta_{r,0}$. On the other hand, let
$\Delta_{SO_{b_{+},b_{-}}^{+}}$ denote the Laplacian of $O(V)$, i.e., the
Casimir element of the universal enveloping algebra of
$\mathfrak{so}_{b_{+},b_{-}}$. We shall need the following
\begin{prop}
Let $\tau\in\mathcal{H}$, $\lambda \in V$, and $v \in G(V)$ be as above, and
let $p$ be as in Theorem \ref{modularity}. Denote
$\frac{b_{+}}{2}+m_{+}-\frac{b_{-}}{2}-m_{-}$ by $k$. Then we have the equality
\[\Delta_{SO_{b_{+},b_{-}}^{+},\lambda}y^{\frac{b_{-}}{2}+m_{-}}F(\tau,v,p,
\lambda)=4\Delta_{k,\tau}y^{\frac{b_{-}}{2}+m_{-}}F(\tau,v,p,\lambda)+\]
\[+[(m_{+}-m_{-})(m_{+}-m_{-}+b_{+}-b_{-}-2)-(b_{+}-2)b_{-}]y^{\frac{b_{-}}{2}
+m_{-}}F(\tau,v,p,\lambda).\] \label{Deltalambdatau}
\end{prop}

\begin{proof}
Given $\tau\in\mathcal{H}$ define $g_{\tau}=\big(\frac{1}{\sqrt{y}}\binom{y\ \
x}{\ \ \ 1},\varphi\equiv+\frac{1}{\sqrt[4]{y}}\big) \in Mp_{2}(\mathbb{R})$.
Part $(i)$ of Theorem \ref{modularity} yields
$\rho_{V}(g_{\tau})F(i,v,p,\lambda)=y^{\frac{b_{+}}{4}+\frac{m_{+}}{2}+\frac{b_{
-}}{4}+\frac{m_{-}}{2}}F(\tau,v,p,\lambda)$. On the other hand, after extending
the Weil representation $\rho_{V}$ to the Lie algebra
$\mathfrak{sl}_{2}(\mathbb{R})$ of $Mp_{2}(\mathbb{R})$ and then to its
universal enveloping algebra, Lemma 1.4 of \cite{[Sn]} (with $m=2k$) and part
$(ii)$ of Theorem \ref{modularity} show that
\begin{equation}
\rho_{V}(C)y^{\frac{k}{2}+\frac{b_{-}}{2}+m_{-}}F(\tau,v,p,\lambda)=4\widetilde{
\Delta}_{\frac{k}{2},-\frac{k}{2},\tau}y^{\frac{k}{2}+\frac{b_{-}}{2}+m_{-}}
F(\tau,v,p,\lambda), \label{Sh14}
\end{equation}
where $C=2EF+2FE+H^{2}$ is the Casimir element of the universal enveloping
algebra of $\mathfrak{sl}_{2}(\mathbb{R})$. Now, the right hand side of equation
\eqref{Sh14} can be written as
$y^{\frac{k}{2}}\big(4\Delta_{k,\tau}+k(k-2)\big)y^{\frac{b_{-}}{2}+m_{-}}F(\tau
,v,p,\lambda)$. Moreover, Lemma 1.5 of \cite{[Sn]} shows that the action of
$\rho_{V}(C)$ coincides with that of
$\Delta_{SO_{b_{+},b_{-}}^{+}}+\frac{(b_{+}+b_{-})(b_{+}+b_{-}-4)}{4}$.
Substituting the value of $k$ in the constant completes the proof of the
proposition.
\end{proof}

\medskip

In the case $p=1$, Proposition \ref{Deltalambdatau} suffices to prove
Proposition 4.5 of \cite{[Bru]}. Indeed, in this case the action of
$SO_{b_{+},b_{-}}^{+}$ in the $\lambda$ variable coincides with its action on
the $v$ variable. Thus the actions of the Laplacians coming from these two
operations of $SO_{b_{+},b_{-}}^{+}$ coincide, so that one only needs to find
the form of $\Delta_{SO_{b_{+},b_{-}}^{+}}$ in the variable of the Grassmannian
$G(V)$. Here we consider the case of a non-trivial polynomial. There is thus no
evident connection between these two actions (and Laplacians). Indeed, $p$
depends on $v$ by the homogeneity condition, and this dependence affects
strongly the form which $\Delta_{SO_{b_{+},b_{-}}^{+},v}$ attains. However, we
now present a particular case in which the actions can be related.

\subsection{Grassmannians with Complex Structures}

We take $b_{+}=2$, where the Grassmannian carries the structure of a complex
manifold, which we now briefly describe (see Section 13 of \cite{[B1]} or
Section 3.2 of \cite{[Bru]} for more details). Indeed, in this case $G(V)$ is
diffeomorphic to a connected component of the analytically open subset which is
defined by the inequality $(Z_{V},\overline{Z_{V}})>0$ on the conic
$\mathbb{P}(V_{\mathbb{C}})_{0}=\{[Z_{V}] \in
\mathbb{P}(V_{\mathbb{C}})|Z_{V}^{2}=0\}$ (here $\mathbb{P}$ stands for the
associated projective space). The inverse image $P$ of this variety in
$V_{\mathbb{C}}$ corresponds, by taking the real and imaginary parts of a
complex vector, to the set of oriented, orthogonal pair of vectors of the same
positive norm. Such a pair of vectors forms a basis for $v_{+}$ in the image of
that point in $G(V)$. We may view $P$ as a $\mathbb{C}^{*}$-bundle over $G(V)$.

Given an indefinite vector space $V$, pick an isotropic vector $z \in V$. Then
$K_{\mathbb{R}}=z^{\perp}/\mathbb{R}z$ is non-degenerate of signature
$(b_{+}-1,b_{-}-1)$. By choosing $\zeta \in V$ with $(z,\zeta)=1$, we identify
$K_{\mathbb{R}}$ with $\{z,\zeta\}^{\perp}$. Vectors of $V$ may then be written
as triplets $(\eta,a,b)$, with $\eta \in K_{\mathbb{R}}$ and $a$ and $b$ in
$\mathbb{R}$. This symbol stands for the vector $\eta+a\zeta+bz$ (under the
identification of $K_{\mathbb{R}}$ with $\{z,\zeta\}^{\perp}$), which has norm
$\eta^{2}+a^{2}\zeta^{2}+2ab$. If $b_{+}=2$ then the choice of $z$ defines a
holomorphic section $G(V) \to P$ by picking the element pairing to 1 with $z$.
Moreover, $K_{\mathbb{R}}$ is then Lorentzian, so that the set of positive norm
vectors is the disjoint union of two cones. The choice of $z$ and of the
orientation on the bases determine one cone $C$ to be the positive one. Our
section then yields a biholomorphic diffeomorphism between $G(V)$ and the tube
domain $K_{\mathbb{R}}+iC$. Under this identification, the section maps $Z=X+iY
\in K_{\mathbb{R}}+iC$ to the element
$Z_{v,V}=\big(Z,1,\frac{-Z^{2}-\zeta^{2}}{2}\big)$ of $P$ (the inverse map is
subtracting $\zeta$ and projecting to $K_{\mathbb{C}}$). The real and imaginary
parts of this vector are $X_{v,V}=\big(X,1,\frac{Y^{2}-X^{2}-\zeta^{2}}{2}\big)$
and $Y_{v,V}=\big(Y,0,-(X,Y)\big)$ respectively. We have $X_{v,V} \perp Y_{v,V}$
and $X_{v,V}^{2}=X_{v,V}^{2}=Y^{2}>0$. Here $v$ stands for the element of
$G(L_{\mathbb{R}})$ in which $v_{+}$ is spanned by the latter two vectors.

\medskip

For non-negative integers $r$, $s$, and $t$, let
$P_{r,s,t}(Z,\lambda)=\frac{(\lambda,Z_{v,V})^{r}(\lambda,\overline{Z_{v,V}})^{t
}}{(Y^{2})^{s}}$. It is homogenous of degree $(r+t,0)$ with respect to $v$.
Hence the constant in Proposition \ref{Deltalambdatau} reduces, for
$p=P_{r,s,t}$, to $(r+t)(r+t-b_{-})$. Given an even lattice $L$ in $V$, we
denote the theta function $\Theta_{L}(\tau,v,P_{r,s,t})$ by
$\Theta_{L,r,s,t}(\tau,v)$.

Let $O^{+}(V)$ be the subgroup of $O(V)$ in which the orientation on the
positive definite part is preserved. It operates on $G(V) \cong
K_{\mathbb{R}}+iC$, with a factor of automorphy which we denote $j(\sigma,Z)$
for $\sigma \in O^{+}(V)$ and $Z \in K_{\mathbb{R}}+iC$. It is defined by the
equation \[\sigma(Z_{v,V})=j(\sigma,Z)Z_{\sigma v,V},\quad\mathrm{or\
equivalently}\quad j(\sigma,Z)=\big(\sigma(Z_{v,V}),z\big).\] Given integers $k$
and $l$ and a discrete subgroup $\Gamma$ of $O^{+}(V)$, we say that a function
$\Phi:K_{\mathbb{R}}+iC\to\mathbb{C}$ (or on $G(V)$) is an \emph{automorphic
form of weight $(k,l)$ with respect to $\Gamma$} if the equality \[\Phi(\sigma
Z)=j(\sigma,Z)^{k}\overline{j(\sigma,Z)}^{l}\Phi(Z)\] holds for every $Z \in
K_{\mathbb{R}}+iC$ and $\sigma\in\Gamma$. A standard argument shows that this is
equivalent to the definition of automorphy appearing in \cite{[B1]}, using
homogeneity and $\Gamma$-invariance on $P$. We now prove
\begin{prop}
For fixed $\tau\in\mathcal{H}$, the function $v\mapsto\Theta_{L,r,s,t}(\tau,v)$
is automorphic of weight $(s-r,s-t)$ with respect to the discriminant group
$\Gamma$ of $L$, namely the kernel of the natural map from $Aut^{+}(L)=Aut(L)
\cap O^{+}(L)$ to $Aut(L^{*}/L)$. \label{modOVTheta}
\end{prop}

\begin{proof}
The definition of $j(\sigma,Z)$ and the equality $\big(\Im(\sigma
Z)\big)^{2}=\frac{Y^{2}}{|j(\sigma,Z)|^{2}}$ from (the simple) Lemma 3.20 of
\cite{[Bru]} imply the equality
\begin{equation}
P_{r,s,t}(\sigma
Z,\lambda)=j(\sigma,Z)^{s-r}\overline{j(\sigma,Z)}^{s-t}P_{r,s,t}(Z,\sigma^{-1}
\lambda) \label{modOVP}
\end{equation}
for all $Z \in K_{\mathbb{R}}+iC$, $\lambda \in V$, and $\sigma \in O^{+}(V)$.
Now, $\Delta_{v}P_{r,s,t}=\Delta_{v_{+}}P_{r,s,t}$ is just $4rtP_{r-1,s-1,t-1}$.
As this operation preserves the differences $s-r$ and $s-t$, Equation
\eqref{modOVP} continues to hold if we replace $P_{r,s,t}$ by
$e^{-\Delta_{v}/8\pi y}P_{r,s,t}$. By the usual automorphic properties of the
expression
$\mathbf{e}\Big(\tau\frac{\lambda_{v_{+}}^{2}}{2}+\overline{\tau}\frac{\lambda_{
v_{-}}^{2}}{2}\Big)$, the validity of Equation \eqref{modOVP} extends also to
$F_{r,s,t}(\tau,Z,\lambda)=F(\tau,v,P_{r,s,t},\lambda)$. Now, as the operation
of $\sigma\in\Gamma$ preserves absolutely convergent sums on cosets of $L$
inside $L^{*}$, this completes the proof of the proposition.
\end{proof}

\medskip

For $\sigma \in O^{+}(V)$, let $[\sigma]_{s-r,s-t}$ denote the \emph{slash
operator} of weight $(k,l)$:
\[\Phi[\sigma]_{k,l}(Z)=\Phi(\sigma Z)j(\sigma,Z)^{-k}\overline{j(\sigma,Z)}^{\
-l}.\] Equation \eqref{modOVP} relates the usual action of $O^{+}(V)$ on the
$\lambda$ variable in $P_{r,s,t}$ or $F_{r,s,t}$ to its action on $Z$ via the
slash operators of weight $(s-r,s-t)$. As the operator
$\Delta_{SO_{b_{+},b_{-}}^{+},\lambda}$ has order 2 and it commutes with the
action of $SO_{b_{+},b_{-}}^{+}$ on $\lambda$, its action must be related to an
operator (of order 2) in the $Z$ variable which commutes with the action of
these slash operators. The standard theory of Laplacians and Casimir operators
in simple Lie groups shows that there should be only one such second order
differential operator (up to multiplicative and additive constants). Let
$\Delta^{G}$ be the operator which in an orthonormal basis for $K_{\mathbb{R}}$
takes the form
\[8\sum_{g,h}y_{g}y_{h}\partial_{g}\overline{\partial}_{h}-4Y^{2}\bigg(\partial_
{1}\overline{\partial}_{1}-\sum_{k>1}\partial_{k}\overline{\partial}_{k}\bigg)\]
(this is 8 times the operator $\Delta_{1}$ of \cite{[Na]} or $8\Omega$ in the
notation of \cite{[Bru]}), and define $D^{*}=\sum_{h}y_{h}\partial_{h}$ and
$\overline{D^{*}}=\sum_{h}y_{h}\overline{\partial}_{h}$. The differential
operator we are looking for is given in the following

\begin{lem}
The combination
\[\widetilde{\Delta}^{G}_{k,l}=\Delta^{G}-4ik\overline{D^{*}}+4ilD^{*}\]
commutes with the slash operators $[\sigma]_{k,l}$ for every $\sigma \in
O^{+}(V)$.
\label{DeltaGkl}
\end{lem}

\begin{proof}
\cite{[Na]} has shown that $O^{+}(V)$ is generated by the elements $p_{\xi}$ for
$\xi \in K_{\mathbb{R}}$, $k_{a,A}$ for $a\in\mathbb{R}^{*}$ and $A \in
O^{sgn(a)}(K_{\mathbb{R}})$, and $w$. In the
$K_{\mathbb{R}}+\mathbb{R}\zeta+\mathbb{R}z$ notation, under the assumption
$\zeta^{2}=0$ (which is made in \cite{[Na]} and can always be satisfied by
replacing $\zeta$ by $\zeta-\frac{\zeta^{2}}{2}z$), these elements take the form
\[p_{\xi}=\left(\begin{array}{ccc}I & 0 & -\xi^{*} \\ \xi & 1 &
-\frac{\xi^{2}}{2} \\ 0 & 0 & 1\end{array}\right),\quad
k_{a,A}=\left(\begin{array}{ccc}A & 0 & 0 \\ 0 & \frac{1}{a} & 0 \\ 0 & 0 &
a\end{array}\right),\quad w=\left(\begin{array}{ccc}\widetilde{w} & 0 & 0 \\ 0
& 0 & -1 \\ 0 & -1 & 0\end{array}\right),\] where
$\xi^{*}:K_{\mathbb{R}}\to\mathbb{R}$ is defined by pairing with $\xi$ and
$\widetilde{w}$ is the reflection with respect to the hyperplane perpendicular
to a pre-fixed positive norm vector in $K_{\mathbb{R}}$. The sign condition on
$A$ ensures that $k_{a,A} \in O^{+}(V)$, and $w$ lies in $SO^{+}(V)$. Hence it
suffices to verify the commutativity of $\widetilde{\Delta}^{G}_{k,l}$ with
$[\sigma]_{k,l}$ for $\sigma$ being one of the elements $p_{\xi}$, $k_{a,A}$, or
$w$.

For $\sigma=p_{\xi}$ and for $\sigma=k_{a,A}$ the automorphy factor
$j(\sigma,Z)$ is a constant function of $Z$ (1 for $p_{\xi}$, $\frac{1}{a}$ for
$k_{a,A}$). Hence the assertion follows easily from the fact that the three
operators $\Delta^{G}$, $D^{*}$, and $\overline{D^{*}}$ are invariant under such
$\sigma$. For $\sigma=w$, the action of the part $-4ik\overline{D^{*}}+4ilD^{*}$
of $\widetilde{\Delta}^{G}_{k,l}$ on $F[\sigma]_{k,l}$ yields
\[-4ik\big[\overline{D^{*}}(F \circ w)\overline{j}^{\ -l}+(F \circ
w)\overline{D^{*}}\overline{j}^{\ -l}\big]j^{-k}+4il\big[D^{*}(F \circ
w)j^{-k}+(F \circ w)D^{*}j^{-k}\big]\overline{j}^{\ -l}.\] On the other hand,
$\Delta^{G}$ is the sum of two operators, so that apart from the expression
$\Delta^{G}(F \circ w)j^{-k}\overline{j}^{\ -l}$, the combination
$\Delta^{G}\big(F[\sigma]_{k,l}\big)$ involves
\[8\overline{D^{*}}(F \circ w)D^{*}j^{-k}\overline{j}^{\ -l}+8D^{*}(F \circ
w)j^{-k}\overline{D^{*}}\overline{j}^{\ -l}+8(F \circ
w)D^{*}j^{-k}\overline{D^{*}}\overline{j}^{\ -l}\] from the action of the
operator $\sum_{g,h}y_{j}y_{h}\partial_{g}\overline{\partial}_{h}$ and similar
three expressions from the action of the other operator. The automorphy factor
$j(w,Z)$ is $\frac{Z^{2}}{2}$. Evaluating $D^{*}j^{-k}$, $\overline{D^{*}}\
\overline{j}^{\ -l}$, and the other derivatives of $j^{k}$ and $\overline{j}^{\
-l}$ thus shows that $\widetilde{\Delta}^{G}_{k,l}\big(F[\sigma]_{k,l}\big)$
equals \[\Delta^{G}(F \circ w)j^{-k}\overline{j}^{\
-l}-2ik\overline{Z}^{2}\overline{D^{*}}(F \circ w)j^{-k-1}\overline{j}^{\
-l}+2ilZ^{2}D^{*}(F \circ w)j^{-k}\overline{j}^{\ -l-1}+\]
\[+4kY^{2}\overline{D}(F \circ w)j^{-k-1}\overline{j}^{\ -l}+4lY^{2}D(F \circ
w)j^{-k}\overline{j}^{\ -l-1}\] (and the coefficients of
$F(wZ)j^{-k-1}\overline{j}^{\ -l-1}$ cancel out). Now, the Theorem of
\cite{[Na]} shows that $\Delta^{G}(F \circ w)(Z)=(\Delta^{G}F)(wZ)$, and the
formulae concerning $D^{*}$ and $\overline{D^{*}}$ in \cite{[Na]} translate to
\[D^{*}(F \circ w)(Z)=\frac{W^{2}}{\overline{W}^{2}}(D^{*}F)(wZ)-2i\frac{(\Im
W)^{2}}{\overline{W}^{2}}(DF)(wZ)\] and \[\overline{D^{*}}(F \circ
w)(Z)=\frac{\overline{W}^{2}}{W^{2}}(\overline{D^{*}}F)(wZ)+2i\frac{(\Im
W)^{2}}{W^{2}}(\overline{D}F)(wZ),\] with $W=w(Z)$ (recall that the expressions
denoted $\delta$ and $d$ in \cite{[Na]} are $\frac{Z^{2}}{2}$ and
$\frac{Y^{2}}{2}$ respectively). One also verifies that $D(F \circ
w)(Z)=-(DF)(wZ)$ and
$\overline{D}(F \circ w)(Z)=-(\overline{D}F)(wZ)$, while
$\frac{W^{2}}{2}=\frac{2}{Z^{2}}$,
$\frac{\overline{W}^{2}}{2}=\frac{2}{\overline{Z}^{2}}$, and $(\Im
W)^{2}=\frac{4Y^{2}}{Z^{2}\overline{Z^{2}}}$. Therefore,
$\widetilde{\Delta}^{G}_{k,l}\big(F[\sigma]_{k,l}\big)(Z)$ equals
\[(\Delta^{G}F)(wZ)j^{-k}\overline{j}^{\ -l}-4ik\overline{D^{*}}F(wZ)j^{-k}
\overline{j}^{\ -l}+4ilD^{*}F(wZ)j^{-k}\overline{j}^{\ -l},\] which agrees with
the value $(\widetilde{\Delta}^{G}_{k,l}F)[\sigma]_{k,l}(Z)$. This proves the
lemma.
\end{proof}

\smallskip

The \emph{weight $(k,l)$ Laplacian on $G(V)$} is the operator
$\Delta^{G}_{k,l}=\widetilde{\Delta}^{G}_{k,l}-2(2k-b_{-})l$. Lemma 3.20 of
\cite{[Bru]} shows that multiplication by $(Y^{2})^{t}$ takes an automorphic
form of weight $(k+t,l+t)$ to an automorphic form of weight $(k,l)$. One also
verifies that the operators $\Delta^{G}_{k,l}$ satisfy the relation
$\Delta^{G}_{k,l}(Y^{2})^{t}=(Y^{2})^{t}\Delta^{G}_{k+t,l+t}$.

We remark that the results of \cite{[Na]} are stated for $b_{-}$ (or $q$ in the
notation of \cite{[Na]}) being at least 3. However, the proof holds equally well
for $q=2$, and the same applies for our Lemma \ref{DeltaGkl}. For $b_{-}=1$ the
Grassmannian $K_{\mathbb{R}}+iC$ is $\mathcal{H}$, the operator $\Delta^{G}$ is
the usual Laplacian $4y^{2}\partial\overline{\partial}$ on $\mathcal{H}$, and
$\Delta^{G}_{k,l}$ is $\Delta_{2k,2l}$. Therefore, Lemma \ref{DeltaGkl} holds
for any value of $b_{-}$.

\subsection{A Differential Equation for $\Theta_{L,r,s,t}$}

The following generalization of Proposition 4.5 of \cite{[Bru]} will turn out
very important for our purposes:

\begin{prop}
Let $L$ be an even lattice in the space $V$ of signature $(2,b_{-})$, and let
$k=1+\frac{b_{-}}{2}+r+t$. Then the theta function $\Theta_{L,r,s,t}$ satisfies
the differential equation
\[4\Delta_{k,\tau}y^{\frac{b_{-}}{2}}\Theta_{L,r,s,t}(\tau,Z)=\big[\Delta^{G}_{
s-r,s-t,Z}+2r(b_{-}-2t)\big]y^{\frac{b_{-}}{2}}\Theta_{L,r,s,t}(\tau,Z).\]
\label{DeltatauZ}
\end{prop}

\begin{proof}
We follow the proof of Proposition 4.5 of \cite{[Bru]}, with the necessary
adjustments. It suffices to prove that the basic summand
$F_{r,s,t}(\tau,Z,\lambda)$ satisfies this differential equation for any
$\lambda$. Let $\rho_{k,l}$ be the representation of $O^{+}(V)$ on
$C^{\infty}(G(V))$ using the weight $(k,l)$ slash operators, namely
$\rho_{k,l}(\sigma)\Phi=\Phi[\sigma^{-1}]_{k,l}$, and extend it to the universal enveloping algebra of $\mathfrak{so}(V)$. Lemma \ref{DeltaGkl} shows that the
action of the Casimir operator of $O^{+}(V)$ via $\rho_{k,l}$ must be the same
as that of $\alpha\Delta^{G}_{k,l}+\beta$ for some constants $\alpha$ and
$\beta$. Equation \eqref{modOVP} and the paragraph following it imply the
equality
\[F_{r,s,t}(\tau,Z,\sigma^{-1}\lambda)=\rho_{s-r,s-t}(\sigma)F_{r,s,t}(\tau,Z,
\lambda)\] for every $\tau\in\mathcal{H}$, $Z \in K_{\mathbb{R}}+iC$, $\lambda
\in V$, and $\sigma \in O^{+}(V)$. Since the Casimir operator of $O^{+}(V)$ acts
on functions of $\lambda$ as $\Delta_{SO_{2,b_{-}}^{+}}$, we obtain
\[\Delta_{SO_{2,b_{-}}^{+},\lambda}F_{r,s,t}(\tau,Z,\lambda)=\pm(\alpha\Delta^{G
}_{s-r,s-t}+\beta)F_{r,s,t}(\tau,Z,\lambda).\] Proposition \ref{Deltalambdatau}
now yields
\[4\Delta_{k,\tau}y^{\frac{b_{-}}{2}}F_{r,s,t}(\tau,Z,\lambda)=(\widetilde{
\alpha}\Delta^{G}_{s-r,s-t}+\widetilde{\beta})y^{\frac{b_{-}}{2}}F_{r,s,t}(\tau,
Z,\lambda)\] for some constants $\widetilde{\alpha}$ and $\widetilde{\beta}$
which are independent of $\lambda$. Choose a basis for $K_{\mathbb{R}}$ in which
the first two basis elements span a hyperbolic plane and the rest are
orthonormal (or with common norm $-2$, as in \cite{[Bru]} and \cite{[Na]}), and
take $\lambda$ to be the second basis element. Then evaluating
$e^{-\Delta_{v}/8\pi y}(P_{r,s,t})$ shows that
\[F_{r,s,t}(\tau,Z,\lambda)=\sum_{j=0}^{\min\{r,t\}}\frac{j!}{(-2\pi)^{j}}\binom
{r}{j}\binom{t}{j}\frac{z_{1}^{r-j}\overline{z_{1}}^{t-j}}{(Y^{2})^{s-j}}e^{
-2\pi y\frac{|z_{1}|^{2}}{Y^{2}}},\] with $z_{1}$ being the first coordinate of
$Z$ in this basis. A straightforward computation shows that
$\widetilde{\alpha}=1$ and $\widetilde{\beta}=2r(b_{-}-2t)$. This proves the
proposition.
\end{proof}

We remark that all these results extend to the case where $P_{r,s,t}$ is
multiplied by $(\lambda_{v_{-}}^{2})^{h}$ for some integer $h$ (powers of
$\lambda_{v_{+}}^{2}$ just change the indices of $P_{r,s,t}$, hence produce no
new functions). Then the weight is $k=1+\frac{b_{-}}{2}+r+t-2h$, and the
constant $\widetilde{\beta}$ from Proposition \ref{DeltatauZ} becomes
$2(r-h)\big(b_{-}-2(t-h)\big)$. However, we shall not need this generalization
in this paper.

\medskip

In fact, as we are concerned with explicit functions and operators, Equation
\eqref{Sh14} and Propositions \ref{Deltalambdatau} and \ref{DeltatauZ} can be
obtained by direct evaluation of the corresponding derivatives. In addition, it
is expedient to evaluate the images of
$y^{\frac{b_{-}}{2}+m_{-}}F(\tau,v,p,\lambda)$ under the \emph{weight raising
operator} $R_{k}=2i\frac{\partial}{\partial\tau}+\frac{k}{y}$ and the
\emph{weight lowering operator}
$L=-2iy^{2}\frac{\partial}{\partial\overline{\tau}}$ (note
the sign difference relative to \cite{[Bru]}!) in general. The resulting
functions are
\begin{subequations}
\begin{align}
-2\pi y^{\frac{b_{-}}{2}+m_{-}}F(\tau,v,\lambda_{v_{+}}^{2}p,\lambda)-\frac{1}
{8\pi}y^{\frac{b_{-}}{2}+m_{-}-2}F(\tau,v,\Delta_{v_{-}}p,\lambda) &
\label{RkTheta} \\ \intertext{and} 2\pi
y^{2+\frac{b_{-}}{2}+m_{-}}F(\tau,v,\lambda_{v_{-}}^{2}p,\lambda)+\frac{1}{8\pi}
y^{\frac{b_{-}}{2}+m_{-}}F(\tau,v,\Delta_{v_{+}}p,\lambda) & \label{LTheta}
\end{align}
\end{subequations}
respectively. Thus similar formulae describe the action of these operators on
arbitrary theta functions with polynomials. One verifies that the action of
$\Delta_{SO_{b_{+},b_{-}}^{+},\lambda}$ on
$y^{\frac{b_{-}}{2}+m_{-}}F(\tau,v,p,\lambda)$ takes it to
$y^{\frac{b_{-}}{2}+m_{-}}q(\lambda)\mathbf{e}\big(\tau\frac{\lambda_{v_{+}}^{2}
}{2}+\overline{\tau}\frac{\lambda_{v_{-}}^{2}}{2}\big)$ with $q$ being the image
of $e^{-\Delta_{v}/8\pi y}(p)$ under the operator
\begin{equation}
\Delta_{SO_{b_{+},b_{-}}^{+}}+8\pi y\lambda_{v_{-}}^{2}I_{v_{+}}-8\pi
y\lambda_{v_{+}}^{2}I_{v_{-}}+4\pi yb_{+}\lambda_{v_{-}}^{2}-4\pi
yb_{-}\lambda_{v_{+}}^{2}-16\pi^{2}y^{2}\lambda_{v_{+}}^{2}\lambda_{v_{-}}^{2}.
\label{direv}
\end{equation}
The symbol $I_{U}$, for a vector space $U$, stands for the \emph{homogeneity
operator} $\sum_{j}u_{j}\frac{\partial}{\partial u_{j}}$ in some (hence any)
basis $(u_{j})_{j=1}^{\dim U}$ of $U$. Evaluating the action of
$\Delta_{k}=R_{k-2} \cdot L$ for $k=\frac{b_{+}}{2}+m_{+}-\frac{b_{-}}{2}-m_{-}$
on $y^{\frac{b_{-}}{2}+m_{-}}F(\tau,v,p,\lambda)$ yields the function described
in Equation \eqref{direv} up to the constant multiple of
$y^{\frac{b_{-}}{2}+m_{-}}F(\tau,v,p,\lambda)$ appearing in Proposition
\ref{Deltalambdatau}. In the case $b_{+}=2$ and the polynomial $P_{r,s,t}$, one
can check directly that the operator from Equation \eqref{direv} takes
$P_{r,s,t}$ to
\begin{equation}
-4\lambda_{-}^{2}a_{r,t}P_{r-1,s-1,t-1}-4\pi
yb_{-}P_{r+1,s+1,t+1}+2tb_{-}P_{r,s,t}, \label{dirPrst}
\end{equation}
where $a_{r,t}$ stands for \[(r-2\pi y\lambda_{v_{+}}^{2})(t-2\pi
y\lambda_{v_{+}}^{2})-2\pi y\lambda_{v_{+}}^{2}=rt-(r+t+1)2\pi
y\lambda_{v_{+}}^{2}+(2\pi y\lambda_{v_{+}}^{2})^{2}.\] A direct evaluation of
the action of the operator $\Delta^{G}_{s-r,s-t}+(r-t)^{2}+b_{-}(r-t)$ on
$P_{r,s,t}(Z,\lambda)\mathbf{e}\Big(\tau\frac{\lambda_{v_{+}}^{2}
}{2}+\overline{\tau}\frac{\lambda_{v_{-}}^{2}}{2}\Big)$ gives the expression
from Equation \eqref{dirPrst} multiplied by
$\mathbf{e}\Big(\tau\frac{\lambda_{v_{+}}^{2}}{2}+\overline{\tau}\frac{\lambda_{
v_{-}}^{2}}{2}\Big)$. Recall that $\Delta_{v_{+}}^{j}P_{r,s,t}$ is some constant
multiple of $P_{r-j,s-j,t-j}$ for any $j\leq\min\{r,t\}$, and the difference
$(r-j)-(t-j)$ equals $r-t$. It follows that
$\Delta^{G}_{s-r,s-t}+(r-t)^{2}+b_{-}(r-t)$ and
$\Delta_{SO_{b_{+},b_{-}}^{+},\lambda}$ give the same result also on
$F_{r,s,t}$. This argument suggests an alternative proof of Equation
\eqref{Sh14} and of Propositions \ref{Deltalambdatau} and \ref{DeltatauZ} (as
well as of Proposition 4.5 of \cite{[Bru]}), a proof which is independent of
Theorem \ref{modularity}, the results of \cite{[Sn]}, and theorems about actions
of Casimir operators and Laplacians in general.

\section{Theta Lifts with Polynomials \label{Lift}}

In this Section we evaluate the theta lifts of certain almost weakly holomorphic
modular forms, which will give the main tool for the arithmetic application
later (see Definitions \ref{pringen} and \ref{prinmer} below). Many parts of
this Section are very technical, and skipping most of it except the statements
of Theorem \ref{PhiLmm0} for general $b_{-}$ and Theorem \ref{b-=1mer} for
$b_{-}=1$ may suffice on the first reading.

\subsection{Modular Forms and their Theta Lifts}

Let $\Gamma \subseteq Mp_{2}(\mathbb{R})$ be a Fuchsian group of the first kind,
and let $\rho$ be a representation of $\Gamma$ on some finite-dimensional
complex vector space $V_{\rho}$. Given $k$ and $l$ in $\frac{1}{2}\mathbb{Z}$,
a \emph{modular form of weight $(k,l)$ and representation $\rho$ with respect to
$\Gamma$} is a real-analytic function $f:\mathcal{H} \to V_{\rho}$ which
satisfies the functional equation
\[f(M\tau)=j(M,\tau)^{k}\overline{j(M,\tau)}^{l}\rho(M)f(\tau)\] for every
$\tau\in\mathcal{H}$ and $M\in\Gamma$ (the metaplectic data makes the
half-integral powers well-defined). For example, part $(iii)$ of Theorem
\ref{modularity} states that $\Theta_{L}$ is modular of weight
$\big(\frac{b_{+}}{2}+m_{+},\frac{b_{-}}{2}+m_{-}\big)$ and representation
$\rho_{L}$ with respect to $Mp_{2}(\mathbb{Z})$. A function $f:\mathcal{H} \to
V_{\rho}$ is called \emph{almost holomorphic} if it takes the form
$f(\tau)=\sum_{k=0}^{k_{max}}\frac{f_{k}(\tau)}{y^{k}}$, for some integer
$k_{max}$, with the $f_{k}$ holomorphic functions. This condition is stable
under the action of $Mp_{2}(\mathbb{R})$. In case $\Gamma$ has cusps, a
\emph{weakly holomorphic} modular form is a modular form which is holomorphic on
$\mathcal{H}$ but may have poles at the cusps. Relaxing the requirement of
holomorphicity on $\mathcal{H}$ to almost holomorphicity yields \emph{almost
weakly holomorphic} modular forms. The weight raising operator $R_{k}$
increases the weight of a modular form of weight $k$ by 2, while the weight
lowering operator $L$ reduces it by 2. Both preserve almost holomorphicity and
almost weakly holomorphicity, with the latter operator annihilating (weakly)
holomorphic modular forms.

A modular form $f$ of any weight $(k,l)$ and representation $\rho_{L}$ with
respect to $Mp_{2}(\mathbb{Z})$ has a Fourier expansion
\begin{equation}
f(\tau)=\sum_{\gamma \in
L^{*}/L}\sum_{n\in\mathbb{Q}}c_{\gamma,n}(y)q^{n}e_{\gamma}, \label{Fourier}
\end{equation}
with the standard notation $q=\mathbf{e}(\tau)$. Here $c_{\gamma,n}$ are smooth
functions of $y$ which vanish unless $n\in\mathbb{Z}+\frac{\gamma^{2}}{2}$. If
$f$ is almost weakly holomorphic then the functions $c_{\gamma,n}$ are
polynomials of bounded degree in $\frac{1}{y}$, and if $f$ is weakly holomorphic
then they are constants. In both cases they vanish unless $n\gg-\infty$.

\medskip

Let $L$ be an even lattice of signature $(b_{+},b_{-})$. We take, for every $v
\in G(L_{\mathbb{R}})$, a polynomial $p_{v}$ on $L_{\mathbb{R}}$ which is
homogenous of degree $(m_{+},m_{-})$ with respect to $v$. Assume that $p_{v}$
depends smoothly on $v$. Let $F$ be a (not necessarily holomorphic) modular form
of weight $\frac{b_{+}}{2}+m_{+}-\frac{b_{-}}{2}-m_{-}$ and representation
$\rho_{L}$. Considering $\mathbb{C}[L^{*}/L]$ as a unitary space in which the
canonical basis is orthonormal, the function $\tau \mapsto
y^{\frac{b_{+}}{2}+m_{+}}\langle F(\tau),\Theta_{L}(\tau,v,p_{v})
\rangle_{\rho_{L}}$ is $Mp_{2}(\mathbb{Z})$-invariant for every $v$. Following
\cite{[B1]} and others, we define the \emph{theta lift of $F$} as the integral
\begin{equation}
\Phi_{L}(v,F,p_{v})=\int_{X(1)}y^{\frac{b_{+}}{2}+m_{+}}\langle
F(\tau),\Theta_{L}(\tau,v,p_{v})\rangle_{\rho_{L}}\frac{dxdy}{y^{2}}
\label{liftp}
\end{equation}
as a function of $v \in G(L_{\mathbb{R}})$. If $F$ decreases exponentially
towards the cusp, then the integral in Equation \eqref{liftp} converges for
every such $v$, yielding a smooth function on $G(L_{\mathbb{R}})$ (depending on
the polynomial $p:v \mapsto p_{v}$). If $F$ grows exponentially toward the cusp,
then the integral in Equation \eqref{liftp} diverges, but can be regularized as
follows. Let
\[D=\big\{\tau\in\mathcal{H}\big||\Re\tau|\leq1/2,|\tau|\geq1\big\}\] be the
classical fundamental domain for $SL_{2}(\mathbb{Z})$, and let $D_{w}$ be the
compact domain $\{\tau \in D|y \leq w\}$ for every $w>1$. Multiply the integrand
from Equation \eqref{liftp} by $y^{-s}$, carry out the integration over
$D_{w}$, and take the limit as $w\to\infty$. For $\lambda \in L^{*}$ with
$\lambda^{2}<0$ we define the sub-Grassmannian
\[\lambda^{\perp}=\{v \in G(L_{\mathbb{R}})|\lambda \in v_{-}\} \subseteq
G(L_{\mathbb{R}})\] of $G(L_{\mathbb{R}})$. It has real codimension $b_{+}$ in
$G(L_{\mathbb{R}})$, and in the case $b_{+}=2$ it has complex codimension 1
there. The results of Sections 6 and 7 of \cite{[B1]} give the following
\begin{thm}
Assume that the coefficients $c_{\gamma,n}(y)$ in Equation \eqref{Fourier}
satisfy the condition that for any two real numbers $\alpha>0$ and $\beta\geq0$,
the integral $\int_{0}^{\infty}c_{\gamma,n}(y)e^{-\alpha/y-\beta y}y^{-t}dy$
converges for $t\in\mathbb{C}$ with $\Re t\gg-\infty$ to give a function of $t$
which can be meromorphically continued to all $t\in\mathbb{C}$. Then the limit
from above produces, for any $v \in G(L_{\mathbb{R}})$, the restriction of a
meromorphic function of $s\in\mathbb{C}$ to some right half-plane in
$\mathbb{C}$. The constant term of this function at $s=0$ yields a function on
$G(L_{\mathbb{R}})$, which is smooth outside a locally finite union of the
sub-Grassmannians $\lambda^{\perp}$ defined above, and has singularities along
these sub-Grassmannians. \label{regthetalift}
\end{thm}

Note that a function is considered ``singular'' along a submanifold if it is not
smooth there, not only discontinuous.

Assume now that the modular form $F$ is almost weakly holomorphic, so that
Equation \eqref{Fourier} becomes \[F(\tau)=\sum_{\gamma \in
L^{*}/L}\sum_{k=0}^{k_{max}}\sum_{n\gg-\infty}c_{\gamma,n,k}\frac{q^{n}}{y^{k}}
e_{\gamma}\] for some integer $k_{max}$ with $c_{\gamma,n,k}\in\mathbb{C}$. The
polar part of $F$ is a finite sum, and we have
\begin{thm}
An almost weakly holomorphic modular form satisfies the condition of Theorem
\ref{regthetalift}. If such a modular form is expanded as above, then given $j$
and $k$ we define $\beta=\frac{b_{+}}{2}+m_{+}-j-k-1$. The singularity of the regularized theta lift $\Phi_{L}$ from Theorem \ref{regthetalift} along $\lambda^{\perp}$ then looks like
\[\sum_{\alpha\lambda \in
L^{*}}\sum_{j,k}\frac{c_{\alpha\lambda,\frac{\alpha^{2}\lambda^{2}}{2},k}\Delta_
{v}^{j}(\overline{p_{v}})(\alpha\lambda)}{(-8\pi)^{j}j!}\left\{\begin{array}{cc}
\displaystyle{\frac{\Gamma(\beta)}{(2\pi\alpha^{2}\lambda_{v_{+}}^{2})^{\beta}}}
& \beta\not\in-\mathbb{N} \\
\displaystyle{\frac{\ln(\alpha^{2}\lambda_{v_{+}}^{2})}{(-2\pi\alpha^{2}\lambda_
{v_{+}}^{2})^{\beta}(-\beta)!}} & \beta\in-\mathbb{N}.\end{array}\right.\]
\label{sing}
\end{thm}

The set $\mathbb{N}$ is assumed to include 0 in Theorem \ref{sing}.
\begin{proof}
The fact that $F$ satisfies the condition of Theorem \ref{regthetalift} follows
from Lemmas 7.2 and 7.3 of \cite{[B1]} (note the difference in conventions
arising from the fact that in Equation \eqref{Fourier} the coefficients
$c_{\gamma,n}(y)$ multiply $q^{n}e_{\gamma}$ while in \cite{[B1]} the notation
$c_{\gamma,n}(y)$ is used for the coefficient of $\mathbf{e}(nx)e_{\gamma}$).
The form of the singularity is now given in Theorem 6.2 of \cite{[B1]}. This proves the theorem.
\end{proof}
The sums in Theorem \ref{sing} are essentially finite: For $\alpha$ we consider
only coefficients from the (finite) polar part of $F$, for $j$ we need
$\Delta_{v}^{j}(\overline{p_{v}})$ not to vanish, and $0 \leq k \leq k_{max}$.

\medskip

If $p_{v}$ is homogenous of degree $(m_{+},m_{-})$ with respect to $v$ and
$\sigma \in O^{+}(V)$ then $p_{v}\circ\sigma^{-1}$ has the same homogeneity
degree with respect to $\sigma v$. Moreover, one has $\Delta_{\sigma
v}(p\circ\sigma^{-1})=(\Delta_{v}p)\circ\sigma^{-1}$. Section 6 of \cite{[B1]}
states that the equality \[\Phi_{L}(\sigma
v,F,p_{v}\circ\sigma^{-1})=\Phi_{L}(v,f,p_{v})\] holds wherever the regularized
theta lift of $F$ from Theorem \ref{regthetalift} is well-defined and $\sigma$
lies in the discriminant kernel $\Gamma$ from Proposition \ref{modOVTheta}. For
$p=1$ this gives the $\Gamma$-invariance of $\Phi_{L}$. For a general polynomial
$p_{v}$ the automorphic property of $\Phi_{L}$ depends on the behavior of $v
\mapsto p_{v}$ under $\Gamma$. The latter may be complicated in general, but for
the case $b_{+}=2$ and the polynomials $P_{r,s,t}$ defined above (or even for
$P_{r,s,t}(\lambda_{v_{-}}^{2})^{h}$), Proposition \ref{modOVTheta} implies
\begin{equation}
\Phi_{L,r,s,t}(\sigma
v,F)=j(\sigma,Z)^{s-t}\overline{j(\sigma,Z)}^{s-r}\Phi_{L,r,s,t}(v,F)
\label{modOVPhi}
\end{equation}
for any $Z \in K_{\mathbb{R}}+iC$ and $\sigma\in\Gamma$ (the complex conjugation
on $\Theta_{L}$ in $\Phi_{L}$ interchanges the powers of $j$ and
$\overline{j}$). Here and throughout, $\Phi_{L,r,s,t}(v,F)$ denotes
$\Phi_{L}(v,F,P_{r,s,t})$.

\smallskip

At this point we need to sort out some inaccuracies in the proofs of \cite{[B1]}
which are relevant for our discussion. In particular, the argument based on
Lemma 14.1 of this reference (specifically Corollary 6.3 and Corollary 14.2
there) can be replaced by the following

\begin{lem}
Let $C$ be a positive integer, and let $p$ be a polynomial of degree
smaller than $C$. Then $\sum_{j=0}^{C}(-1)^{j}\binom{C}{j}p(j)=0$.
\label{binompol}
\end{lem}

\begin{proof}
Let $\binom{x}{r}$ is the polynomial $\prod_{i=0}^{r-1}(x-i)/r!$, of degree
$r$. Since we can write $p(x)$ as $\sum_{r=0}^{\deg p}a_{r}\binom{x}{r}$, it
suffices to prove the claim for $p(x)=\binom{x}{r}$ with $0 \leq r<C$. In this
case only terms with $j \geq r$ contribute, and the contribution is given by
$\binom{C}{j}\binom{j}{r}=\binom{C}{r}\binom{C-r}{j-r}$. Hence we find that
\[\sum_{j=r}^{C}(-1)^{j}\binom{C}{j}\binom{j}{r}=(-1)^{r}\binom{C}{r}\sum_{i=0}^
{C-r}(-1)^{i}\binom{C-r}{i}=0,\] as $r<C$. This proves the lemma.
\end{proof}

Indeed, both Corollary 6.3 and Corollary 14.2 of \cite{[B1]} involve sums over
binomial coefficients of the form $\binom{A+D-2j}{A-2j}=\binom{A+D-2j}{D}$ for
$A$ and $D$ non-negative integers with $D<C$ which are independent of $j$. This
expression is a polynomial function of $j$ of degree $D<C$, so that our Lemma
\ref{binompol} applies for both cases.

Somewhat more disturbing are the assertions in Corollary 6.3 and Theorem 10.3 of
this reference which state that certain expressions are polynomials in an
oriented norm 1 vector $v_{1}$ in a Lorentzian space. Indeed, in our notation
(with $v \mapsto p_{v}$), the ``wall crossing formula'' from Corollary 6.3
states that the difference between the values of the theta lift on two adjacent
Weyl chambers $W$ and $\widetilde{W}$ with separating ``wall'' $\lambda^{\perp}$
for $(\lambda,W)>0$ is
\begin{equation}
\sum_{x\lambda \in
L^{*}}\sum_{j,k}\frac{4c_{x\lambda,\frac{x^{2}\lambda^{2}}{2},k}\Delta_{v}^{j
}(\overline{p_{v}})(x\lambda)\Gamma\big(m_{+}-j-k-\frac{1}{2}\big)}{(-8\pi)^{j}
j!}(\sqrt{2\pi}x(\lambda,v_{1}))^{1+2j+2k-2m_{+}}, \label{wallcross}
\end{equation}
with only $x>0$ and with $v_{1}$ being an oriented norm 1 vector spanning
$v_{+}$. The proof there shows that only positive powers of $(\lambda,v_{1})$
appear. However, this is not necessarily a polynomial in $v_{1}$ because of the
unknown dependence of $p_{v}$ on $v$ (or on $v_{1}$). The fact that in
\cite{[B1]} one does not work with $v \in G(M)$ but with $v$ an isometry from
$L_{\mathbb{R}}$ to $\mathbb{R}^{b_{+},b_{-}}$ does not overcome this problem,
since then the expression in Equation \eqref{wallcross} depends on $v$ as an
isometry and not only on its image in $G(L_{\mathbb{R}})$. Hence this expression
cannot be described as a function of $v_{1}$ alone. Only the smoothness of this
difference as a function on all of $G(L_{\mathbb{R}})$ survives (at least in our
conventions with $p_{v}$). On the other hand, when
$p_{v}(\lambda)=(\lambda,v_{1})^{m_{+}}$ of homogeneity degree $(m_{+},0)$ (and
no multiplying constant), the expression in Equation \eqref{wallcross} is indeed
the restriction of a polynomial on $L_{\mathbb{R}}$ of the asserted degree to
the set of vectors of the form $v_{1}$.

\smallskip

The decomposition of $p$ appearing just above Lemma 5.1 of \cite{[B1]} takes, in
our notation, the following form. Given $v \in G(L_{\mathbb{R}})$, a polynomial
$p_{v}$ with the usual homogeneity property, and natural numbers $h_{+}$ and
$h_{-}$, there is a (unique) polynomial $p_{v,h_{+},h_{-}}$ on $K_{\mathbb{R}}$
such that the equality
\[p_{v}(\lambda)=\sum_{h_{+},h_{-}}(\lambda,z_{v_{+}})^{h_{+}}(\lambda,z_{v_{-}}
)^{h_{-}}p_{v,h_{+},h_{-}}\bigg[\bigg(\lambda-\frac{(\lambda,z)z_{v_{+}}}{z_{v_{
+}}^{2}}\bigg)\bigg/\mathbb{R}z\bigg]\] holds for every $\lambda \in
L_{\mathbb{R}}$. $p_{v,h_{+},h_{-}}$ is homogenous of degree
$(m_{+}-h_{+},m_{-}-h_{-})$ with respect to the decomposition $w$ of
$K_{\mathbb{R}}$ into the images $w_{\pm}$ of $z_{v_{\pm}}^{\perp} \subseteq
v_{\pm}$ in $K_{\mathbb{R}}$. Note that it depends on $v$ (not only on $w$),
and the map $v \mapsto w$ is not injective. The map
$\lambda\mapsto\big(\lambda-\frac{(\lambda,z)z_{v_{+}}}{z_{v_{+}}^{2}}
\big)\big/\mathbb{R}z$ corresponds to the map denoted $w$ in \cite{[B1]}. A
careful verification of Sections 5 and 7 of \cite{[B1]} shows that the reduction
formula in Theorem 7.1 of \cite{[B1]} holds in our notation with each
$p_{w,h_{+},h_{-}}$ replaced by $p_{v,h_{+},h_{-}}$. Now, Theorem 10.2 of
\cite{[B1]} implies that the theta lift $\Phi_{L}$ is a smooth function inside
any Weyl chamber. However, the assertion that it is a polynomial does not
necessarily hold, because of the same problem discussed in the previous
paragraph. Only in the special case in which
$p_{v}(\lambda)=(\lambda,v_{1})^{m_{+}}$, where we have $h_{-}=0$ and
$p_{v,m_{+},0}=1$, the assertion of Theorem 10.3 of \cite{[B1]} holds as stated.
This is related to the fact that the equality $p_{v}\circ\sigma^{-1}=p_{\sigma
v}$ holds for this $p_{v}$, since $\Gamma$-invariance is used in the proof of
that theorem.

These remarks show that one must be careful when investigating properties of
lifts of modular forms using theta functions with polynomials. However, in the
applications appearing in \cite{[B1]} one considers only the case $p=1$ (in
Sections 11, 13, and 15 there), or some multiple of
$\eta\mapsto(\eta,v_{1})^{m_{+}}$ (in Section 14). Hence the results of these
sections hold as stated.

\subsection{Differential Properties of Certain Theta Lifts}

Let
\[\delta_{l}=-\frac{R_{l}}{4\pi}=\frac{\partial_{\tau}}{2\pi i}-\frac{l}{4\pi
y}\] be the (normalized) weight raising operator for weight $l$, where
$\partial_{\tau}$ denotes $\frac{\partial}{\partial\tau}$ from now on. A
well-known formula for the composition of these operators, which is easily
proved by induction, states that
\[\delta_{l}^{m}=\delta_{l+2m-2}\circ\ldots\circ\delta_{l}=\sum_{k}\binom{m}{k}
\Bigg[\prod_{s=m-k}^{m-1}(l+s)\Bigg]\bigg(\frac{-1}{4\pi
y}\bigg)^{k}\bigg(\frac{\partial_{\tau}}{2\pi i}\bigg)^{m-k}\] (see, for
example, Equation (56) in \cite{[Za]}). Moreover, $\delta_{l}$ sends
eigenfunctions of (minus) the weight $l$ Laplacian $\Delta_{l}$ having eigenvalue $\lambda$ to eigenfunctions of (minus) $\Delta_{l+2}$ with eigenvalue $\lambda+l$. Hence the $\delta_{l}^{m}$-images of the former functions have
the eigenvalue $\lambda+m(l+m-1)$. In particular, if
$f=\sum_{\gamma,n}c_{\gamma,n}q^{n}e_{\gamma}$ is weakly holomorphic of weight
$1-\frac{b_{-}}{2}-m$ and representation $\rho_{L}$ for some even lattice $L$ of
signature $(2,b_{-})$ then
\begin{equation}
F=\delta_{1-\frac{b_{-}}{2}-m}^{m}f=\sum_{\gamma,n,k}c_{\gamma,n,k}\frac{q^{n}
}{y^{k}}e_{\gamma}, \quad
c_{\gamma,n,k}=\binom{m}{k}n^{m-k}\cdot\prod_{r=0}^{k-1}\bigg(r+\frac{b_{-}}{2}
\bigg)\cdot\frac{c_{\gamma,n}}{(4\pi)^{k}} \label{coeff}
\end{equation}
is an almost weakly holomorphic modular form of weight $1-\frac{b_{-}}{2}+m$ and
representation $\rho_{L}$ which is an eigenfunction of (minus)
$\Delta_{1-\frac{b_{-}}{2}+m}$ with eigenvalue $-\frac{mb_{-}}{2}$.

\medskip

An important feature of $\frac{i^{m}}{2}\Phi_{L,m,m,0}(v,F)$ is that it is an
eigenfunction of (minus) the operator $\Delta^{G}_{m}=\Delta^{G}_{m,0}$. The
proof is similar to the results appearing in Section 4.1 of \cite{[Bru]}.
However, as the regularization in \cite{[Bru]} is different from ours, we give
the proofs of all assertions, with an emphasis on the differences relative to
\cite{[Bru]}.

\smallskip

\begin{lem}
Let $f$ and $g$ be modular forms of weights $k$ and $k+2$ respectively, and
representation $\rho_{L}$. The equality
\[\int_{D_{w}}\!\!\!\!\big(y^{k+2}\langle R_{k}f,g \rangle+y^{k}\langle
f,Lg\rangle\big)y^{-s}d\mu\!=\!s\!\!\int_{D_{w}}\!\!\!\!y^{k+1}\langle
f,g\rangle
y^{-s}d\mu-\!\int_{\mathbb{R}/\mathbb{Z}+iw}\!\!\!\!\!\!\!w^{k}\langle
f,g\rangle w^{-s}dx\] holds for any $w>1$ and $s\in\mathbb{C}$.
\label{intRkL}
\end{lem}

\begin{proof}
For the proof, see Lemma 4.2 \cite{[Bru]} (with some sign differences). The
first term on the right hand side arises from the difference between
$y^{-s}d\omega$ and $d(y^{-s}\omega)$, for $\omega$ being the
$SL_{2}(\mathbb{Z})$-invariant 1-form $y^{k}\langle f,\overline{g} \rangle
d\overline{\tau}$ appearing in \cite{[Bru]}. Note that while $y^{-s}\omega$ is
not $Mp_{2}(\mathbb{Z})$-invariant in general, the fact that $y^{-s}$ is
$T$-invariant and its restriction to the curve $|\tau|=1$ is also $S$-invariant
(since $\Im(S\tau)=\frac{y}{|\tau|^{2}}$) allows us to apply the argument also
in this case. This proves the lemma.
\end{proof}

Using Lemma \ref{intRkL} twice yields the following analog of Lemma 4.3 of
\cite{[Bru]}:

\begin{lem}
Let $f$ and $g$ be two modular forms of weight $k$ and representation
$\rho_{L}$, and let $w>1$ and $s\in\mathbb{C}$ be given. The difference
\[\int_{D_{w}}y^{k}\langle\Delta_{k}f,g\rangle
y^{-s}d\mu-\int_{D_{w}}y^{k}\langle f,\Delta_{k}g \rangle y^{-s}d\mu\] equals
\[s\int_{D_{w}}y^{k-1}\big(\langle Lf,g\rangle-\langle
f,Lg\rangle\big)y^{-s}d\mu-\int_{\mathbb{R}/\mathbb{Z}+iw}w^{k-2}\big(\langle
Lf,g\rangle-\langle f,Lg\rangle\big)w^{-s}dx.\] \label{intDeltak}
\end{lem}

We are interested in the case where
$g(\tau)=y^{\frac{b_{-}}{2}+m_{-}}\Theta_{L}(\tau,v,p_{v})$ (so that
$k=\frac{b_{+}}{2}+m_{+}-\frac{b_{-}}{2}-m_{-}$). The line integrals at the
limit $w\to\infty$ are dealt with in the following

\begin{lem}
Let $L$ be an even lattice of signature $(b_{+},b_{-})$, let $v$ be an element
of $G(L_{\mathbb{R}})$ which does not belong to any $\lambda^{\perp}$ for
$\lambda \in L^{*}$, and let $p$ be a polynomial on $L_{\mathbb{R}}$ which is
homogenous of degree $(m_{+},m_{-})$ with respect to $v$. Let $F$ be a modular
form of weight $\frac{b_{+}}{2}+m_{+}-\frac{b_{-}}{2}-m_{-}$ and representation
$\rho_{L}$, and write $F$ as in Equation \eqref{Fourier}. Assume that for every
$\gamma$ and $n$ we have $c_{\gamma,n}(y)=o(e^{\varepsilon y})$ as $y\to\infty$
for every $\varepsilon>0$. In the case where $m_{+}+m_{-}$ is even and the
constant $\Delta_{v}^{j}(p)$ for $j=\frac{m_{+}+m_{-}}{2}$ does not vanish, we
assume further that $c_{0,0}(y)=o(y^{T})$ as $y\to\infty$ for some $T$. Then
\[\lim_{w\to\infty}\int_{\mathbb{R}/\mathbb{Z}}w^{\frac{b_{+}}{2}+m_{+}-2-s}
\langle F(x+iw),\Theta_{L}(x+iw,v,p_{v})\rangle_{\rho_{L}}dx=0\] for large
enough $\Re s$. \label{boun0}
\end{lem}

\begin{proof}
For fixed $w$, the integral equals some power of $w$ times the constant term of
the Fourier expansion of $\langle F,\Theta_{L}\rangle$ at $y=w$. Hence we are
considering the limit of the expression
\begin{equation}
w^{\frac{b_{+}}{2}+m_{+}-2-s}\sum_{\lambda \in L^{*}}e^{-\Delta_{v}/8\pi
w}(\overline{p})(\lambda)c_{\lambda,\frac{\lambda^{2}}{2}}(w)e^{-2\pi
w\lambda_{v_{+}}^{2}} \label{FThetaw}
\end{equation}
as $w\to\infty$. As $v\not\in\lambda^{\perp}$ for any $\lambda \in L^{*}$, we
have $\lambda_{v_{+}}^{2}>0$ for any non-zero $\lambda \in L^{*}$, so that
$e^{-2\pi w\lambda_{v_{+}}^{2}}$ eliminates the sub-exponential growth of
$c_{\lambda,\frac{\lambda^{2}}{2}}(w)$. The contribution of the term with
$\lambda=0$ may be non-zero only if $m_{+}+m_{-}$ is even and the constant
$\Delta_{v}^{j}(p)$ with $j=\frac{m_{+}+m_{-}}{2}$ does not vanish. In this
case we consider some constant multiple of
$c_{0,0}(w)w^{\frac{b_{+}}{2}+m_{+}-2-j-s}$, which tends to 0 for $\Re s\gg0$
because of the polynomial growth of $c_{0,0}$. This proves the lemma.
\end{proof}

Using Lemmas \ref{intDeltak} and \ref{boun0} we generalize Lemma 4.4 of
\cite{[Bru]} as follows.

\begin{lem}
Let $L$, $v$, $p$, and $F$ be as in Lemma \ref{boun0}. Assume that the
regularized theta lift of $F$ from Theorem \ref{regthetalift} is well-defined.
If $e^{-\Delta_{v}/8\pi y}(q)(\lambda=0)$ vanishes for either $q=p$,
$q=\Delta_{v_{+}}p$, or $q=\lambda_{v_{-}}^{2}p$, then the theta lift of
$\Delta_{k}F(\tau)$, with the weight $k$ being
$\frac{b_{+}}{2}+m_{+}-\frac{b_{-}}{2}+m_{-}$, gives the same result as the
regularized integral \[\int_{X(1)}\langle
F(\tau),\Delta_{k}y^{\frac{b_{+}}{2}+m_{+}}\Theta_{L}(\tau,v,p_{v})\rangle_{
\rho_{L}}\frac{dxdy}{y^{2}}.\] \label{Lapmov}
\end{lem}

\begin{proof}
Fix $w>1$ and $s\in\mathbb{C}$. By Lemma \ref{intDeltak}, the difference between
the corresponding integrals on $D_{w}$ is the sum of a line integral at $y=w$
and a certain integral over $D_{w}$ multiplied by $s$. The weight lowering
operator $L$ preserves the properties of $F$ needed for Lemma \ref{boun0}, and
by Equation \eqref{LTheta} the image of
$y^{\frac{b_{-}}{2}+m_{-}}\Theta_{L}(\tau,v,p_{v})$ under $L$ is the sum of two
theta functions with polynomials. Thus, for large enough $\Re s$, Lemma
\ref{boun0} implies that the line integral vanishes as $w\to\infty$. For the
other integral, we may change the integral by a finite number without affecting
the constant term in the Laurent expansion at $s=0$ because of the factor $s$.
We thus replace the domain $D_{w}$ by the rectangle $x\in\mathbb{R}/\mathbb{Z}$
and $1 \leq y \leq w$. Equation \eqref{LTheta} reduces the expression we
consider to a linear combination of integrals of the form
\[s\int_{1}^{w}\int_{\mathbb{R}/\mathbb{Z}}y^{\alpha-s}\langle
G(\tau),\Theta_{L}(\tau,v,q_{v})\rangle dxdy\] for $(G,q)$ being $(LF,p)$,
$(F,\Delta_{v_{+}}p)$, and $(F,\lambda_{v_{+}}^{2}p)$, each with the
corresponding power $\alpha$. Expanding $\langle G,\Theta_{L}\rangle$ and
integrating over $x$ we obtain $s$ times the integral of a function of the form
appearing in Equation \eqref{FThetaw}. Our assumption on $v$ implies that the
integral over $[1,\infty)$ of every term with $\lambda\neq0$ is finite for every
$s$, so that the factor $s$ eliminates it at $s=0$. The difference between the
theta lift and the integral asserted in the lemma is thus the constant term at
$s=0$ of a linear combination of expressions of the form
\begin{equation}
s\Delta_{v}^{j}(\overline{q})\int_{1}^{\infty}y^{\alpha-j-s}c_{0,0}(y)dy
\label{intdif0}
\end{equation}
for the appropriate $j$. Since we assume that the coefficient
$\Delta_{v}^{j}(\overline{q})$ in Equation \eqref{intdif0} vanishes for the three possible polynomials $q$, the lemma follows.
\end{proof}

In general the expression in Equation \eqref{intdif0} does not vanish. For
example, if $F$ is weakly holomorphic then for large enough $\Re s$ Equation
\eqref{intdif0} becomes a linear combination of expressions of the form
$\frac{s}{\beta-s}$. Such a combination may have a non-zero constant term at
$s=0$ in case a term with $\beta=0$ appears.

The conditions required by Lemma \ref{boun0} and Lemma \ref{Lapmov} are
satisfied by a large variety of modular forms: Almost weakly holomorphic modular
forms, various eigenforms of $\Delta_{k}$, and the functions from Definition 1.8
in Section 1.3 of \cite{[Bru]} to name a few. These lemmas are thus applicable
in many settings.

\smallskip

Returning to the theta lift $\Phi_{L,m,m,0}(v,F)$ for
$F=\delta_{1-\frac{b_{-}}{2}-m}^{m}f$ and $f$ weakly holomorphic, we obtain

\begin{cor}
For $m>0$, or for $m=0$ (with $F=f$) under the condition $c_{0,0}=0$, the action
of the operator $\Delta^{G}_{m}$ multiplies $\Phi_{L,m,m,0}(v,F)$ by $2mb_{-}$.
\label{DeltaGmPhi}
\end{cor}

\begin{proof}
The operator $\Delta^{G}_{m}=\widetilde{\Delta}^{G}_{m,0}$ acts on the
conjugated theta function $\overline{\Theta_{L,m,m,0}}$ which shows up in the
definition of $\Phi_{L,m,m,0}$. Hence its action is the same as the action of
the conjugated operator $\widetilde{\Delta}^{G}_{0,m}=\Delta_{0,m}^{G}+2mb_{-}$
on the theta function $\Theta_{L,m,m,0}$ itself (after conjugating). But with
$r=s=m$ and $t=0$, the combination $\Delta_{0,m}^{G}+2mb_{-}$ coincides with the
operator appearing on the right hand side of Proposition \ref{DeltatauZ}. It
follows that $\Delta^{G}_{m}\Phi_{L,m,m,0}(v,F)$ is 4 times the integral from
Lemma \ref{Lapmov}. Since for $m>0$ the polynomial $e^{-\Delta_{v}/8\pi
y}(\overline{P}_{m,m,0})$ vanishes at $\lambda=0$ and for $m=0$ we assume
$c_{0,0}=0$, the expression in Equation \eqref{intdif0} vanishes in both cases.
Hence Lemma \ref{Lapmov} implies that $\Delta^{G}_{m}\Phi_{L,m,m,0}(v,F)$
coincides with $4\Phi_{L,m,m,0}(v,\Delta_{k}F)$. As $\Delta_{k}$ multiplies $F$
by $\frac{mb_{-}}{2}$, the corollary follows.
\end{proof}

\subsection{The Theta Lift of $F=\delta_{1-\frac{b_{-}}{2}-m}^{m}f$}

We begin by evaluating $p_{v,h_{+},h_{-}}$ for
$p_{v}=\frac{(-i)^{m}}{2}P_{m,m,0}$. Only elements with $h_{-}=0$ appear (since
$m_{-}=0$), and as $z_{v_{+}}=\frac{X_{v,V}}{Y^{2}}$ the binomial decomposition
\[\frac{(-i)^{m}}{2}\cdot\frac{(\lambda,X_{v,V}+iY_{v,V})^{m}}{(Y^{2})^{m}}
=\sum_{h_{+}}\binom{m}{h_{+}}\bigg(\frac{(\lambda,X_{v,V})}{Y^{2}}\bigg)^{h_{+}}
\frac{i^{m-h_{+}}}{2i^{m}}\cdot\frac{(\lambda,Y_{v,V})^{m-h_{+}}}{(Y^{2})^{m-h_{
+}}}\] implies
\begin{equation}
p_{v,h_{+},0}(\eta)=\frac{(-i)^{h_{+}}}{2}\binom{m}{h_{+}}\frac{(\eta,Y)^{m-h_{+
}}}{(Y^{2})^{m-h_{+}}}. \label{pvh+}
\end{equation}
This polynomial is homogenous of degree $(m-h_{+},0)$ with respect to the
element $w \in G(K_{\mathbb{R}})$ in which $w_{+}$ is spanned by $Y$ (or by the
normalized generator $\frac{Y}{|Y|}$). This illustrates the reason for using the
notation $p_{v,h_{+},0}$ rather than $p_{w,h_{+},0}$: The latter notation allows dependence only on $w$, i.e., on $\frac{Y}{|Y|}$, while Equation \eqref{pvh+} displays dependence on $Y$ itself (hence on $v$). On the other hand, we can write $p_{v,0,0}(\eta)$ as $\frac{1}{2|Y|^{m}}(\eta,w_{1})^{m}$ with
$w_{1}=\frac{Y}{|Y|}$ (as in Theorem 14.3 of \cite{[B1]}), so that the term with
$\Phi_{K}$ in Theorem 7.1 of \cite{[B1]} can be evaluated using the polynomial
$\widetilde{p}(\eta)=(\eta,w_{1})^{m}$. Since for this polynomial Theorem 10.3
of \cite{[B1]} remains valid, we deduce that $\Phi_{K}(v,F,p_{v,0,0})$ is a
polynomial in $\frac{Y}{|Y|}$ divided by $|Y|^{m}$.

\medskip

We can now state the properties of the theta lift $\Phi_{L,m,m,0}(v,F)$.

\begin{thm}
For $F=\delta_{1-\frac{b_{-}}{2}-m}^{m}f$ and $f$ a weakly holomorphic modular
form of weight $1-\frac{b_{-}}{2}-m$, the theta lift
$\frac{i^{m}}{2}\Phi_{L,m,m,0}(v,F)$ is a function of $Z \in K_{\mathbb{R}}+iC$
whose singularity along $\lambda^{\perp}$ for negative norm $\lambda$ is given
by \[\frac{1}{2}\sum_{\alpha\lambda \in
L^{*}}c_{\alpha\lambda,\frac{\alpha^{2}\lambda^{2}}{2}}\frac{(i\alpha)^{m}}{
(2\pi)^{m}}\Bigg[\prod_{r=0}^{m-1}\bigg(r+\frac{b_{-}}{2}\bigg)\cdot\frac{
(\lambda,\overline{Z_{v,V}})^{m}}{2^{m}(Y^{2})^{m}}\cdot\bigg(-\ln\frac{
|(\lambda,Z_{v,V})|^{2}}{Y^{2}}\bigg)+\]
\[+\sum_{k=0}^{m-1}\frac{m!}{2^{k}k!}\bigg(\frac{\lambda^{2}}{2}\bigg)^{m-k}
\cdot\prod_{r=0}^{k-1}\bigg(r+\frac{b_{-}}{2}\bigg)\cdot\frac{(\lambda,\overline
{Z_{v,V}})^{k}}{(\lambda,Z_{v,V})^{m-k}(Y^{2})^{k}}\cdot\frac{1}{m-k}\Bigg].\]
This function is annihilated by $\Delta^{G}_{m}-2mb_{-}$ outside its
singularities. Its Fourier expansion at the primitive norm 0 vector $z$ of $L$
(if it exists) decomposes, in a Weyl chamber $W$ containing $z$ in its closure,
as \[\frac{\sqrt{2}\varphi(Y)}{\pi^{m-1}|Y|^{m-1}}+\sum_{k=0}^{m}\sum_{C=0}^{k}
\sum_{\rho \in
K^{*}}\frac{A_{k,C,\rho}}{\pi^{k+C}}\frac{(\rho,Y)^{k-C}}{(Y^{2})^{k}}
\times\left\{\begin{array}{cc}\mathbf{e}\big((\rho,Z)\big) & (\rho,W)>0 \\
\mathbf{e}\big((\rho,\overline{Z})\big) & (\rho,W)<0.\end{array}\right.\] Here
$\varphi(Y)$ is a polynomial of degree not exceeding $m+1$ in $\frac{Y}{|Y|}$,
plus some constant divided by $|Y|^{m+1}$. The constants $A_{k,C,\rho}$ involve
rational numbers, roots of unity of some finite order, and the Fourier
coefficients of $f$. \label{PhiLmm0}
\end{thm}

\begin{proof}
The fact that $\Delta^{G}_{m}\Phi_{L,m,m,0}=2mb_{-}\Phi_{L,m,m,0}$ is the
content of Corollary \ref{DeltaGmPhi}. The singularities can be read off
Theorem \ref{sing}, with $b_{+}=2$, $m_{+}=m$, $k \leq m$, and $j=0$ since
$p_{v}$ is harmonic. The condition $\beta\in-\mathbb{N}$ holds only for $k=m$,
and the term with $\ln(\alpha^{2})$ multiplies a smooth function. Substituting
the coefficients $c_{\gamma,n,k}$ from Equation \eqref{coeff} thus yields the
asserted singularity near $\lambda^{\perp}$. The rest of the proof follows the
proof of Theorem 14.3 of \cite{[B1]}. We assume that $L$ contains a primitive
norm 0 vector (otherwise the assertion about Fourier expansions is vacuous), and
that $v$ (or $w$) is not on any wall between two Weyl chambers. Since the
polynomial $p_{v}$ is of homogeneity degree $(m,0)$, we can take $h_{-}=h=0$ in
Theorem 7.1 of \cite{[B1]}. Thus, $\frac{i^{m}}{2}\Phi_{L,m,m,0}(v,F)$ is the
sum of $\frac{|Y|}{\sqrt{2}}\Phi_{K}(w,F_{K},p_{v,0,0})$, a term coming from the
element $0 \in K^{*}$, and the expression
\[\sqrt{2}|Y|\sum_{h_{+},j,k,n}\frac{n^{h_{+}}}{(2i)^{h_{+}}j!(-8\pi)^{j}}\sum_{
0\neq\eta \in
K^{*}}\Delta_{w}^{j}(\overline{p}_{v,h_{+},0})(\eta)\sum_{\gamma}\mathbf{e}
\big(n[(\eta,X)+(\gamma,\zeta)]\big)\times\] \[\times
c_{\gamma,\frac{\eta^{2}}{2},k}\cdot2\bigg(\frac{nY^{2}}{2|(\eta,Y)|}\bigg)^{
m-h_{+}-j-k-\frac{1}{2}}K_{m-h_{+}-j-k-\frac{1}{2}}\big(2\pi n|(\eta,Y)|\big).\]
Here $\gamma$ is the $L^{*}/L$-image of an element of $L^{*}$ whose restriction
to $z^{\perp} \subseteq L$ is the pull-back of $\eta:K\to\mathbb{Z}$ under the
projection $z^{\perp} \to K$. The derivation employs Lemma 7.2 of \cite{[B1]},
together with the fact that $b_{+}=2$, $z_{v_{+}}^{2}=\frac{1}{Y^{2}}$, $\mu=X$,
and $\eta_{w_{+}}^{2}=\frac{(\eta,Y)^{2}}{Y^{2}}$.

Next, applying $\Delta_{w_{+}}$ (which differentiates twice with respect to the
pairing with $\frac{Y}{|Y|}$) to the complex conjugate of $p_{v,h_{+},0}$ from
Equation \eqref{pvh+} gives
\begin{equation}
\Delta_{w}^{j}(\overline{p_{v,h_{+},0}})(\eta)=\frac{i^{h_{+}}}{2}\cdot\frac{m!}
{h_{+}!(m-h_{+}-2j)!}\cdot\frac{(\eta,Y)^{m-h_{+}-2j}}{(Y^{2})^{m-h_{+}-j}}.
\label{Lapjpvh+}
\end{equation}
Furthermore, we quote from the proof of Theorem 14.3 of \cite{[B1]} the formula
\[K_{\nu+\frac{1}{2}}(t)=\sqrt\frac{\pi}{2t}e^{-t}\sum_{r=0}^{\nu}\frac{(\nu+r)!
}{r!(\nu-r)!}\frac{1}{(2t)^{r}}\] for the $K$-Bessel function of half-integral
index (here $\nu\in\mathbb{N}$, including 0). Since inverting the index leaves
the $K$-Bessel function invariant, the same expansion holds for
$K_{-\nu-\frac{1}{2}}(t)$. Substituting, and collecting the total powers of 2,
$\pi$, $n$, $(\eta,Y)$, $|(\eta,Y)|$, and $Y^{2}$ we find that this expression
reduces to
\[\sum_{h_{+},j,k,n}\sum_{\eta,\gamma,r}\frac{m!(\nu+r)!c_{\gamma,\frac{\eta^{2}
}{2},k}(-1)^{j}}{h_{+}!(m-h_{+}-2j)!j!r!(\nu-r)!}2^{k-m-2j-2r}\pi^{-j-r}n^{
m-k-j-r-1}\times\]
\begin{equation}
\times|(\eta,Y)|^{k-j-r}sgn(\eta,Y)^{m-h_{+}}(Y^{2})^{-k}\mathbf{e}\big(n[
(\eta,X)+i|(\eta,Y)|+(\gamma,\zeta)]\big). \label{FourexpPhi}
\end{equation}
Here $\nu$ stands for $m-h_{+}-j-k-1$ when this integer is non-negative and for
$j+k+h_{+}-m$ otherwise. We observe that $j$ and $r$ appear in exponents only
through their sum $C=j+r$, and we wish to show that only terms with $C \leq k$
appear. Considering the terms in which $\nu=j+k+h_{+}-m$, we recall that $2j$ is
bounded by $m-h_{+}$. Hence we have the inequality $j+h_{+}-m\leq-j$, which
implies $r\leq\nu \leq k-j$ hence $C=j+r \leq k$. For the other terms we fix $C$
and write $r=C-j$, so that the only part depending on $j$ is
\[\frac{1}{C!}\sum_{j}(-1)^{j}\binom{C}{j}\frac{(m-h_{+}+C-k-1-2j)!}{(m-h_{+}
-2j)!}.\] For $C \geq k+1$ this sum equals
\[\frac{(C-k-1)!}{C!}\sum_{j}(-1)^{j}\binom{C}{j}\binom{m-h_{+}+C-k-1-2j}{C-k-1}
,\] hence vanishes by Lemma \ref{binompol} since the rightmost binomial
coefficient is a polynomial of degree $C-k-1<C$ in $j$. Thus, only terms with $C
\leq k$ survive, which implies that $|(\eta,Y)|^{k-C}$ is continuous for every
$\eta$.

Now, by Equation \eqref{coeff}, $c_{\gamma,\frac{\eta^{2}}{2},k}$ is some
rational multiple of
$c_{\gamma,\frac{\eta^{2}}{2}}\frac{(\eta^{2})^{m-k}}{\pi^{k}}$. Thus, given
$k$, $C$, $n$, $\eta$, and $\gamma$ (with $C \leq k$ and the usual relation
between $\gamma$ and $\eta$), the corresponding terms in Equation
\eqref{FourexpPhi} take the form
\[a_{k,C,\varepsilon}\frac{n^{m-k-C-1}(\eta^{2})^{m-k}}{\pi^{k+C}}c_{\gamma,
\frac{\eta^{2}}{2}}\frac{(\eta,Y)^{k-C}}{(Y^{2})^{k}}\mathbf{e}\big(n[(\eta,
X)+i|(\eta,Y)|+(\gamma,\zeta)]\big),\] where $\varepsilon=sgn(\eta,Y)$ and the
coefficients $a_{k,C,\varepsilon}$ are rational numbers. \cite{[B1]} shows that
$a_{0,0,\varepsilon}$ equals $2^{m}$ for $\varepsilon=+1$ and 0 for
$\varepsilon=-1$. Now, the power of $n$ lies in $\mathbb{Q}$, and raising
$\mathbf{e}((\gamma,\zeta))$ to a power which equals the level of $L$ gives 1.
The remaining part of the exponent is $\mathbf{e}\big(n(\eta,Z)\big)$ if
$(\eta,Y)>0$ and $\mathbf{e}\big(n(\eta,\overline{Z})\big)$ if $(\eta,Y)<0$. We
may substitute $n\eta=\rho$, and the total coefficient of
$\frac{(\rho,Y)^{k-C}}{\pi^{k+C}(Y^{2})^{k}}$ times
$\mathbf{e}\big((\rho,Z)\big)$ or $\mathbf{e}\big((\rho,\overline{Z})\big)$
becomes
\begin{equation}
A_{k,C,\rho}=a_{k,C,\varepsilon}(\rho^{2})^{m-k}\sum_{n>0,\frac{\rho}{n} \in
L^{*}}\sum_{\gamma|_{z^{\perp}}=\frac{\rho}{n}}\frac{c_{\gamma,\frac{\rho^{2}}{
2n^{2}}}}{n^{m+1}}\mathbf{e}\big(n(\gamma,\zeta)\big). \label{AkCrho}
\end{equation}
This coefficient clearly has the asserted properties. As the dependence on
$X=\mu$ comes only from $\mathbf{e}\big((\eta,\mu)\big)$, the terms with
$\Phi_{K}$ or with $\eta=0$ all go into $\varphi(Y)$.

Now, the term involving $\Phi_{K}$ was discussed above (note the factor
$\frac{1}{|Y|^{m}}$ from $p_{v,0,0}$ but also the coefficient $\sqrt{2}|Y|$
from Theorem 7.1 of \cite{[B1]}). The reason for the coefficient
$\frac{\sqrt{2}}{\pi^{m-1}}$ is explained below. The polynomial from Theorem
10.3 of \cite{[B1]} has degree at most $m+1$ since $m_{+}=k_{max}=m$ and
$m_{-}=0$. For the term with $\eta=0$, fix $h_{+}$ and $j$. As
$\Delta_{w}^{j}(\overline{p}_{v,h_{+},0})(\eta)$ does not vanish for $\eta=0$
only if $2j=m-h_{+}$, we fix only $j$ and write $h_{+}=m-2j$. Recall that for
$\eta=0$ the values which $\gamma \in L^{*}/L$ attains are the images of
$\frac{\delta z}{N}$ for $\delta\in\mathbb{Z}/N\mathbb{Z}$ (where $N$ is defined
by $(L,z)=N\mathbb{Z}$). Combining this with Equation \eqref{Lapjpvh+} and Lemma
7.3 of \cite{[B1]} thus yields the expression
\[\sqrt{2}|Y|\sum_{j,k,n}\frac{n^{m-2j}}{(2i)^{m-2j}j!(-8\pi)^{j}}\cdot\frac{i^{
m-2j}}{2}\cdot\frac{m!}{(m-2j)!(Y^{2})^{j}}\sum_{\delta\in\mathbb{Z}/N\mathbb{Z}
}\mathbf{e}\bigg(\frac{\delta n}{N}\bigg)\times\] \[\times c_{\frac{\delta
z}{N},0,k}\cdot\bigg(\frac{\pi
n^{2}Y^{2}}{2}\bigg)^{j-k-\frac{1}{2}-s}\Gamma\bigg(s+\frac{1}{2}+k-j\bigg).\]
Collecting powers gives \[\sum_{k,\delta}c_{\frac{\delta
z}{N},0,k}\frac{2^{k-m}}{(Y^{2})^{k}\pi^{k+\frac{1}{2}}}\sum_{j}\frac{(-1)^{j}m!
}{4^{j}j!(m-2j)!}\Gamma\bigg(s+\frac{1}{2}+k-j\bigg)\bigg(\frac{2}{\pi
Y^{2}}\bigg)^{s}\times\]
\begin{equation}
\times\sum_{n}\mathbf{e}\bigg(\frac{\delta n}{N}\bigg)n^{m-2k-1-2s}.
\label{eta0}
\end{equation}
The sum over $n$ equals
$N^{m-2k-1-2s}\sum_{\varepsilon=1}^{N}\mathbf{e}\big(\frac{\delta\varepsilon}{N}
\big)\zeta\big(1+2k-m+2s,\frac{\varepsilon}{N}\big)$, involving values of the
Hurwitz zeta function. Now, Equation \eqref{coeff} implies that $c_{\frac{\delta
z}{N},0,k}$ equals
$\prod_{r=0}^{k-1}\big(r+\frac{b_{-}}{2}\big)\cdot\frac{c_{\frac{\delta
z}{N},0}}{(4\pi)^{m}}$ for $k=m$ and vanishes otherwise. Since the Hurwitz zeta
function is holomorphic at $m+1$, the constant term of this expression at $s=0$
is obtained by substitution. The fact that
$\Gamma\big(\frac{1}{2}+k-j\big)\in\mathbb{Q}\cdot\sqrt{\pi}$ renders the
expression from Equation \eqref{eta0}
\[\frac{\widetilde{a}}{(Y^{2})^{m}\pi^{2m}}\sum_{\delta\in\mathbb{Z}/N\mathbb{Z}
}c_{\frac{\delta
z}{N},0}\sum_{\varepsilon=1}^{N}\mathbf{e}\bigg(\frac{\delta\varepsilon}{N}
\bigg)\zeta\bigg(m+1,\frac{\varepsilon}{N}\bigg)\] with some
$\widetilde{a}\in\mathbb{Q}$. Since this is a constant times
$\frac{1}{(Y^{2})^{m}}$, adding it to the term with $\Phi_{K}$ shows that
$\varphi(Y)$ has the desired form. This proves the theorem.
\end{proof}

A remark about the evaluation of the constant term $P_{2}$ in Theorem 14.3 of
\cite{[B1]} is in order here. In the case of weakly holomorphic $F$ considered
there, only $k=0$ appears. Evaluating the Hurwitz zeta function thus obtained
and using the symmetry and duplication formulae for the gamma function shows
that the true value of $P_{2}$ is half the value given in the proof of Theorem
14.3 of \cite{[B1]}. The constant appearing in the assertion of that theorem is,
on the other hand, correct. One should thus be careful when evaluating $P_{2}$
as the limit of the expression for $\eta\neq0$ (indeed, with $(\eta,W)<0$ one
gets 0). In any case, the proof of Theorem \ref{PhiLmm0} shows how to evaluate
this term directly.

Note that the coefficients of the polynomial part of $\varphi(Y)$ involve the
lattice $K$, some roots of unity, and the Fourier coefficients of the modular
form $f$ (this is the reason for the coefficient $\frac{\sqrt{2}}{\pi^{m-1}}$).
Hence for ``algebraic'' $f$ these coefficients are algebraic. On the other hand,
evaluating the constant coming from Equation \eqref{eta0} for $\eta=0$ is very
hard. For example, for $N=1$ we must take even $m$, yielding the values of the
Riemann zeta function at odd positive integers, whose properties (not to mention
a finite formula) are not yet known.

We remark that Theorem \ref{PhiLmm0} extends to the case $m=0$, if $c_{0,0}=0$.
In general, $\Delta^{G}_{0}\frac{1}{2}\Phi_{L,0,0,0}(v,F)$ is a non-zero
constant multiple of $c_{0,0}$. Indeed, in this case one obtains the usual theta
lift with $p=1$ from Theorem 13.3 of \cite{[B1]} (up to a factor of 2),
$\ln|\Psi|^{2}$ is harmonic for meromorphic $\Psi$, and $\Delta^{G}_{0}=8\Omega$
sends $\ln|Y|$ to a non-zero constant by Equation (4.3) of \cite{[Bru]}. We
mention that Theorem \ref{PhiLmm0}, as well as Theorems 4.6 and 4.7 of
\cite{[Bru]}, provide answers to Problem 16.6 of \cite{[B1]} in some
(interesting) special cases.

\smallskip

We remark that $\frac{i^{m}}{2}\Phi_{L,m,m,0}(v,F)$ can be written in terms of
the weakly holomorphic modular form $f$ itself, namely as
$\frac{(-i)^{m}}{2^{m+1}}\Phi_{L}(v,f,p_{v})$ with the polynomial
$p_{v}(\lambda)=P_{m,m,0}(\lambda,Z)(\lambda_{v_{-}}^{2})^{m}$ (recall that
$P_{m,m,0}$ is $\Delta_{v_{+}}$-harmonic). This is a consequence of Equation
\eqref{LTheta}, using an argument similar to Lemmas \ref{intRkL}, \ref{boun0},
and \ref{Lapmov}. This is useful in case one wishes to apply the embedding trick
from Section 8 of \cite{[B1]}. Indeed, the lifted modular form is then divided
by the classical holomorphic cusp form $\Delta$ of weight 12 (not to be confused
with the various Laplacians, whose notation also involves the symbol $\Delta$).
Now, if $f$ is weakly holomorphic then so is $\frac{f}{\Delta}$, while if $F$ is
an eigenfunction of $\Delta_{k}$ then $\frac{F}{\Delta}$ does not necessarily
share this property.

\subsection{The Case $b_{-}=1$}

In the case $b_{-}=1$ we may consider our lattices as embedded in the vector
space $V=M_{2}(\mathbb{R})_{0}$ of traceless $2\times2$ real matrices. This
space has signature $(2,1)$ with $(A,B)=Tr(AB)$ and $A^{2}=-2\det A$. Take
$z=\binom{\ \ \ 1/\beta}{}$ for some non-zero $\beta$ and $\zeta=\binom{\ \
-h/2\beta}{\beta\ \ \ \ \ \ \ }$ (where $h=\zeta^{2}$). Then
$K_{\mathbb{R}}\cong\{z,\zeta\}^{\perp}$ is the space of traceless diagonal
matrices. We send $x\in\mathbb{R}$ to $\binom{\beta x\ \ \ \ \
\ }{\ \ \ \ -\beta x} \in K_{\mathbb{R}}$ with the norm $2\beta^{2}x^{2}$, and
take $C$ to be the cone of positive reals. Thus $G(V) \cong K_{\mathbb{R}}+iC$
is identified with $\mathcal{H}$. Explicitly, given $\tau\in\mathcal{H}$, the
vector $Z_{v,V}$ is $\beta M_{\tau}$ and $v_{-}=\mathbb{R}J_{\tau}$ using the
matrices defined in the beginning of Section \ref{CMUF}. Automorphic forms are
modular forms of twice the weight, as Equation \eqref{Jtauconj} shows. The group
$SL_{2}(\mathbb{R})$ acts on $V$ by conjugation, yielding an isomorphism
$PSL_{2}(\mathbb{R}) \cong SO^{+}(V)$. We remark that right multiplication by
$S$ (and rescaling the bilinear form) gives the model in which the vector space
consists of the symmetric $2\times2$ matrices from Example 5.1 of \cite{[B2]}.
In order for $z$ to be a primitive norm 0 vector in the lattice $L$ considered
there we must take $\beta=\sqrt{N}$.

\medskip

Consider now $\Gamma^{+}=\Gamma \cap SO^{+}(L_{\mathbb{R}})$ (in the case
$b_{-}=1$) as a Fuchsian subgroup of $PSL_{2}(\mathbb{R})$.
$\frac{i^{m}}{2}\Phi_{L,m,m,0}(v,F)$ is then a modular form of weight $2m$ on
$\mathcal{H}$ with respect to $\Gamma^{+}$, which has eigenvalue $-2m$ with
respect to $\Delta^{G}_{m}=\Delta_{2m}$. A negative norm vector $\lambda \in
L^{*}$ is a multiple of $J_{\sigma}$ for some $\sigma=s+it\in\mathcal{H}$, and
then $\lambda^{\perp}=\{\sigma\}$. By a slight abuse of notation, we replace the
variable $v \in G(V)$ of $\Phi$ by $\tau=x+iy$ once more (as we consider only
$\Phi$ and not $\Theta$ here, this should lead to no confusion). The description
of $\frac{i^{m}}{2}\delta_{2m}\Phi_{L,m,m,0}(v,F)$ is given in

\begin{thm}
The function $\frac{i^{m}}{2}\delta_{2m}\Phi_{L,m,m,0}(v,F)$ is a meromorphic
modular form of weight $2m+2$ with respect to $\Gamma^{+}$. Its poles are at
points $\sigma\in\mathcal{H}$ for which a real multiple of $J_{\sigma}$ lies in
(the isomorphic copy of) $L^{*}$. The principal part at such $\sigma$ is
\[\frac{i}{(4\pi)^{m+1}}\sum_{\alpha J_{\sigma} \in L^{*}}c_{\alpha
J_{\sigma},-\alpha^{2}}\frac{\alpha^{m}}{\beta^{m}}\cdot\frac{m!(2it)^{m+1}}{
(\tau-\sigma)^{m+1}(\tau-\overline{\sigma})^{m+1}}.\] In case $\Gamma^{+}$ has
cusps, the Fourier expansion at such a cusp is
\[\sum_{r>0}\frac{r^{m}}{\beta^{2m}}\sum_{d|r}d^{m+1}\sum_{\gamma|_{z^{\perp}}
=\frac{d}{2\beta^{2}}}c_{\gamma,\frac{d^{2}}{4\beta^{2}}}\mathbf{e}\bigg(\frac{r
}{d}(\gamma,\zeta)\bigg)q^{r},\] plus some constant if $m=0$. \label{b-=1mer}
\end{thm}

\begin{proof}
We first observe that applying $\delta_{2m}$ to a modular form of weight $2m$
having eigenvalue $-2m$ yields a meromorphic modular form of weight $2m+2$. The
singularities of $\frac{i^{m}}{2}\delta_{2m}\Phi_{L,m,m,0}(v,F)$ are the
$\delta_{2m}$-images of the singularities of
$\frac{i^{m}}{2}\Phi_{L,m,m,0}(v,F)$ given in Theorem \ref{PhiLmm0}. The
constant $\prod_{r=0}^{k-1}\big(r+\frac{b_{-}}{2}\big)$ is
$\frac{(2k)!}{4^{k}k!}$, for $\lambda=J_{\sigma}$ we have $\lambda^{2}=-2$,
and $Y^{2}$ equals $2\beta^{2}y^{2}$. The pairing of $Z_{v,V}=\beta M_{\tau}$
and of $\overline{Z_{v,V}}=\beta \overline{M_{\tau}}$ with $\lambda=J_{\sigma}$
gives $-\frac{\beta}{t}(\tau-\sigma)(\tau-\overline{\sigma})$ and its complex
conjugate respectively. Hence the singularity of
$\frac{i^{m}}{2}\Phi_{L,m,m,0}(v,F)$ at $\sigma$ is \[\frac{1}{2}\!\sum_{\alpha
J_{\sigma} \in L^{*}}\!\!\!c_{\alpha
J_{\sigma},-\alpha^{2}}\frac{(i\alpha)^{m}}{(2\pi)^{m}}\!\Bigg[\frac{(2m)!}{4^{m
}m!}\cdot\frac{[-(\overline{\tau}-\sigma)(\overline{\tau}-\overline{\sigma})]^{m
}}{(4y^{2})^{m}\beta^{m}t^{m}}\cdot\!\bigg(\!-\ln\frac{|(\tau-\sigma)|^{2}
|(\tau-\overline{\sigma})|^{2}}{2y^{2}}\bigg)+\]
\[+\sum_{k=0}^{m-1}\frac{m!(-1)^{m-k}}{2^{k}k!}\cdot\frac{(2k)!}{4^{k}k!}\frac{
[-(\overline{\tau}-\sigma)(\overline{\tau}-\overline{\sigma})]^{k}t^{m-2k}}{[
-(\tau-\sigma)(\tau-\overline{\sigma})]^{m-k}\beta^{m}(2y^{2})^{k}}\cdot\frac{1}
{m-k}\Bigg].\] We write this expression as
\[\frac{m!}{2}\sum_{\alpha J_{\sigma} \in L^{*}}c_{\alpha
J_{\sigma},-\alpha^{2}}\bigg(\frac{i\alpha
t}{2\pi\beta}\bigg)^{m}\sum_{k=0}^{m}\frac{(2k)!(-1)^{k}}{(k!)^{2}(4yt)^{2k}}
\cdot(\overline{\tau}-\sigma)^{k}(\overline{\tau}-\overline{\sigma})^{k}g_{k}
(\tau),\] with $g_{k}(\tau)$ being
$\frac{1}{(m-k)(\tau-\sigma)^{m-k}(\tau-\overline{\sigma})^{m-k}}$ for $0 \leq
k<m$ and $-\ln\frac{|(\tau-\sigma)|^{2}|(\tau-\overline{\sigma})|^{2}}{2y^{2}}$
for $k=m$. Applying $\delta_{2m}$ reduces to letting $\delta_{2m-2k}$ operate on
$g_{k}$ for each $k$. Hence the singularity of
$\frac{i^{m}}{2}\delta_{2m}\Phi_{L,m,m,0}(v,F)$ at $\sigma$ is
\[\frac{m!i}{4\pi}\sum_{\alpha J_{\sigma} \in L^{*}}c_{\alpha
J_{\sigma},-\alpha^{2}}\bigg(\frac{i\alpha
t}{2\pi\beta}\bigg)^{m}\sum_{k=0}^{m}\frac{(2k)!(-1)^{k}}{(k!)^{2}(4yt)^{2k}}
(\overline{\tau}-\sigma)^{k}(\overline{\tau}-\overline{\sigma})^{k}\times\]
\[\times\frac{(\tau-\overline{\sigma})+(\tau-\sigma)+\frac{
2(\tau-\sigma)(\tau-\overline{\sigma})}{-2iy}}{(\tau-\sigma)^{m+1-k}
(\tau-\overline{\sigma})^{m+1-k}}.\]

For ease of presentation we omit, until the last step, the coefficient in front
of the sum over $k$. Expand the $k$th powers of
$\overline{\tau}-\sigma=(\tau-\sigma)-2iy$ and
$\overline{\tau}-\overline{\sigma}=(\tau-\overline{\sigma})-2iy$, and write
$\frac{(-1)^{k}}{(4yt)^{2k}}$ as $\frac{1}{(2t)^{2k}(-2iy)^{2k}}$. This gives
\[\sum_{k=0}^{m}\frac{(2k)!}{(2t)^{2k}}\!\sum_{a,b}\frac{1}{
a!(k-a)!b!(k-b)!(-2iy)^{a+b}}\bigg[\frac{1}{(\tau-\sigma)^{m+1-k-a}
(\tau-\overline{\sigma})^{m-k-b}}+\]
\[+\frac{1}{(\tau-\sigma)^{m-k-a}(\tau-\overline{\sigma})^{m+1-k-b}}+\frac{\frac
{2}{-2iy}}{(\tau-\sigma)^{m-k-a}(\tau-\overline{\sigma})^{m-k-b}}\bigg].\] Now
fix $c$ and $l$, and collect the terms involving
$\frac{1}{(\tau-\sigma)^{m-k-l+1}(-2iy)^{c}}$. We thus take $a=l$ and $b=c-l$ in
the first summand within the square brackets, $a=l-1$ and $b=c-l+1$ in the
second summand, and $a=l-1$ and $b=c-l$ in the third summand. This gives
\[\sum_{k,c,l}\frac{(2k)!(2t)^{-2k}}{(\tau-\sigma)^{m-k-l+1}(\tau-\overline{
\sigma})^{m-k-c+l}(-2iy)^{c}}\bigg[\frac{1}{l!(k-l)!(c-l)!(k+l-c)!}+\]
\[+\frac{1}{(l-1)!(k+1-l)!(c-l+1)!(k+l-1-c)!}+\]
\[+\frac{2}{(l-1)!(k+1-l)!(c-l)!(k+l-c)!}\bigg].\] The sum of the combinatorial
expressions reduces to $\frac{(k+1)(c+1)}{l!(k+1-l)!(c-l+1)!(k+l-c)!}$, so that
we can write the latter expression as
\[\sum_{c}\frac{(c+1)(-2iy)^{-c}}{(\tau-\sigma)^{m+1}(\tau-\overline{\sigma})^{
m+1-c}}\sum_{k,l}\frac{(2k)!(k+1)(\tau-\sigma)^{k+l}(\tau-\overline{\sigma})^{
k+1-l}}{(2t)^{2k}l!(k+1-l)!(c-l+1)!(k+l-c)!}.\]

Expand the power of $\tau-\overline{\sigma}=\tau-\sigma+2it$, and write
$(2t)^{2k}=(-1)^{k}(2it)^{2k}$. We thus obtain for every $c$ a sum of the form
\[\sum_{k,l,h}\frac{(-1)^{k}(2k)!(k+1)(\tau-\sigma)^{k+l+h}}{(2it)^{k+l+h-1}
l!(c-l+1)!(k+l-c)!h!(k+1-l-h)!}.\] The index change $k=s-l-h$ yields
\[\sum_{c,s}\!\frac{(-1)^{s}(c+\!1)(-2iy)^{-c}(2it)^{1-s}}{(\tau-\sigma)^{m+1-s}
(\tau-\overline{\sigma})^{m+1-c}}\!\sum_{l,h}\!\frac{(-1)^{l+h}
(2s-2l-2h)!(s-l-h+1)}{ l!(c-\!l+\!1)!(s-\!h-c)!h!(s+\!1-2l-2h)!}.\] For every
$c$ and $s$ now write $r=h+l$, and the corresponding term becomes
\[\sum_{r}\frac{(-1)^{r}(2s-2r)!(s-r+1)}{(s+1-2r)!(c+1)!(s-c)!}\sum_{l}\binom{
c+1}{l}\binom{s-c}{r-l}.\] A well-known combinatorial identity shows that the
sum over $l$ is simply $\binom{s+1}{r}$. By writing $(s+1-r)\binom{s+1}{r}$ as
$(s+1)\binom{s}{r}$ and $\frac{c+1}{(c+1)!(s-c)!}$ as
$\frac{1}{s!}\binom{s}{c}$, the expression in question reduces to
\[\sum_{c,s}\frac{(-1)^{s}(s+1)(-2iy)^{-c}(2it)^{1-s}}{(\tau-\sigma)^{m+1-s}
(\tau-\overline{\sigma})^{m+1-c}s!}\binom{s}{c}\sum_{r}(-1)^{r}\binom{s}{r}\frac
{(2s-2r)!}{(s+1-2r)!}.\] As
$\frac{(2s-2r)!}{(s+1-2r)!}=\prod_{i=s+2}^{2s}(i-2r)$ is a polynomial of degree
$s-1<s$ for all $s>0$, Lemma \ref{binompol} shows that the sum over $r$ vanishes
unless $s=0$. Thus only the term
$\frac{2it}{(\tau-\sigma)^{m+1}(\tau-\overline{\sigma})^{m+1}}$, corresponding
to $s=c=0$, survives. Multiplying by the factor we omitted above now gives the
asserted singularity at $\sigma$.

Assume that $\Gamma^{+}$ has a cusp, corresponding to a primitive norm 0 vector
$z \in L$ which is normalized as above. Theorem \ref{PhiLmm0} shows that for
$b_{-}=1$, the expansion of $\frac{i^{m}}{2}\Phi_{L,m,m,0}(v,F)$ at the cusp is
the sum of an almost holomorphic function of depth $\leq 2m$, the complex
conjugate of such a function, and an expression of the form
$\frac{B}{y^{m-1}}+\frac{D}{y^{2m}}$ for some constants $B$ and $D$. Note that
the condition on the sign of $(\rho,W)$ implies a sub-exponential growth at the
cusp. Moreover, the eigenvalue $-2m$ implies that $B=0$ (unless $m=0$) and that
the conjugate almost holomorphic part combines with $\frac{D}{y^{2m}}$ to a give
$\frac{h(\tau)}{y^{2m}}$ with $h$ anti-holomorphic. $\delta_{2m}$ thus
annihilates this part. Now, if the almost holomorphic part is
$\sum_{r}\frac{\psi_{r}(\tau)}{y^{r}}$ with meromorphic functions $\psi_{r}$
then the meromorphicity of $\frac{i^{m}}{2}\delta_{2m}\Phi_{L,m,m,0}(v,F)$ shows
that its expansion is just $\frac{\psi_{0}'(\tau)}{2\pi i}$. We normalize $z$
and $\beta$ such that $K^{*}=\frac{1}{2\beta^{2}}\mathbb{Z}\subseteq\mathbb{R}$
and $(\rho,W)>0$ is equivalent, for $\rho=\frac{r}{2\beta^{2}}$, to $r>0$. Then
the part with $(\rho,W)>0$ in the expansion of
$\frac{i^{m}}{2}\Phi_{L,m,m,0}(v,F)$ in Theorem \ref{PhiLmm0} is
$\sum_{k,C}A_{k,C,\rho}\frac{r^{k-C}}{(2\beta^{2})^{k}(\pi
y)^{k+C}}\mathbf{e}(r\tau)$. Hence
$\psi_{0}(\tau)=\sum_{r>0}A_{0,0,\frac{r}{2\beta^{2}}}q^{r}$, so that
$\frac{\psi_{0}'(\tau)}{2\pi i}=\sum_{r>0}rA_{0,0,\frac{r}{2\beta^{2}}}q^{r}$.
We substitute $A_{k,C,\rho}$ from Equation \eqref{AkCrho}, the value
$a_{0,0,+1}=2^{m}$, and the norm $\frac{s^{2}}{2\beta^{2}}$ of
$\frac{s}{2\beta^{2}} \in K^{*}$, and make the index change $d=\frac{r}{n}$.
This yields the asserted Fourier expansion for
$\frac{i^{m}}{2}\delta_{2m}\Phi_{L,m,m,0}(v,F)$ at the cusp, and completes the
proof of the theorem.
\end{proof}
Note that the Fourier coefficients in Theorem \ref{b-=1mer} are based, up to
rational numbers and the global power of $\beta$, on roots of unity of bounded
order and the Fourier coefficients of the weight $\frac{1}{2}-m$ weakly
holomorphic modular form $f$.

As an example for the map
$f\mapsto\frac{i^{m}}{2}\delta_{2m}\Phi_{L,m,m,0}(v,F)$, fix $N\in\mathbb{N}$,
and take $L$ to be the lattice spanned by $z=\binom{\ \ \ 1/\sqrt{N}}{}$,
$\zeta=\binom{}{\sqrt{N}\ \ \ }$, and the vector $\binom{\sqrt{N}\ \ \ \ \ \ }{\
\ \ \ -\sqrt{N}}$. In this case $\beta=\sqrt{N}$, $L^{*}/L=K^{*}/K$ is cyclic of
order $2N$, $\rho_{L}=\rho_{K}$, $(z,L)=\mathbb{Z}$ and $\zeta \in L$, and
$\Gamma^{+}=\Gamma_{0}(N)$ acts by conjugation. Hence Theorem \ref{b-=1mer}
yields a singular Shimura-type correspondence, as stated in the following

\begin{cor}
Let
$f(\tau)=\sum_{\gamma\in\mathbb{Z}/2N\mathbb{Z}}\sum_{n\in\mathbb{Q}}c_{\gamma,n
}q^{n}e_{\gamma}$ be a weakly holomorphic modular form of weight $\frac{1}{2}-m$
and representation $\rho_{K}$ for some $m\in\mathbb{N}$. The $q$-series
\[\sum_{r>0}\frac{r^{m}}{N^{m}}\bigg(\sum_{d|r}d^{m+1}c_{d,\frac{d^{2}}{4N}}
\bigg)q^{r}\] defines a meromorphic modular form of weight $2m+2$ with respect
to $\Gamma_{0}(N)$, which has poles of order $m+1$ at some points in
$\mathcal{H}$ which are quadratic over $\mathbb{Q}$. \label{Shimcor}
\end{cor}

In the case $N=1$ in Corollary \ref{Shimcor} the modular forms of
$\frac{1}{2}-m$ and representation $\rho_{L}$ correspond to weakly holomorphic
modular forms of the same weight with respect to $\Gamma_{0}(4)$ which lie in
the corresponding Kohnen plus-space (see Chapter 5 of \cite{[EZ]} or Proposition
1 of \cite{[K]}, which easily generalize to the weakly holomorphic case). They
exist only if $m$ is even (since $L^{*}/L$ has exponent 2), yielding the
statement given in the corresponding Theorem in the Introduction.

\smallskip

The subsequent work \cite{[Ze2]} investigates the behavior of
$\frac{i^{m}}{2}\Phi_{L,m,m,0}$ under weight changing operators for automorphic
forms on Grassmannians (defined in that reference) in higher dimensions. As
meromorphic images are constructed in that reference, it should be possible to
obtain a Gross--Kohnen--Zagier type theorem for higher-codimensional Heegner
cycles in universal families over higher-dimensional varieties along the lines
of the proof of the main result of this paper (see \cite{[Ze4]} for the case of
dimension 2). On the other hand, \cite{[Ze3]} investigates the modular forms
from Theorem \ref{b-=1mer} further, and relates them to the results of
\cite{[BK]} as well as to other modular objects.

\section{Relations between CM Cycles \label{Rels}}

Let $B$ and $i$ be as in Section \ref{CMUF}, and let $L$ be the lattice
$\tilde{\Lambda} \subseteq B_{0}$ from part $(ii)$ of Corollary \ref{H2lat}. The
arithmetic applications of the theory of theta lifts are based on
\begin{prop}
The following assertions hold:
\begin{enumerate}[$(i)$]
\item The map $\frac{i}{\sqrt{disc(I)}}$ identifies $L_{\mathbb{R}}$ with
$M_{2}(\mathbb{R})_{0}$ as quadratic spaces.
\item A point $\tau\in\mathcal{H}$ is a CM point for $B$ and $i$ if and only if it is of the form $\lambda^{\perp}$ for some $\lambda$ from $L^{*}$.
\end{enumerate}\label{lambdaperpCM}
\end{prop}

\begin{proof}
Part $(i)$ follows from part $(ii)$ of Theorem \ref{H2Aprop}. For part $(ii)$
use the description for $G(L_{\mathbb{R}})$ preceding Theorem \ref{b-=1mer} and
apply Proposition \ref{CMptchar}. This proves the proposition.
\end{proof}

The observation from Proposition \ref{lambdaperpCM} relates the special points
$\lambda^{\perp}$ on $\mathcal{H}$ as a Grassmannian to the CM points on modular
and Shimura curves. This shows that the automorphic forms of \cite{[B1]} have
divisors with arithmetic meaning, and as we shall see below, similar assertion
holds for the singularities of our automorphic forms.

\subsection{$m$-Divisors, Heegner $m$-Divisors, and their Classes}

We now define the objects for which we prove the main result. Let $L$ be an even
lattice of signature $(2,b_{-})$, and let $G(L_{\mathbb{R}})$ be its positive
Grassmannian.
\begin{defn}
An \emph{$m$-divisor} on $G(L_{\mathbb{R}})$ is a locally finite sum of elements
in $Sym^{m}L_{\mathbb{C}} \otimes Div(G(L_{\mathbb{R}}))$, namely of expressions
of the sort
\begin{equation}
\bigg(\prod_{i=1}^{m}\lambda_{i}\bigg) \otimes Y, \label{mdivexp}
\end{equation}
where $\lambda_{i} \in L_{\mathbb{C}}$ and $Y \subseteq G(L_{\mathbb{R}})$ is an
irreducible divisor. Recalling that $O(L_{\mathbb{R}})$ acts on both
$G(L_{\mathbb{R}})$ and $Sym^{m}L_{\mathbb{C}}$, we define an \emph{$m$-divisor
on the quotient $X(\Gamma)=\Gamma \backslash G(L_{\mathbb{R}})$} to be an
$m$-divisor on $G(L_{\mathbb{R}})$ which is invariant under the action of the
discrete subgroup $\Gamma$ of $O(L_{\mathbb{R}})$. An $m$-divisor on
$G(L_{\mathbb{R}})$ or on $X(\Gamma)$ is \emph{of totally negative type} if it
satisfies the following condition: For every term of the form appearing in
Equation \eqref{mdivexp}, each $\lambda_{i}$ lies in (the complexification of)
the space $v_{-}$ corresponding to every point $v \in Y$. This condition is
well-defined, and is equivalent to every term being of the form
$a_{\lambda}\lambda^{m}\otimes\lambda^{\perp}$ for some negative norm vector  $\lambda \in M_{\mathbb{R}}$, with $a_{\lambda}\in\mathbb{C}$. \label{mdiv}
\end{defn}

So $m$-divisors are divisors with coefficients in local systems. Lemma
\ref{abcdz1z2} and part $(ii)$ of Theorem \ref{H2Aprop} show that the local
system from Definition \ref{mdiv} becomes $Sym^{m}V_{2}$ when $b_{-}=1$.

\smallskip

We need to introduce another property of automorphic forms on
$G(L_{\mathbb{R}})$. Let $\Phi$ be an automorphic form of weight $m$ on
$G(L_{\mathbb{R}}) \cong K_{\mathbb{R}}+iC$ which is an eigenform of
$\Delta_{m}^{G}$ with eigenvalue $-2mb_{-}$, and assume first that the lattice
$L$ contains a primitive norm zero vector $z$. The expansion of $\Phi$ around
$z$ in a Weyl chamber $W$ containing $z$ in its closure is of the form
\[\widetilde{\varphi}(Y)+\sum_{k=0}^{m}\sum_{C=0}^{k}\sum_{\rho \in
K^{*}}\frac{A_{k,C,\rho}}{\pi^{k+C}}\frac{(\rho,Y)^{k-C}}{(Y^{2})^{k}}
\times\left\{\begin{array}{cc}\mathbf{e}\big((\rho,Z)\big) & (\rho,W)>0 \\
\mathbf{e}\big((\rho,\overline{Z})\big) & (\rho,W)<0\end{array}\right.\] (as in
Theorem \ref{PhiLmm0}). We call $\Phi$ \emph{algebraic} if the coefficients
$A_{k,C,\rho}$ are algebraic over $\mathbb{Q}$. If $L$ contains no norm zero
vectors, then as in Section 8 of \cite{[B1]}, we can embed $L$ in two different
lattices $M$ and $N$ of signature $(2,b_{-}+24)$, giving embeddings of
$G(L_{\mathbb{R}})$ into the larger Grassmannians $G(M_{\mathbb{R}})$ and
$G(N_{\mathbb{R}})$. Assume that we can present $\Phi$ as the difference of
weight $m$ automorphic forms on $G(M_{\mathbb{R}})$ and $G(N_{\mathbb{R}})$,
which are eigenforms with eigenvalue $-2m(b_{-}+24)$ of the corresponding
Laplacians, minus their singularities. We call $\Phi$ \emph{algebraic} if it is
obtained in this way from \emph{algebraic} automorphic forms on the larger
Grassmannians. We note again that one has to apply the argument carefully in
this case---see the last remark following Theorem \ref{PhiLmm0} above. The
$q$-expansion principle shows that in the case where $b_{-}=1$ and $\Gamma$ has
cusps, the $\delta_{2m}$-images of an algebraic automorphic form of automorphic
weight $m$ is an algebraic meromorphic modular form of $2m+2$. This assertion is
likely to extend also to the case where $\Gamma$ has no cusps.

\smallskip

Let $\Gamma$ be a discrete subgroup of $O^{+}(L_{\mathbb{R}})$ (we will take the
intersection of the discriminant kernel with $SO^{+}(L)$ in our applications).
In general we make the following

\begin{defn}
The $m$-divisor of totally negative type on $X(\Gamma)$, specified as
$\sum_{\lambda \in
L_{\mathbb{R}},\lambda^{2}<0}a_{\lambda}\lambda^{m}\otimes\lambda^{\perp}$ (a
locally finite, $\Gamma$-invariant sum) is called \emph{strongly principal} if
there exists an algebraic automorphic form of weight $m$ on $G(L_{\mathbb{R}})$
whose singularities lie on the divisors $\lambda^{\perp}$, and along each such
$\lambda^{\perp}$ the singularity is
\[\frac{a_{\lambda}}{2}\bigg(\frac{i}{2\pi}\bigg)^{m}\Bigg[\prod_{r=0}^{m-1}
\bigg(r+\frac{b_{-}}{2}\bigg)\cdot\frac{(\lambda,\overline{Z_{v,V}})^{m}}{2^{m}
(Y^{2})^{m}}\cdot\bigg(-\ln\frac{|(\lambda,Z_{v,V})|^{2}}{Y^{2}}\bigg)+\]
\[+\sum_{k=0}^{m-1}\frac{m!}{2^{k}k!}\bigg(\frac{\lambda^{2}}{2}\bigg)^{m-k}
\cdot\prod_{r=0}^{k-1}\bigg(r+\frac{b_{-}}{2}\bigg)\cdot\frac{(\lambda,\overline
{Z_{v,V}})^{k}}{(\lambda,Z_{v,V})^{m-k}(Y^{2})^{k}}\cdot\frac{1}{m-k}\Bigg].\]
\label{pringen}
\end{defn}

If $b_{-}=1$ we adopt the alternative

\begin{defn}
We call the $m$-divisor of totally negative type on $X(\Gamma)$ defined by
$\sum_{\lambda \in
L_{\mathbb{R}},\lambda^{2}<0}a_{\lambda}\lambda^{m}\otimes\lambda^{\perp}$
(again a locally finite, $\Gamma$-invariant sum) \emph{principal} if there
exists an algebraic modular form of weight $2m+2$ on
$G(L_{\mathbb{R}})\cong\mathcal{H}$, whose poles are at the points
$\lambda^{\perp}$, and at each point $\lambda^{\perp}=\sigma$ the singularity is
\[\frac{i}{(4\pi)^{m+1}}\sum_{\alpha}a_{\alpha
J_{\sigma}}\alpha^{m}\frac{m!(2it)^{m+1}}{(\tau-\sigma)^{
m+1}(\tau-\overline{\sigma})^{m+1}}.\] \label{prinmer}
\end{defn}

We remark that while the (principal) divisor of a rational function depends also
on the zeros of the function, our (strongly) principal $m$-divisors involve only
the singularities of the function. Indeed, when comparing the case $m=0$ here
with \cite{[B2]}, the theta lift $\Phi$ from \cite{[B1]} (considered in
Definition \ref{pringen}) is (roughly) the logarithm of the absolute value of
the Borcherds product $\Psi$, and for $b_{-}=1$ Definition \ref{prinmer}
considers the singularities of $\delta_{0}\Phi=\frac{1}{2\pi
i}\frac{\partial\Phi}{\partial\tau}\sim\frac{\Psi'}{\Psi}$, the logarithmic
derivative of $\Psi$ (see also the remarks at the end of this Section). The
algebraicity constraint corresponds to the fact that the Fourier coefficients of
negative indices of the weakly holomorphic modular form appearing in Theorem
13.3 of \cite{[B1]} are required to be integral.

A subtle point in Definitions \ref{pringen} and \ref{prinmer} is dealt with in
the following
\begin{lem}
If the two terms $a_{\lambda}\lambda^{m}\otimes\lambda^{\perp}$ and
$a_{\mu}\mu^{m}\otimes\mu^{\perp}$ coincide then so are the corresponding
singularities in Definitions \ref{pringen} and \ref{prinmer}. \label{welldef}
\end{lem}

\begin{proof}
The assumption of the lemma is equivalent to $\lambda=x\mu$ and
$a_{\mu}=x^{m}a_{\lambda}$ for some real $x\neq0$. As
$\big(\frac{\lambda^{2}}{2}\big)^{m-k}\cdot\frac{(\lambda,\overline
{Z_{v,V}})^{k}}{(\lambda,Z_{v,V})^{m-k}}$ equals
$x^{m}\big(\frac{\mu^{2}}{2}\big)^{m-k}\cdot\frac{(\mu,\overline{Z_{v,V}})^{k}}{
(\mu,Z_{v,V})^{m-k}}$ for every $0 \leq k \leq m$ in this case (and the
difference in the logarithm gives a smooth function of $Z$), the assertion holds
for Definition \ref{pringen}. For Definition \ref{prinmer}, with $b_{-}=1$, we
have $\lambda=\alpha J_{\sigma}$ for some $\alpha\in\mathbb{R}$ and
$\sigma\in\mathcal{H}$. Then $\mu=\frac{\alpha}{x}J_{\sigma}$, and the assertion
follows from the equality $\alpha^{m}=x^{m}\big(\frac{\alpha}{x}\big)^{m}$. This
proves the lemma.
\end{proof}

\smallskip

Given $\gamma \in M^{*}/M$ and $n\in\mathbb{Q}$, let
$y_{n,\gamma}^{(m)}=\sum_{\lambda \in
L+\gamma,\lambda^{2}=-2n}\lambda^{m}\otimes\lambda^{\perp}$. It is trivial if
$n\not\equiv\gamma^{2}/2(\mathrm{mod\ }\mathbb{Z})$, and also if $2\gamma=0$ in
$L^{*}/L$ and $m$ is odd. In particular, if $L^{*}/L$ has exponent 2 (or 1) then
$y_{n,\gamma}^{(m)}=0$ for all $n$ and $\gamma$ if $m$ is odd. The group
$Heeg^{(m)}(X(\Gamma))$ of \emph{Heegner $m$-divisors} is the free Abelian group
of $m$-divisors which is generated by the $m$-divisors $y_{n,\gamma}^{(m)}$. Let
$PrinHeeg^{(m)}_{st}(X(\Gamma))$ be the subgroup of $Heeg^{(m)}(X(\Gamma))$
consisting of those Heegner $m$-divisors which are strongly principal. If
$b_{-}=1$ then we denote the group of principal Heegner $m$-divisors by
$PrinHeeg^{(m)}(X(\Gamma))$. The spaces $Heeg^{(m)}_{st}Cl(X(\Gamma))$ and
$Heeg^{(m)}Cl(X(\Gamma))$ of \emph{(strong) Heegner $m$-divisor classes on
$X(\Gamma)$} are the corresponding quotient groups. We denote the image of
$y_{n,\gamma}^{(m)}$ in these quotients by $y_{n,\gamma}^{(m)}$ as well.

We can now state and prove the main result of this paper.

\begin{thm}
The formal power series $\sum_{n,\gamma}y_{n,\gamma}^{(m)}q^{n}e_{\gamma}$ is a
modular form of weight $1+\frac{b_{-}}{2}+m$ and representation $\rho_{L}^{*}$ with respect to $Mp_{2}(\mathbb{Z})$, with coefficients in the finite-dimensional space $Heeg^{(m)}_{st}Cl(X(\Gamma))\otimes\mathbb{Q}$. \label{main}
\end{thm}

\begin{proof}
The proof follows \cite{[B2]}. Consider the space of weight $k$ modular forms
with representation $\rho_{L}^{*}$ (as a finite-dimensional subspace of the
space of power series), and the space of singular parts of weakly holomorphic
modular forms of weight $2-k$ and representation $\rho_{L}$ (which is a subspace
of finite codimension of the space of singular parts). By Theorem 3.1 of
\cite{[B2]}, these two spaces are the full perpendicular spaces of one another
for any half-integral $k$. Let $f$ be a weakly holomorphic modular form of
weight $1-\frac{b_{-}}{2}-m$ and representation $\rho_{L}$ with respect to
$Mp_{2}(\mathbb{Z})$. Expand $f$ as in Equation \eqref{Fourier}. Then Theorem
\ref{PhiLmm0} shows that the singularities of the theta lift
$\frac{i^{m}}{2}\Phi_{L,m,m,0}(v,F)$ of $F=\delta_{1-\frac{b_{-}}{2}-m}^{m}f$
are along the divisors $\lambda^{\perp}$ with $\lambda \in L^{*}$, where
every $\lambda \in L^{*}$ contributes a singularity along $\lambda^{\perp}$ which takes the form appearing in Definition \ref{pringen} with $a_{\lambda}=c_{\lambda,\frac{\lambda^{2}}{2}}$. If we assume that $f$ has algebraic Fourier coefficients, then it follows from the description of the coefficients $A_{k,C,\rho}$ that the automorphic form $\frac{i^{m}}{2}\Phi_{L,m,m,0}(v,F)$ is algebraic. Assuming that the Fourier coefficients $c_{n,\gamma}$ with $n<0$ are integral, this implies that the $m$-divisor $\sum_{n>0,\gamma}c_{-n,\gamma}y_{n,\gamma}^{(m)}$ is strongly principal, hence it vanishes in $Heeg^{(m)}_{st}Cl(X(\Gamma))$. Let $\xi^{(m)}$ be the map which takes a ``singular part'' $\sum_{n<0,\gamma}a_{\gamma,n}q^{n}e_{\gamma}$ to the element
$\sum_{n>0,\gamma}a_{\gamma,-n}y_{n,\gamma}^{(m)}$ in the class group
$Heeg^{(m)}_{st}Cl(X(\Gamma))$. The map $\xi^{(m)}$ is surjective, and since the
space of weakly holomorphic modular forms of any weight and representation
$\rho_{L}$ has a basis consisting of modular forms with integral Fourier
coefficients (see \cite{[McG]}), $\xi^{(m)}$ factors through the quotient of the
space of singular parts by singular parts of weakly holomorphic modular forms.
By Lemmas 4.2, 4.3, and 4.4 of \cite{[B2]} this quotient space is
finite-dimensional. Hence $\dim Heeg^{(m)}_{st}Cl(X(\Gamma))<\infty$ as well.
Thus the Serre duality pairing extends to the case where the space of singular
parts and the image of the pairing are tensored with
$Heeg^{(m)}_{st}Cl(X(\Gamma))$, and continues to be non-degenerate.

Now, for an algebraic modular form $f$ as above, the vanishing expression
$\sum_{n>0,\gamma}c_{-n,\gamma}y_{n,\gamma}^{(m)}$ is obtained as the Serre
duality pairing of the formal power series
$\sum_{n,\gamma}y_{n,\gamma}^{(m)}q^{n}e_{\gamma}$ with the singular part
of $f$. Hence this power series is perpendicular to all the singular parts of
weakly holomorphic meromorphic modular form of weight $1-\frac{b_{-}}{2}-m$ and
representation $\rho_{L}$ with respect to $Mp_{2}(\mathbb{Z})$. By Lemma 4.3 of
\cite{[B2]} and the result of \cite{[McG]}, the formal power series in question
lies in the space of (holomorphic) modular forms of weight $1+\frac{b_{-}}{2}+m$ and representation $\rho_{L}^{*}$ tensored with
$Heeg^{(m)}_{st}Cl(X_{\Gamma})\otimes{\mathbb{C}}$. The fact that the Serre
duality pairing and the spaces we consider are defined over $\mathbb{Q}$
implies the ``algebraicity'' of
$\sum_{n,\gamma}y_{n,\gamma}^{(m)}q^{n}e_{\gamma}$ as a modular
form with coefficients in $Heeg^{(m)}Cl(X_{\Gamma})\otimes\mathbb{Q}$. This
proves the theorem.
\end{proof}

In the case $b_{-}=1$ we obtain an assertion like that of Theorem \ref{main} but
with coefficients in $Heeg^{(m)}Cl(X(\Gamma))\otimes\mathbb{Q}$ (no subscript
$st$). This follows either by omitting this subscript and replacing Definition
\ref{pringen} by Definition \ref{prinmer} in the proof of Theorem \ref{main},
or using the fact that if every strongly principal Heegner $m$-divisor is
principal. Indeed, if $\Phi$ is an automorphic form which shows that some
$m$-divisor is strongly principal, then $\delta_{2m}\Phi$ reveals the
principality of this $m$-divisor. Whether or not the notions of principality and
strong principality are equivalent in the case $b_{-}=1$ remains an interesting
question. We also remark that unlike the result of \cite{[B2]}, here the
``constant term'' $y_{0,0}$ does not show up. Thus our expression is in fact a
``cusp form''.

\medskip

Let us give an example of an automorphic form yielding the principality of a
Heegner $m$-divisor. This also illustrates Corollary \ref{Shimcor} above. Take
$N=1$ in the example preceding Corollary \ref{Shimcor}, and let $m$ be even.
Theorem \ref{main} implies that the generating series of the Heegner
$m$-divisors is a cusp form of weight $\frac{3}{2}+m$ and representation
$\rho_{L}^{*}$. Equivalently, it lies in the Kohnen plus-space of cusp forms of
weight $\frac{3}{2}+m$ on $\Gamma_{0}(4)$ (see \cite{[EZ]} or \cite{[K]}). This
space is isomorphic to the space of cusp forms of weight $2m+2$ on
$SL_{2}(\mathbb{Z})$ by Theorem 1 of \cite{[K]}. Hence this series may not
vanish only if $m\geq8$. We take $m=2$, and present the function yielding the
principality $y_{\frac{3}{4},1}^{(2)}$. Define
\[\theta(\tau)=\sum_{n\in\mathbb{Z}}q^{n^{2}}\qquad\mathrm{and}\qquad
H_{\frac{5}{2}}(\tau)=120\sum_{n=0}^{\infty}H(2,n)q^{n}\] as in \cite{[K]} (here $H(2,n)$ is the function introduced in \cite{[Co]}, and $H_{\frac{5}{2}}$ is normalized to attain 1 at $\infty$ and to have integer coefficients). Let $E_{l}$ be the classical (normalized) Eisenstein series of even weight $l\geq4$ on $SL_{2}(\mathbb{Z})$, and let $\Delta$ be the classical cusp form of weight 12. The weight $-\frac{3}{2}$ weakly holomorphic modular form $f$ with
representation $\rho_{L}$ having the appropriate principal part corresponds to
the weakly holomorphic modular form
\[g(\tau)=\frac{E_{10}(4\tau)\theta(\tau)-E_{8}(4\tau)H_{\frac{5}{2}}(\tau)}{
\Delta(4\tau)}=\]
\[=q^{-3}-56+384q-15024q^{4}+39933q^{5}-523584q^{8}+1129856q^{9}+O(q^{12})\] in
the Kohnen plus-space of weight $-\frac{3}{2}$. The lift
$\frac{i^{2}}{2}\delta_{4}\Phi_{L,2,2,0}(v,F)$ of $F=\delta_{-\frac{3}{2}}^{2}f$
is given by the Fourier expansion
\[384q-479232q^{2}+274558464q^{3}-118219210752q^{4}+43867326009600q^{5}+O(q^{6}
)\] around $\infty$. It is a modular form of weight 6 with respect to
$SL_{2}(\mathbb{Z})$, which has a pole of the form
$\frac{18\sqrt{3}/(4\pi)^{3}}{(\tau-\sigma)^{3}(\tau-\overline{\sigma})^{3}}$ at
$\sigma$ the 3rd root of unity in $\mathcal{H}$ (recall that $\beta$ and the
Fourier coefficient are 1 and $\alpha=2t=\sqrt{3}$). These properties determine
$\frac{i^{2}}{2}\delta_{4}\Phi_{L,2,2,0}(v,F)$ as
$\frac{384E_{6}\Delta}{E_{4}^{3}}$, a fact which can also be verified directly
by an explicit evaluation of the Fourier coefficients.

\subsection{Algebaricity and Relations to Other Works}

We now explain why the condition of algebraicity of the automorphic form is
required in the definition of principal $m$-divisors (both in Definition
\ref{pringen} and in Definition \ref{prinmer}). The eigenvalue of the function
$\frac{i^{m}}{2}\Phi_{L,m,m,0}(v,F)$ from Theorem \ref{PhiLmm0} under the action
of $\Delta_{m}^{G}$ is a consequence of the fact that
$F=\delta_{1-\frac{b_{-}}{2}-m}^{m}f$ has a specific eigenvalue under the weight
$1-\frac{b_{-}}{2}+m$ Laplacian on $\mathcal{H}$. Replacing the weakly
holomorphic modular form $f$ by a harmonic weak Maa\ss\ form would yield
$F=\delta_{1-\frac{b_{-}}{2}-m}^{m}f$ with the same eigenvalue. Hence
$\frac{i^{m}}{2}\Phi_{L,m,m,0}(v,F)$ would have the same properties. For further
details on harmonic weak Maa\ss\ forms we refer the reader to Section 3 of
\cite{[BF]}.

Now, every harmonic weak Maa\ss\ form $f$ of weight $l$ and representation
$\rho_{L}$ with respect to $Mp_{2}(\mathbb{Z})$ decomposes as the sum of a
holomorphic part (which yields a similar expression in the theta lift) and a
non-holomorphic part. The non-holomorphic part consists of ``constant terms''
multiplying a power of $y$ and expressions which are based on incomplete gamma
functions. Hence the part of $F=\delta_{l}^{m}f$ arising from the
non-holomorphic part of $f$ is again a power of $y$, plus the complex conjugate
of an almost holomorphic function, plus similar incomplete gamma functions. For
$l<1-m$ the map $\delta_{l}^{m}f$ is invertible (its inverse is a constant
multiple of the $m$th power of the weight lowering operator $L$). Thus every
modular form $F$ of weight $l+2m$ and representation $\rho_{L}$ with respect to
$Mp_{2}(\mathbb{Z})$ having eigenvalue $m(l+m-1)$ is $\delta_{l}^{m}f$ for some
harmonic weak Maa\ss\ form $f$. Take $l=1-\frac{b_{-}}{2}-m$, and consider first
the case where the operator $\xi_{l}$ of $\cite{[BF]}$ takes $f$ to a weakly
holomorphic modular form which is not holomorphic at the cusp. An argument
similar to Section 6 of \cite{[B1]} (or Theorem \ref{sing}) shows that the
corresponding theta lift $\frac{i^{m}}{2}\Phi_{L,m,m,0}(v,F)$ from Theorem
\ref{PhiLmm0} has, apart from the usual singularities, also singularities along
the sub-Grassmannians of the form $\{v \in G(L_{\mathbb{R}})|\lambda \in
v_{+}\}$ for $\lambda \in L^{*}$ with positive norm (these are $\lambda^{\perp}$
on $G(L_{\mathbb{R}}(-1)) \cong G(L_{\mathbb{R}})$). In the case $b_{-}=1$ these
lifts may be related to (non-harmonic) locally Maa\ss\ forms as defined in
\cite{[BKK]}, and their $\delta_{2m}$-images are probably examples of the
locally harmonic Maa\ss\ forms considered in that reference. In any case, lifts
of such forms do not yield additional principal Heegner $m$-divisors.

Consider now the case where $\xi_{1-\frac{b_{-}}{2}-m}f$ is holomorphic at the
cusp. The singularity then arises only from the holomorphic part of $f$. Now, if
$L$ contains a norm zero vector $z$ then the Fourier expansion of
$\frac{i^{m}}{2}\Phi_{L,m,m,0}(v,F)$ contains, apart from the expression from
Theorem \ref{PhiLmm0}, terms involving Bessel $K$-functions evaluated at $2\pi
n|Y|\cdot|\eta_{w_{-}}|=2\pi n\sqrt{(\eta,Y)^{2}-Y^{2}\eta^{2}}>0$ (the strict
positivity follows from the condition on $\xi_{1-\frac{b_{-}}{2}-m}f$, since
only elements with $\eta^{2}\leq0$ contribute in the non-holomorphic part of
$f$). Hence one may distinguish, for $b_{-}>1$, the theta lifts arising from
weakly holomorphic modular forms from those in which
$\xi_{1-\frac{b_{-}}{2}-m}f\neq0$ by the fact that the former do not include
terms of the form
$\mathbf{e}\big((\rho,X)+i\sqrt{(\rho,Y)^{2}-Y^{2}\rho^{2}})\big)$. Note that
for $b_{-}=2$, with $(\rho,Z)=\rho_{1}\sigma+\rho_{2}\tau$, the latter
expression is either $\mathbf{e}(\rho_{1}\sigma+\rho_{2}\overline{\tau})$ or
$\mathbf{e}(\rho_{1}\overline{\sigma}+\rho_{2}\tau)$. These expressions may thus
be related to the class map defined in Section 5 of \cite{[Bru]} in terms of the
operator $\partial\overline{\partial}$. On the other hand, in case $b_{-}=1$ the
lattice $K$ is positive definite. Hence only the holomorphic part of $f$
contributes to the theta lift. As any principal part at $\infty$ can be obtained
as the principal part of a harmonic weak Maa\ss\ form of every weight (see
Proposition 3.11 of \cite{[BF]}), omitting the algebraicity condition in
Definitions \ref{pringen} and \ref{prinmer} would reduce the the (strong)
Heegner $m$-divisor class group (for $b_{-}=1$) to 0. Now, (algebraic) weakly
holomorphic modular forms give rise to algebraic theta lifts, while for harmonic
weak Maa\ss\ forms we pose the following

\begin{conj}
Let $f$ be a harmonic weak Maa\ss\ form of weight $\frac{1}{2}-m$ (with $m>0$)
and representation $\rho_{L}$, and assume that $\xi_{\frac{1}{2}-m}f$ is
holomorphic at $\infty$. Assume further that the principal part of $f$ involves
only Fourier coefficients which are algebraic over $\mathbb{Q}$, or even
integral. Then the condition $\xi_{\frac{1}{2}-m}f\neq0$ implies the existence of some Fourier coefficient of the holomorphic part of $f$ which is transcendental over $\mathbb{Q}$. \label{hwmftrans}
\end{conj}

Although we do not investigate the Fourier coefficients of harmonic weak Maa\ss\
forms in this paper, let us indicate what evidence towards Conjecture
\ref{hwmftrans} does exist. First, we mention Corollary 1.4 of \cite{[BO]}, as
well as the conjecture preceding it in that reference. Indeed, in the weight
$\frac{1}{2}$ case considered there one has certain theta series of weight
$\frac{3}{2}$ which map by the Shimura correspondence to Eisenstein series of
weight 2. Harmonic weak Maa\ss\ forms which map to these forms under the
operator $\xi_{\frac{1}{2}}$ are known to have algebraic Fourier coefficients.
Since in higher half-integral weights the Shimura correspondence takes cusp
forms to cusp forms, we expect in Conjecture \ref{hwmftrans} only the weakly
holomorphic modular forms to have algebraic coefficients. In any case, it will
be interesting to find, under Conjecture \ref{hwmftrans}, whether every relation
between the Heegner $m$-divisors in $Heeg^{(m)}_{st}Cl(X(\Gamma))$ (or in
$Heeg^{(m)}Cl(X(\Gamma))$ for $b_{-}=1$) arises from a weakly holomorphic
modular form as we described.

\medskip

Theorems 0.2.1 and 0.3.1 of \cite{[Zh]} (with $k=m+1$) relate the images of
$m$-dimensional CM cycles in the Kuga--Sato variety $W_{2m}$ over modular curves in a certain vector space to modular forms of weight $2m+2$. Such a modular form may be related to the modular form of weight $\frac{3}{2}+m$ obtained from our Theorem \ref{main} by the Shimura--Shintani correspondence. This may also indicate, once a relation between our construction and the one in \cite{[Zh]} is established, that Conjecture \ref{hwmftrans} may be true, since otherwise the power series from Theorem \ref{main} must lie in a strict quotient Hecke module. Furthermore, \cite{[H]} defines, for any $m$, a map on the Heegner divisors on $X_{0}(N)$ into some elliptic curve. This reference conjectures that the images of these divisors under this map correspond to the coefficients of a modular form of weight $2m+2$. In some cases this map coincides with the Abel-Jacobi map on the CM cycles into a certain sub-torus of the intermediate Jacobian of the Kuga--Sato variety $W_{2m}$. \cite{[H]} supplies numerical evidence that this is true in some other cases as well.

Now, the results of Section \ref{CMUF} relate the Heegner $m$-divisors in the
case $b_{-}=1$ to symmetric powers of normalized cohomology classes of Heegner
cycles inside Kuga--Sato type varieties. In the case $m=0$ considered in
\cite{[B2]}, the automorphic form $\delta_{0}\frac{1}{2}\Phi_{L,0,0,0}$ is
roughly the logarithmic derivative of the Borcherds product $\Psi$. Hence it is
(under some normalization, when $c_{0,0}=0$) a meromorphic modular form of
weight 2 with only simple poles in CM points, and the residues in these poles
are integral. It thus corresponds to a differential of the first kind in the
description of \cite{[Sch]} or Section 3 of \cite{[BO]}. Its algebraicity thus
corresponds to the fact that its residue divisor vanishes in the Jacobian of the
curve $X(\Gamma)$ (being the divisor of the rational function $\Psi$)---see
Theorem 1 of \cite{[Sch]} or Theorem 3.2 of \cite{[BO]}. Returning to the case
of general $m$, we attach the element $M_{\tau}^{m}$, or $\binom{\tau}{1}^{2m}$,
of the local system $Sym^{m}L_{\mathbb{C}}$, to the meromorphic modular form
$\delta_{2m}\frac{i^{m}}{2}\Phi_{L,m,m,0}$. We thus obtain a meromorphic modular
form of weight 2 and representation $V_{2m}$. This corresponds to a meromorphic
differential form of degree $2m+1$ on (the Shimura curve analog of) $W_{2m}$.
Pairing this modular form with $\lambda^{m}\otimes\lambda^{\perp}$ for
$\lambda=\alpha J_{\sigma}$ gives the function
$\delta_{2m}\frac{(i\alpha)^{m}}{2t^{m}}\Phi_{L,m,m,0}(\tau)(\tau-\sigma)^{m}
(\tau-\overline{\sigma})^{m}$ (up to normalization), whose pole at $\sigma$ is
simple. Moreover, if we assume that $f$ has integral Fourier coefficients, then
the residues are integral. In addition, it seems that the corresponding
differential of the third kind is canonical in the following sense: The
cohomology group $H^{2m+1}(W_{2m})$ has a component $H^{1}(X,V_{2m})$ of Hodge
weight $(2m+1,0)$ and $(0,2m+1)$ whose holomorphic part consists of
$g\binom{\tau}{1}^{2m}d\tau$ for $g$ a cusp form of weight $2m+2$ with respect
to $\Gamma$. This is shown explicitly for $m=1$ in Section 6 of \cite{[Be]}, and it is not hard to generalize to arbitrary $m$. The $H^{1}(X,V_{2m})$ part of the cohomology class which is Poincar\'{e} dual to a $(2m+1)$-dimensional cycle $\mathcal{Z}$ in $W_{2m}$ equals $g(\tau)\binom{\tau}{1}^{2m}d\tau+\overline{g(\tau)}\binom{\overline{\tau}}{1}^{
2m}d\overline{\tau}$ for some cusp form $g$. We should therefore have the
equality
\[\int_{\mathcal{Z}}\delta_{2m}\frac{i^{m}}{2}\Phi_{L,m,m,0}\binom{\tau}{1}^{2m}
d\tau=\int_{X(\Gamma)}(2iy)^{2m}\delta_{2m}\frac{i^{m}}{2}\Phi_{L,m,m,0}
\overline{g(\tau)}d\tau d\overline{\tau}\] (up to the poles). As the latter form
is exact like in Lemma \ref{intRkL}, the integral reduces to the contribution
from the poles, which also vanishes since $\overline{g}$ is smooth and the
integral over a circle of radius $\varepsilon$ around each pole $\sigma$ gives
an expression which vanishes as $\varepsilon\to0$. This evaluation process is
well-defined up to the location of the poles of
$\delta_{2m}\frac{i^{m}}{2}\Phi_{L,m,m,0}$ with respect to the choice of the
cycle. However, jumping over a pole $\sigma$ changes the value by a totally
imaginary multiple of the residue of
$\delta_{2m}\frac{i^{m}}{2}\Phi_{L,m,m,0}(\tau-\sigma)^{m}(\tau-\overline{\sigma
})^{m}$, which is assumed to be integral. It may be interesting to compare this
evaluation with the regularized integrals from \cite{[BK]}. In any case, if we
could establish some generalization of Theorem 1 of \cite{[Sch]} assuring us
that the algebraicity of this differential form with local coefficients implies
the vanishing of a certain expression inside some generalized Jacobian (up to
torsion), then we would know that the images of our Heegner $m$-divisors in
$Heeg^{(m)}Cl(X(\Gamma))$ are the same as their images in this Jacobian.

The map $\alpha$ of \cite{[H]} also seems to be related to the
$(H^{2m+1,0})^{*}$ part of the intermediate Jacobian of the Kuga--Sato variety
$W_{2m}$. For any CM point $\sigma$, consider the $m$th symmetric power
$z_{\sigma}^{m}$ of the normalized CM cycle corresponding to $\sigma$ in
$A_{\sigma}$. In addition, consider the closure of the $(2m+1)$-dimensional
cycle $[\sigma,\infty) \times z_{\sigma}^{m}$ in the notation of Section 8 of
\cite{[Be]} (in the modular case this is possible) plus some cycle bounding the
fiber over $\infty$. This resembles, in the case $m=1$, the cycle considered in
\cite{[Sc]}, where the latter part is the counterpart of the combination of the
cycles $\Delta_{k}$ from \cite{[Be]}. Then the integral of some form
$g(\tau)(dz_{1} \wedge dz_{2})^{m}d\tau$ of type $(2m+1,0)$ over this cycle
decomposes as in the proof of Theorem 8.5 of \cite{[Be]}. The integral from
$\sigma$ to $\infty$ is of the function
$g(\tau)(\tau-\sigma)^{m}(\tau-\overline{\sigma})^{m}$ (up to some constant).
Since $\sigma$ is a (modular) CM point on $\mathcal{H}$, this gives (under
the correct normalization) the integral considered in \cite{[H]}. The integral
over the cycle in the fiber over $\infty$ probably gives some period of $g$, and
when we focus on one newform $g$ (of weight $2m+2$) this reduces to the image
modulo the lattice in $\mathbb{C}$ appearing in \cite{[H]}. It will be very
interesting to investigate what relations can be established between the
existence of the form $\delta_{2m}\frac{i^{m}}{2}\Phi_{L,m,m,0}$ and the values
of these integrals.

\noindent\textsc{Einstein Institute of Mathematics, the Hebrew University of Jerusalem, Edmund Safra Campus, Jerusalem 91904, Israel}

\noindent E-mail address: zemels@math.huji.ac.il

\end{document}